\documentclass[10pt]{amsart}
\usepackage{amssymb, amsmath, amsthm, amscd, mathtools}
\usepackage{graphicx}
\usepackage{stmaryrd}
\usepackage[hidelinks]{hyperref}
\usepackage[all, cmtip]{xy}
\usepackage{fullpage}
\usepackage{booktabs}
\usepackage{epsfig}
\usepackage{appendix}
\usepackage{listings}

\usepackage{latexsym}
\usepackage{expl3}
\usepackage{tikz,ifthen}
\usetikzlibrary{positioning, math, decorations.markings, decorations.pathreplacing, arrows.meta}
\usetikzlibrary{calc,shapes,cd}
\usetikzlibrary{knots}
\usepackage{xcolor}
\usepgfmodule{decorations}

\ExplSyntaxOn
\cs_new_eq:NN \ifstreqF \str_if_eq:nnF
\cs_new_eq:NN \ifstreqTF \str_if_eq:nnTF
\ExplSyntaxOff

\tikzset{edge/.style={line width=0.5}}

\tikzset{-o-/.code 2 args={\ifstreqF{#2}{} 
    {\ifstreqTF{#2}{>}
        {\pgfkeysalso{decoration={markings,mark=at position #1 with {\arrow[scale=0.8]{#2}}}
                    ,postaction={decorate}}
        }
       {\ifstreqTF{#2}{<}
           {\pgfkeysalso{decoration={markings,mark=at position #1 with {\arrow[scale=0.8]{#2}}}
                        ,postaction={decorate}}
            }
        }
    }
        {}
}}

\newcommand{\cross}[3] 
    {\raisebox{-0.3\height}
        {\begin{tikzpicture}
        \tikzmath{\xw=0.7;}
        \fill[white] (0,0)rectangle(\xw,1);
        \tikzmath{\ya=0.6; \mathbf{y}=0.2;}
        \tikzmath{\xcent=\xw*4/7;\ycent=\ya/2+\mathbf{y}/2;}
        \ifthenelse{\equal{#1}{p}}
            {\draw[edge,  -o-={0.9}{#2}] (0,\ya) -- (\xw,\mathbf{y});
            \draw[edge] (0,\mathbf{y}) -- (\xcent-0.12,\ycent-0.05);
            \draw[edge, -o-={0.7}{#3}] (\xcent+0.04,\ycent+0.04) -- (\xw,\ya);}   
            {\draw[edge, -o-={0.87}{#3}] (0,\mathbf{y}) -- (\xw,\ya);
            \draw[edge] (0,\ya) -- (\xcent-0.12,\ycent+0.05);
            \draw[edge, -o-={0.7}{#2}] (\xcent+0.04,\ycent -0.04) -- (\xw,\mathbf{y});}
        \end{tikzpicture}}   
    }

\newcommand{\kink}
    {\raisebox{-0.5\height}
        {\begin{tikzpicture}[pics/arrow/.style={code={\draw[edge, -{>[scale=0.8]}] (-0.5ex,0) -- (0.5ex,0);}}, scale=0.5]
        \fill[white] (0,-0.6) rectangle (2,1); 
            \begin{knot}[clip width=5, clip radius=5pt, consider self intersections, end tolerance=3pt, background color=white, fill opacity=0.5]
            \strand[edge] (2,0) to[out=left, in=right] pic[pos=0.5,sloped]{arrow}       
            (1.4,0) to[out=left, in=down] 
            (0.6,0.4)
            to[out=up, in=left] (1,0.8)
            to[out=right, in=up]  (1.4,0.4)
            to[out=down, in=right]  (0.6,0)
            to[out=left, in=right]  (0,0);
            \end{knot}
        \end{tikzpicture}}
    }

\newcommand{\squarepic}[1]
    {\raisebox{-0.4\height}
        {\begin{tikzpicture}[scale=1.2]
        \tikzmath{\rad=0.4;}
        \fill[white] (-\rad,-\rad) rectangle (\rad,\rad);#1
        \end{tikzpicture}}
    }

\newcommand{\horizontaledge}[1]
    {\squarepic
        {\draw[edge, -o-={0.55}{#1}] (-\rad,0) --(\rad,0);}
    }

\newcommand{\circlediag}[1]
    {\raisebox{-0.4\height}
        {\begin{tikzpicture}[scale=0.5]
        \tikzmath{\yo=-0.3; \xw=2; \yw=1.1;}
        \fill[white] (0,\yo) rectangle (\xw,\yw);
        \draw[edge, -o-={0.1}{#1}] (\xw/2,{(\yo + \yw)/2})  circle (0.4);
        \end{tikzpicture}}
    }

\newcommand{\sourcesinks}
    {\raisebox{-.30in}
        {\begin{tikzpicture}[scale=0.7]
		\tikzset{->-/.style=
		{decoration={markings,mark=at position ##1 with
		{\arrow{latex}}},postaction={decorate}}}
		\filldraw[draw=white,fill=white] (-1,-1) rectangle (1.2,1);
		\draw [decoration={markings, mark=at position 0.5 with {\arrow{>}}},postaction={decorate}](-1,0.8)--(0,0);
		\draw [decoration={markings, mark=at position 0.5 with {\arrow{>}}},postaction={decorate}](-1,-0.4)--(0,0);
		\draw [decoration={markings, mark=at position 0.5 with {\arrow{>}}},postaction={decorate}](-1,-0.8)--(0,0);
		\draw [decoration={markings, mark=at position 0.45 with {\arrow{<}}},postaction={decorate}](1.2,0.8)  --(0.2,0);
		\draw [decoration={markings, mark=at position 0.45 with {\arrow{<}}},postaction={decorate}](1.2,-0.4)  --(0.2,0);
		\draw [decoration={markings, mark=at position 0.45 with {\arrow{<}}},postaction={decorate}](1.2,-0.8)--(0.2,0);
		\node[scale=0.6]  at(-0.55,0.25) {$\vdots$};
		\node[scale=0.6]  at(0.78,0.23) {$\vdots$};
        \end{tikzpicture}}
    }

\newcommand{\coupon}
    {\raisebox{-.30in}{
	   \begin{tikzpicture}[scale=0.7]
		\tikzset{->-/.style=
			{decoration={markings,mark=at position ##1 with
			{\arrow{latex}}},postaction={decorate}}}
		\filldraw[draw=white, fill=white] (-1,-1) rectangle (1.2,1);
		\draw [decoration={markings, mark=at position 0.5 with {\arrow{>}}},postaction={decorate}](-1,0.8)--(0,0);
		\draw [decoration={markings, mark=at position 0.5 with {\arrow{>}}},postaction={decorate}](-1,-0.4)--(0,0);
		\draw [decoration={markings, mark=at position 0.5 with {\arrow{>}}},postaction={decorate}](-1,-0.8)--(0,0);
		\draw [decoration={markings, mark=at position 0.45 with {\arrow{<}}},postaction={decorate}](1.2,0.8)  --(0.2,0);
		\draw [decoration={markings, mark=at position 0.45 with {\arrow{<}}},postaction={decorate}](1.2,-0.4)  --(0.2,0);
		\draw [decoration={markings, mark=at position 0.45 with {\arrow{<}}},postaction={decorate}](1.2,-0.8)--(0.2,0);
		\node[scale=0.6]  at(-0.55,0.25) {$\vdots$};
		\node[scale=0.6]  at(0.78,0.23) {$\vdots$};
		\filldraw[draw=black,fill=white]  (0.1,0) ellipse (0.4 and 0.7);
		\node[scale=0.8]  at(0.1,0){$\sigma_{+}$};
        \end{tikzpicture}}
    }

\newcommand{\twogon}
    {\raisebox{-.30in}{
	   \begin{tikzpicture}[scale=0.7]
		\tikzset{->-/.style=
			{decoration={markings,mark=at position ##1 with
			{\arrow{latex}}},postaction={decorate}}}
		\filldraw[draw=white, fill=white] (-1,-1) rectangle (1.2,1);
		\draw [decoration={markings, mark=at position 0.5 with {\arrow{>}}},postaction={decorate}](-1,0.2)--(-0.3,0.2);
		\draw [decoration={markings, mark=at position 0.45 with {\arrow{<}}},postaction={decorate}](1.2,0.2)--(0.5,0.2);
        
        \path [draw=black, -o-={0.5}{<}] (-0.3,0.2) to[bend left=60] (0.5,0.2);
        \path [draw=black, -o-={0.5}{<}] (-0.3,0.2) to[bend right=60] (0.5,0.2);
        \end{tikzpicture}}
    }

\newcommand{\fourgon}
    {\raisebox{-.30in}{
	   \begin{tikzpicture}[scale=0.5]
		\tikzset{->-/.style=
			{decoration={markings,mark=at position ##1 with
			{\arrow{latex}}},postaction={decorate}}}
		\filldraw[draw=white, fill=white] (-1.2,-1.2) rectangle (1.2,1.2);
		\draw [decoration={markings, mark=at position 0.6 with {\arrow{>}}},postaction={decorate}](-1.2,0.8)--(-0.4,0.4);
        \draw [decoration={markings, mark=at position 0.5 with {\arrow{<}}},postaction={decorate}](-1.2,-0.8)--(-0.4,-0.4);
        \draw [decoration={markings, mark=at position 0.6 with {\arrow{>}}},postaction={decorate}](0.4,0.4)--(1.2,0.8);
        \draw [decoration={markings, mark=at position 0.5 with {\arrow{<}}},postaction={decorate}](0.4,-0.4)--(1.2,-0.8);
        \draw [decoration={markings, mark=at position 0.5 with {\arrow{<}}},postaction={decorate}](-0.4,0.4)--(0.4,0.4);
        \draw [decoration={markings, mark=at position 0.6 with {\arrow{>}}},postaction={decorate}](-0.4,-0.4)--(0.4,-0.4);
        \draw [decoration={markings, mark=at position 0.6 with {\arrow{>}}},postaction={decorate}](-0.4,-0.4)--(-0.4,0.4);
        \draw [decoration={markings, mark=at position 0.5 with {\arrow{<}}},postaction={decorate}](0.4,-0.4)--(0.4,0.4);
        \end{tikzpicture}}
    }
\newcommand{\verresol}
    {\raisebox{-.30in}{
	   \begin{tikzpicture}[scale=0.5]
		\tikzset{->-/.style=
			{decoration={markings,mark=at position ##1 with
			{\arrow{latex}}},postaction={decorate}}}
		\filldraw[draw=white, fill=white] (-1.2,-1.2) rectangle (1.2,1.2);
		\path [draw=black, -o-={0.5}{>}] (-1.2,0.8) to[bend left=60] (-1.2,-0.8);
        \path [draw=black, -o-={0.5}{<}] (1.2,0.8) to[bend right=60] (1.2,-0.8);
        \end{tikzpicture}}
    }

\newcommand{\verresolAL}
    {\raisebox{-.30in}{
	   \begin{tikzpicture}[scale=0.5]
		\tikzset{->-/.style=
			{decoration={markings,mark=at position ##1 with
			{\arrow{latex}}},postaction={decorate}}}
		\filldraw[draw=white, fill=white] (-1.2,-1.2) rectangle (1.2,1.2);
		\path [draw=black] (-1.2,0.8) to[bend left=60] (-1.2,-0.8);
        \path [draw=black] (1.2,0.8) to[bend right=60] (1.2,-0.8);
        \end{tikzpicture}}
    }

\newcommand{\horresol}
    {\raisebox{-.30in}{
	   \begin{tikzpicture}[scale=0.5]
		\tikzset{->-/.style=
			{decoration={markings,mark=at position ##1 with
			{\arrow{latex}}},postaction={decorate}}}
		\filldraw[draw=white, fill=white] (-1.2,-1.2) rectangle (1.2,1.2);
		\path [draw=black, -o-={0.5}{>}] (-1.2,0.8) to[bend right=30] (1.2,0.8);
        \path [draw=black, -o-={0.5}{<}] (-1.2,-0.8) to[bend left=30] (1.2,-0.8);
        \end{tikzpicture}}
    }

\newcommand{\horresolAL}
    {\raisebox{-.30in}{
	   \begin{tikzpicture}[scale=0.5]
		\tikzset{->-/.style=
			{decoration={markings,mark=at position ##1 with
			{\arrow{latex}}},postaction={decorate}}}
		\filldraw[draw=white, fill=white] (-1.2,-1.2) rectangle (1.2,1.2);
		\path [draw=black] (-1.2,0.8) to[bend right=30] (1.2,0.8);
        \path [draw=black] (-1.2,-0.8) to[bend left=30] (1.2,-0.8);
        \end{tikzpicture}}
    }

\newcommand{\largecross}
    {\raisebox{-.30in}{
	   \begin{tikzpicture}[scale=0.5]
		\tikzset{->-/.style=
			{decoration={markings,mark=at position ##1 with
			{\arrow{latex}}},postaction={decorate}}}
		\filldraw[draw=white, fill=white] (-1.2,-1.2) rectangle (1.2,1.2);
        \draw [decoration={markings, mark=at position 0.8 with {\arrow{>}}},postaction={decorate}](-1.2,0.8)--(1.2,-0.8);
        \draw [decoration={markings, mark=at position 0.8 with {\arrow{>}}},postaction={decorate}](-1.2,-0.8)--(1.2, 0.8);
        \end{tikzpicture}}
    }

\newcommand{\largecrossAL}
    {\raisebox{-.30in}{
	   \begin{tikzpicture}[scale=0.5]
		\tikzset{->-/.style=
			{decoration={markings,mark=at position ##1 with
			{\arrow{latex}}},postaction={decorate}}}
		\filldraw[draw=white, fill=white] (-1.2,-1.2) rectangle (1.2,1.2);
        \draw [decoration={markings},postaction={decorate}](-1.2,0.8)--(1.2,-0.8);
        \draw [decoration={markings},postaction={decorate}](-1.2,-0.8)--(1.2, 0.8);
        \end{tikzpicture}}
    }

\newcommand{\horresolii}
    {\raisebox{-.30in}{
	   \begin{tikzpicture}[scale=0.5]
		\tikzset{->-/.style=
			{decoration={markings,mark=at position ##1 with
			{\arrow{latex}}},postaction={decorate}}}
		\filldraw[draw=white, fill=white] (-1.2,-1.2) rectangle (1.2,1.2);
		\path [draw=black, -o-={0.5}{>}] (-1.2,0.8) to[bend right=30] (1.2,0.8);
        \path [draw=black, -o-={0.5}{>}] (-1.2,-0.8) to[bend left=30] (1.2,-0.8);
        \end{tikzpicture}}
    }

\newcommand{\hshape}
    {\raisebox{-.30in}{
	   \begin{tikzpicture}[scale=0.5]
		\tikzset{->-/.style=
			{decoration={markings,mark=at position ##1 with
			{\arrow{latex}}},postaction={decorate}}}
		\filldraw[draw=white, fill=white] (-1.2,-1.2) rectangle (1.2,1.2);
        \draw [decoration={markings, mark=at position 0.5 with {\arrow{>}}},postaction={decorate}](-1.2,0.8)--(-0.4,0);
        \draw [decoration={markings, mark=at position 0.5 with {\arrow{>}}},postaction={decorate}](-1.2,-0.8)--(-0.4,0);
		\draw [decoration={markings, mark=at position 0.5 with {\arrow{<}}},postaction={decorate}](-0.4,0)--(0.4,0);
        \draw [decoration={markings, mark=at position 0.6 with {\arrow{>}}},postaction={decorate}](0.4,0)--(1.2,0.8);
        \draw [decoration={markings, mark=at position 0.6 with {\arrow{>}}},postaction={decorate}](0.4,0)--(1.2,-0.8);
        \end{tikzpicture}}
    }

\newcommand{\compSLtwoTT}
{\begin{tikzpicture}[scale=1.5]

    \coordinate (T1) at (0.5, 1);  
    \coordinate (T2) at (1.5, 0); 
    \coordinate (T3) at (0.5, -1); 
    
    \draw[thick] (T1) -- node[above right] {$e_2$} (T2);
    \draw[thick] (T2) -- node[below right] {$e_1$} (T3);
    \draw[thick] (T3) -- (T1); 
    
    \coordinate (Tp1) at (0,1);    
    \coordinate (Tp2) at (-1,0);   
    \coordinate (Tp3) at (0, -1); 
    
    \draw[thick] (Tp1) -- node[above left] {$e_3$} (Tp2);
    \draw[thick] (Tp2) -- node[below left] {$e_4$} (Tp3);
    \draw[thick] (Tp3) -- (Tp1); 
    
    \node at (1.7, 0) {$T$};
    \node at (-1.2, 0) {$T'$};
    \node at (0.25, 0) {$f$};
    
    \draw ($(Tp2)!0.25!(Tp1)$)[out=-45, in=180] to ($(Tp3)!0.5!(Tp1)$);
    \draw ($(Tp2)!0.5!(Tp1)$)[out=-45, in=180] to ($(Tp3)!0.66!(Tp1)$);
    \draw ($(Tp3)!0.5!(Tp2)$)[out=45, in=180] to ($(Tp3)!0.33!(Tp1)$);

    \draw ($(T2)!0.25!(T3)$)[out=135, in=0] to ($(T1)!0.5!(T3)$);
    \draw ($(T2)!0.5!(T1)$)[out=-135, in=0] to ($(T3)!0.66!(T1)$);
    \draw ($(T3)!0.5!(T2)$)[out=135, in=0] to ($(T3)!0.33!(T1)$);
    \end{tikzpicture}}

\newcommand{\rhombSLtwoTT}
    {\begin{tikzpicture}[scale=1.3]
    \begin{scope}[xshift=-3cm]
        \coordinate (V1) at (1, {sqrt(3)}); 
    \coordinate (V2) at (2, 0);  
    \coordinate (V3) at (0, 0);        
    \coordinate (M1) at ($(V1)!0.5!(V2)$); 
    \coordinate (M2) at ($(V2)!0.5!(V3)$); 
    \coordinate (M3) at ($(V3)!0.5!(V1)$); 

    \fill[gray!30] (V1) -- (M3) -- (M2) -- (M1) -- cycle;
    
    \draw[thick] (V1) -- (V2) -- (V3) -- cycle;

    \node[above] at (V1) {$0$};
    \node[right] at (V2) {$0$};
    \node[left] at (V3) {$0$};

    \node[above right, xshift=-3pt, yshift=-5pt] at (M1) {$i(\Gamma, e_2)$};
    \node[below] at (M2) {$i(\Gamma, e_1)$};
    \node[above left, xshift=3pt, yshift=-5pt] at (M3) {$i(\Gamma, e_3)$};

    \node[below, font=\huge] at (V1) {-};
    \node[left, font=\tiny] at (M1) {\textbf{+}};
    \node[above, font=\huge] at (M2) {-};
    \node[right, font=\tiny] at (M3) {\textbf{+}};
    \end{scope}

    \begin{scope}[xshift=0cm]
    \coordinate (V1) at (1, {sqrt(3)}); 
    \coordinate (V2) at (2, 0);    
    \coordinate (V3) at (0, 0);        
    \coordinate (M1) at ($(V1)!0.5!(V2)$); 
    \coordinate (M2) at ($(V2)!0.5!(V3)$); 
    \coordinate (M3) at ($(V3)!0.5!(V1)$); 

    \fill[gray!30] (V3) -- (M2) -- (M1) -- (M3) -- cycle;
    
    \draw[thick] (V1) -- (V2) -- (V3) -- cycle;

    \node[above] at (V1) {$0$};
    \node[right] at (V2) {$0$};
    \node[left] at (V3) {$0$};

    \node[above right, xshift=-3pt, yshift=-5pt] at (M1) {$i(\Gamma, e_2)$};
    \node[below] at (M2) {$i(\Gamma, e_1)$};
    \node[above left, xshift=3pt, yshift=-5pt] at (M3) {$i(\Gamma, e_3)$};

    \node[above right, font=\huge, xshift=3pt, yshift=-2pt] at (V3) {-};
    \node[below left, font=\huge, xshift=-3pt, yshift=2pt] at (M1) {-};
    \node[above left, font=\tiny, xshift=4pt, yshift=-1pt] at (M2) {\textbf{+}};
    \node[below right, font=\tiny, xshift=-4pt, yshift=1pt] at (M3) {\textbf{+}};
    \end{scope}

    \begin{scope}[xshift=3cm]
        \coordinate (V1) at (1, {sqrt(3)}); 
    \coordinate (V2) at (2, 0);    
    \coordinate (V3) at (0, 0);        
    \coordinate (M1) at ($(V1)!0.5!(V2)$); 
    \coordinate (M2) at ($(V2)!0.5!(V3)$); 
    \coordinate (M3) at ($(V3)!0.5!(V1)$); 

    \fill[gray!30] (V2) -- (M1) -- (M3) -- (M2) -- cycle;
    
    \draw[thick] (V1) -- (V2) -- (V3) -- cycle;

    \node[above] at (V1) {$0$};
    \node[right] at (V2) {$0$};
    \node[left] at (V3) {$0$};

    \node[above right, xshift=-3pt, yshift=-5pt] at (M1) {$i(\Gamma, e_2)$};
    \node[below] at (M2) {$i(\Gamma, e_1)$};
    \node[above left, xshift=3pt, yshift=-5pt] at (M3) {$i(\Gamma, e_3)$};

    \node[above left, font=\huge, xshift=-3pt, yshift=-2pt] at (V2) {-};
    \node[below left, font=\tiny, xshift=4pt, yshift=1pt] at (M1) {\textbf{+}};
    \node[above right, font=\tiny, xshift=-4pt, yshift=-1pt] at (M2) {\textbf{+}};
    \node[below right, font=\huge, xshift=3pt, yshift=2pt] at (M3) {-};
    \end{scope}
    \end{tikzpicture}}

\newcommand{\SLthreeTT}{
    \begin{tikzpicture}[scale=2.0,
        edge/.style={draw=black,thick},
        corner/.style={draw=black, thick}
        ]
        \coordinate (V1) at (0, 0);                 
        \coordinate (V2) at (2, 0);                 
        \coordinate (V3) at (1, {sqrt(3)});     

        \draw[edge] (V1) -- (V2) node[midway, below] {$e_1$};
        \draw[edge] (V2) -- (V3) node[midway, above right] {$e_2$};
        \draw[edge] (V3) -- (V1) node[midway, above left] {$e_3$};
    
        \coordinate (L1) at ($(V1)!1/6!(V3)$);
        \coordinate (L2) at ($(V1)!2/6!(V3)$);
        \coordinate (L3) at ($(V1)!3/6!(V3)$);
        \coordinate (L4) at ($(V1)!4/6!(V3)$);
        \coordinate (L5) at ($(V1)!5/6!(V3)$);
    
        \coordinate (R1) at ($(V3)!1/6!(V2)$);
        \coordinate (R2) at ($(V3)!2/6!(V2)$);
        \coordinate (R3) at ($(V3)!3/6!(V2)$);
        \coordinate (R4) at ($(V3)!4/6!(V2)$);
        \coordinate (R5) at ($(V3)!5/6!(V2)$);
    
        \coordinate (B1) at ($(V1)!1/6!(V2)$);
        \coordinate (B2) at ($(V1)!2/6!(V2)$);
        \coordinate (B3) at ($(V1)!3/6!(V2)$);
        \coordinate (B4) at ($(V1)!4/6!(V2)$);
        \coordinate (B5) at ($(V1)!5/6!(V2)$);
    
        \coordinate (C) at (1, {sqrt(3)/3});

        \draw[edge, -o-={0.55}{>}] (L1) to[out=-30, in=90] (B1) node[above right] {$r_{31}$};
        \draw[edge, -o-={0.5}{<}] (L2) to[out=-30, in=90] (B2) node[above right] {$l_{13}$};
        \draw[edge, -o-={0.55}{>}] (B5) to[out=90, in=210] (R5) node[left] {$r_{12}$};
        \draw[edge, -o-={0.5}{<}] (B4) to[out=90, in=210] (R4) node[left, yshift=2pt] {$l_{21}$};
        \draw[edge, -o-={0.55}{>}] (L5) to[out=-30, in=210] (R1) node[below left, xshift=-4pt, yshift=-2pt] {$r_{23}$};
        \draw[edge, -o-={0.5}{<}] (L4) to[out=-30, in=210] (R2) node[below left, xshift=-22pt, yshift=-2pt] {$l_{32}$};
        \draw[edge] (B3) -- (C) node[above] {$H_d$};
        \draw[edge] (C) to[out=150, in=-30] (L3);
        \draw[edge={.6}] (C) to[out=30, in=210] (R3);
    \end{tikzpicture}}

\newcommand{\nhoneycomb}{\begin{tikzpicture}[
    scale=0.5,
    edge/.style={
        draw=black,
        postaction={decorate, decoration={markings, mark=at position 0.6 with {\arrow{>}}}}
    }
]
    \draw[thick] (0,0) -- (10,0) -- (5, {5*sqrt(3)}) -- cycle;

    \coordinate (C) at (5, {5/sqrt(3)});

    \def\drawcorner{
        \coordinate (B1) at (1.5, 0);
        \coordinate (B2) at (3, 0);
        \coordinate (L1) at (60:1.5);
        \coordinate (L2) at (60:3);
        
        \coordinate (Y1) at (1.5, {1.5/sqrt(3)}); 
        \coordinate (Y2) at (3, {1.5/sqrt(3)}); 
        \coordinate (Y3) at (2.25, {3.75/sqrt(3)}); 
        
        \coordinate (H1) at (2.25, {2.25/sqrt(3)}); 
        
        \draw[edge] (Y1) -- (B1);
        \draw[edge] (Y2) -- (B2);
        \draw[edge] (Y1) -- (L1);
        \draw[edge] (Y3) -- (L2);
        
        \draw[edge] (Y1) -- (H1);
        \draw[edge] (Y2) -- (H1);
        \draw[edge] (Y3) -- (H1);
        
        \draw[edge] (Y2) -- ++(30:0.8);
        \draw[edge] (Y3) -- ++(30:0.8);
    }

    \begin{scope}[rotate around={0:(C)}] \drawcorner \end{scope}
    \begin{scope}[rotate around={120:(C)}] \drawcorner \end{scope}
    \begin{scope}[rotate around={240:(C)}] \drawcorner \end{scope}

    \node[scale=1.0] at (5, {3.875/sqrt(3)}) {$\cdots$};
    \node[scale=1.0] at (5, {6.125/sqrt(3)}) {$\cdots$};
    \node[scale=1.0, rotate=60] at (3.875, {5/sqrt(3)}) {$\cdots$};
    \node[scale=1.0, rotate=-60] at (6.125, {5/sqrt(3)}) {$\cdots$};
    
    \draw[decorate,decoration={brace,amplitude=10pt,mirror}, thick] 
        (1.5,-0.2) -- (8.5,-0.2) node[midway,below=12pt] {$d$};
        
    \draw[decorate,decoration={brace,amplitude=10pt}, thick] 
        ($(0.75, {0.75*sqrt(3)}) + (-0.173, 0.1)$) -- ($(4.25, {4.25*sqrt(3)}) + (-0.173, 0.1)$) node[midway,above left=8pt] {$d$};

    \draw[decorate,decoration={brace,amplitude=10pt}, thick] 
        ($(5.75, {4.25*sqrt(3)}) + (0.173, 0.1)$) -- ($(9.25,{0.75*sqrt(3)}) + (0.173, 0.1)$) node[midway,above right=8pt] {$d$};
\end{tikzpicture}}

\newcommand{\ladderandbraid}
    {\begin{tikzpicture}[
        scale=1.0,
        arr_rt/.style={postaction={decorate, decoration={markings, mark=at position 0.55 with {\arrow{>}}}}},
        arr_lt/.style={postaction={decorate, decoration={markings, mark=at position 0.55 with {\arrow{<}}}}},
        arr_dn_up/.style={postaction={decorate, decoration={markings, 
            mark=at position 0.21 with {\arrow{<}}, 
            mark=at position 0.79 with {\arrow{>}}}}},
        arr_up_dn/.style={postaction={decorate, decoration={markings, 
            mark=at position 0.21 with {\arrow{>}}, 
            mark=at position 0.79 with {\arrow{<}}}}},
        bnd/.style={draw=black, thick}
    ]
    
        \def\drawbiangle{
            
            \path[name path=top_bnd, bnd] (0,0) to[bend left=40] (4,0);
            \path[name path=bot_bnd, bnd] (0,0) to[bend right=40] (4,0);
    
            \node[above left] at (4, 0.25) {$e_1$};
            \node[below left] at (4, -0.25) {$e_1'$};
    
            \foreach \x/\i in {0.8/1, 1.6/2, 2.4/3, 3.2/4} {
                \path[name path=vert\i] (\x,-1.5) -- (\x,1.5);
                \path[name intersections={of=top_bnd and vert\i, by=T\i}];
                \path[name intersections={of=bot_bnd and vert\i, by=B\i}];
            }
    
            \node[above] at (T1) {\textcolor{blue}{$-$}};
            \node[above] at (T2) {\textcolor{blue}{$-$}};
            \node[above] at (T3) {\textcolor{red}{$+$}};
            \node[above] at (T4) {\textcolor{red}{$+$}};
    
            \node[below] at (B1) {\textcolor{blue}{$-$}};
            \node[below] at (B2) {\textcolor{blue}{$-$}};
            \node[below] at (B3) {\textcolor{red}{$+$}};
            \node[below] at (B4) {\textcolor{red}{$+$}};
            }
    
        \begin{scope}[xshift=0cm]
            \drawbiangle
            \draw[arr_up_dn] (B1) -- (T1);
            \draw[arr_up_dn] (B2) -- (T2);
            \draw[arr_dn_up] (B3) -- (T3);
            \draw[arr_dn_up] (B4) -- (T4);
    
            \draw[arr_lt] ($(B1)!0.5!(T1)$) -- ($(B2)!0.5!(T2)$);
            \draw[arr_lt] ($(B2)!0.33!(T2)$) -- ($(B3)!0.33!(T3)$);
            \draw[arr_lt] ($(B2)!0.67!(T2)$) -- ($(B3)!0.67!(T3)$);
            \draw[arr_lt] ($(B3)!0.5!(T3)$) -- ($(B4)!0.5!(T4)$);
        \end{scope}
    
        \begin{scope}[xshift=5.5cm]
            \drawbiangle
            \draw[arr_rt] (B1) -- (T3);
            \draw[arr_rt] (B2) -- (T4);
            \draw[arr_lt] (B3) -- (T1);
            \draw[arr_lt] (B4) -- (T2);
        \end{scope}
    
    \end{tikzpicture}}

\newcommand{\compSLthreeTT}{
    \begin{tikzpicture}[scale=1.0]

    \def\H{6}   
    \def\W{2.4}   
    \def\gap{1} 

    \coordinate (L_bot) at (0, 0);
    \coordinate (L_top) at (0, \H);
    \coordinate (L_mid) at (-\H/2, \H/2);
    
    \draw[thick] (L_bot) -- (L_top) -- node[above left] {$e_3$} (L_mid) -- node[below left] {$e_4$} cycle;

    \coordinate (R_BL) at (\gap, 0);
    \coordinate (R_TL) at (\gap, \H);
    \coordinate (R_BR) at (\gap+\W, 0);
    \coordinate (R_TR) at (\gap+\W, \H);
    
    \draw[thick] (R_BL) -- (R_TL) -- (R_TR) -- (R_BR) -- cycle;

    \coordinate (R_bot) at (2*\gap+\W, 0);
    \coordinate (R_top) at (2*\gap+\W, \H);
    \coordinate (R_mid) at (2*\gap+\W+\H/2, \H/2);
    
    \draw[thick] (R_bot) -- (R_top) -- node[above right] {$e_2$} (R_mid) -- node[below right] {$e_1$} cycle;

    \foreach \i in {1,2,3,4} {
        \def\y{\H * \i / 5}
        
        \coordinate (L\i) at (0, \y);
        
        \coordinate (RL\i) at (\gap, \y);
        
        \coordinate (RR\i) at (\gap+\W, \y);
        
        \coordinate (R\i) at (2*\gap+\W, \y);
    }

     \node[left] at (RL1) {\textcolor{red}{$+$}};
     \node[right] at (L1) {\textcolor{blue}{$-$}};
     \node[left] at (RL2) {\textcolor{blue}{$-$}};
     \node[right] at (L2) {\textcolor{red}{$+$}};
     \node[left] at (RL3) {\textcolor{blue}{$-$}};
     \node[right] at (L3) {\textcolor{red}{$+$}};
     \node[left] at (RL4) {\textcolor{red}{$+$}};
     \node[right] at (L4) {\textcolor{blue}{$-$}};
     \node[right] at (RR1) {\textcolor{blue}{$-$}};
     \node[left] at (R1) {\textcolor{red}{$+$}};
     \node[right] at (RR2) {\textcolor{blue}{$-$}};
     \node[left] at (R2) {\textcolor{red}{$+$}};
     \node[right] at (RR3) {\textcolor{red}{$+$}};
     \node[left] at (R3) {\textcolor{blue}{$-$}};
     \node[right] at (RR4) {\textcolor{red}{$+$}};
     \node[left] at (R4) {\textcolor{blue}{$-$}};

     \draw[edge, -o-={0.5}{<}] (-0.6,5.4) -- (L4);
     \draw[edge, -o-={0.6}{>}] (-1,3.6) -- (L3);
     \draw[edge, -o-={0.6}{>}] (-1,3.6) -- (-1.7,4.3);
     \draw[edge, -o-={0.5}{>}] (-1,3.6) -- (-2.2,2.2);
     \draw[edge, -o-={0.5}{>}] (-1.2,1.2) -- (L2);
     \draw[edge, -o-={0.5}{<}] (-0.6,0.6) -- (L1);

     \draw[edge, -o-={0.5}{>}] (1.8,4.8) -- (RR4);
     \draw[edge, -o-={0.5}{>}] (1.8,4.8) -- (RL4);
     \draw[edge, -o-={0.5}{>}] (1.8,4.8) -- (1.8,3.6);
     \draw[edge, -o-={0.5}{>}] (RL3) -- (1.8,3.6);
     \draw[edge, -o-={0.5}{>}] (2.6,3.6) -- (1.8,3.6);
     \draw[edge, -o-={0.5}{>}] (2.6,3.6) -- (RR3);
     \draw[edge, -o-={0.5}{>}] (RR2) -- (2.6,2.4);
     \draw[edge, -o-={0.5}{>}] (2.6,3.6) -- (2.6,2.4);
     \draw[edge, -o-={0.5}{>}] (RR2) -- (2.6,2.4);
     \draw[edge, -o-={0.5}{>}] (RL2) -- (2.6,2.4);
     \draw[edge, -o-={0.5}{<}] (RL1) -- (RR1);

     \draw[edge, -o-={0.5}{>}] (R4) -- (5, 5.4);
     \draw[edge, -o-={0.5}{>}] (R3) -- (5.6, 4.8);
     \draw[edge, -o-={0.5}{<}] (R2) -- (4.8, 2.4);
     \draw[edge, -o-={0.5}{<}] (R1) -- (4.8, 1.2);
     \draw[edge, -o-={0.5}{<}] (5.2,0.8) -- (4.8, 1.2);
     \draw[edge, -o-={0.5}{<}] (5.1,1.8) -- (4.8, 1.2);
     \draw[edge, -o-={0.5}{>}] (4.8, 2.4) -- (5.1,1.8);
     \draw[edge, -o-={0.5}{>}] (4.8, 2.4) -- (6.4,4);
     \draw[edge, -o-={0.5}{>}] (5.4,1.8) -- (5.1,1.8);
     \draw[edge, -o-={0.5}{>}] (5.4,1.8) -- (5.8,1.4);
     \draw[edge, -o-={0.5}{>}] (5.4,1.8) -- (7,3.4);

    \draw[draw=white] (0.1, \H+0.5) -- (\gap-0.1, \H+0.5) node[midway, below] {$f'$};
    \draw[draw=white] (\gap+\W+0.1, \H+0.5) -- (2*\gap+\W-0.1, \H+0.5) node[midway, below] {$f$};

\end{tikzpicture}
}

\newcommand{\trianglecoords}[4]{ 
    \coordinate (A#4) at ({#1 - #3/2}, {#2 - #3*sqrt(3)/6});
    \coordinate (B#4) at ({#1 + #3/2}, {#2 - #3*sqrt(3)/6});
    \coordinate (C#4) at ({#1}, {#2 + #3*sqrt(3)/3});
}

\newcommand{\sidepoints}[4]{ 
    \coordinate (P1#4) at ($(#1)!0.33!(#2)$);
    \coordinate (P2#4) at ($(#1)!0.67!(#2)$);
    \coordinate (P3#4) at ($(#2)!0.33!(#3)$);
    \coordinate (P4#4) at ($(#2)!0.67!(#3)$);
    \coordinate (P5#4) at ($(#3)!0.33!(#1)$);
    \coordinate (P6#4) at ($(#3)!0.67!(#1)$);
    \coordinate (M#4)  at ($(#1)!0.5!(#2)!0.33!(#3)$);
    
    \coordinate (MAB#4) at ($(#1)!0.5!(#2)$);
    \coordinate (MBC#4) at ($(#2)!0.5!(#3)$);
    \coordinate (MCA#4) at ($(#3)!0.5!(#1)$);
}

\newcommand{\drawpoints}[1]{
    \node[point] at (P1#1) {};
    \node[point] at (P2#1) {};
    \node[point] at (P3#1) {};
    \node[point] at (P4#1) {};
    \node[point] at (P5#1) {};
    \node[point] at (P6#1) {};
    \node[point] at (M#1) {};
}

\newcommand{\DSdiagram}{
\begin{tikzpicture}[
    thin, black,
    point/.style={circle, fill=black, inner sep=1.2pt},
    midarrow/.style={postaction={decorate, decoration={markings, mark=at position 0.6 with {\arrow{stealth}}}}},
    labelpoint/.style={font=\scriptsize, text=black}
]

\def\r{2}
\def\dx{2.5}
\def\dy{2.2}

\begin{scope}[shift={(0,0)}]
    \trianglecoords{0}{0}{\r}{1}
    \draw (A1) -- (B1) -- (C1) -- cycle;
    \sidepoints{A1}{B1}{C1}{1}
    \drawpoints{1}
    
    \node[labelpoint, below] at (P11) {$e_{11}$};
    \node[labelpoint, below] at (P21) {$e_{12}$};
    \node[labelpoint, right] at (P31) {$e_{21}$};
    \node[labelpoint, right] at (P41) {$e_{22}$};
    \node[labelpoint, left] at (P51) {$e_{31}$};
    \node[labelpoint, left] at (P61) {$e_{32}$};
    \node[labelpoint, above] at (M1) {$e$};
\end{scope}

\begin{scope}[shift={(\dx,0)}]
    \trianglecoords{0}{0}{\r}{2}
    \draw (A2) -- (B2) -- (C2) -- cycle;
    \sidepoints{A2}{B2}{C2}{2}
    \drawpoints{2}
    
    \node[labelpoint, below] at (P12) {$1$}; \node[labelpoint, below] at (P22) {$2$};
    \node[labelpoint, right] at (P32) {$2$}; \node[labelpoint, right] at (P42) {$1$};
    \node[labelpoint, left] at (P52) {$0$};  \node[labelpoint, left] at (P62) {$0$};
    
    \draw[-o-={0.5}{>}, out=90, in=-150] ($(P12)!0.6!(P22)$) to node[midway, below right] {$r_{12}$} ($(P32)!0.4!(P42)$); 
    \node[labelpoint, below left] at (M2) {$1$};
\end{scope}

\begin{scope}[shift={(2*\dx,0)}]
    \trianglecoords{0}{0}{\r}{3}
    \draw (A3) -- (B3) -- (C3) -- cycle;
    \sidepoints{A3}{B3}{C3}{3}
    \drawpoints{3}
    
    \node[labelpoint, below] at (P13) {$2$}; \node[labelpoint, below] at (P23) {$1$};
    \node[labelpoint, right] at (P33) {$1$}; \node[labelpoint, right] at (P43) {$2$};
    \node[labelpoint, left] at (P53) {$0$};  \node[labelpoint, left] at (P63) {$0$};
    
    \draw[-o-={0.5}{<}, out=90, in=-150] ($(P13)!0.6!(P23)$) to node[midway, below right=-2pt] {$l_{21}$} ($(P33)!0.4!(P43)$); 
    \node[labelpoint, below left] at (M3) {$2$};
\end{scope}

\begin{scope}[shift={(0,-\dy)}]
    \trianglecoords{0}{0}{\r}{4}
    \draw (A4) -- (B4) -- (C4) -- cycle;
    \sidepoints{A4}{B4}{C4}{4}
    \drawpoints{4}
    
    \node[labelpoint, below] at (P14) {$0$}; \node[labelpoint, below] at (P24) {$0$};
    \node[labelpoint, right] at (P34) {$1$}; \node[labelpoint, right] at (P44) {$2$};
    \node[labelpoint, left] at (P54) {$2$};  \node[labelpoint, left] at (P64) {$1$};
    
    \draw[-o-={0.5}{>}, out=-150, in=-30] ($(P34)!0.6!(P44)$) to node[midway, above] {$r_{23}$} ($(P54)!0.4!(P64)$);
    \node[labelpoint, right] at (M4) {$1$};
\end{scope}

\begin{scope}[shift={(\dx,-\dy)}]
    \trianglecoords{0}{0}{\r}{5}
    \draw (A5) -- (B5) -- (C5) -- cycle;
    \sidepoints{A5}{B5}{C5}{5}
    \drawpoints{5}
    
    \node[labelpoint, below] at (P15) {$0$}; \node[labelpoint, below] at (P25) {$0$};
    \node[labelpoint, right] at (P35) {$2$}; \node[labelpoint, right] at (P45) {$1$};
    \node[labelpoint, left] at (P55) {$1$};  \node[labelpoint, left] at (P65) {$2$};
    
    \draw[-o-={0.5}{<}, out=-150, in=-30] ($(P35)!0.6!(P45)$) to node[midway, above] {$l_{32}$} ($(P55)!0.4!(P65)$);
    \node[labelpoint, right] at (M5) {$2$};
\end{scope}

\begin{scope}[shift={(2*\dx,-\dy)}]
    \trianglecoords{0}{0}{\r}{6}
    \draw (A6) -- (B6) -- (C6) -- cycle;
    \sidepoints{A6}{B6}{C6}{6}
    \drawpoints{6}
    
    \node[labelpoint, below] at (P16) {$2$}; \node[labelpoint, below] at (P26) {$1$};
    \node[labelpoint, right] at (P36) {$0$}; \node[labelpoint, right] at (P46) {$0$};
    \node[labelpoint, left] at (P56) {$1$};  \node[labelpoint, left] at (P66) {$2$};
    
    \draw[-o-={0.5}{>}, out=-30, in=90] ($(P56)!0.6!(P66)$) to node[midway, below left=-4pt] {$r_{31}$} ($(P16)!0.4!(P26)$);
    \node[labelpoint, right] at (M6) {$1$};
\end{scope}

\begin{scope}[shift={(0,-2*\dy)}]
    \trianglecoords{0}{0}{\r}{7}
    \draw (A7) -- (B7) -- (C7) -- cycle;
    \sidepoints{A7}{B7}{C7}{7}
    \drawpoints{7}
    
    \node[labelpoint, below] at (P17) {$1$}; \node[labelpoint, below] at (P27) {$2$};
    \node[labelpoint, right] at (P37) {$0$}; \node[labelpoint, right] at (P47) {$0$};
    \node[labelpoint, left] at (P57) {$2$};  \node[labelpoint, left] at (P67) {$1$};
    
    \draw[-o-={0.55}{>}, out=90, in=-30] ($(P17)!0.4!(P27)$) to node[midway, below left=-4pt] {$l_{13}$} ($(P57)!0.6!(P67)$);
    \node[labelpoint, right] at (M7) {$2$};
\end{scope}

\begin{scope}[shift={(\dx,-2*\dy)}]
    \trianglecoords{0}{0}{\r}{8}
    \draw (A8) -- (B8) -- (C8) -- cycle;
    \sidepoints{A8}{B8}{C8}{8}
    \drawpoints{8}
    
    \node[labelpoint, below] at (P18) {$2$}; \node[labelpoint, below] at (P28) {$1$};
    \node[labelpoint, right] at (P38) {$2$}; \node[labelpoint, right] at (P48) {$1$};
    \node[labelpoint, left] at (P58) {$2$};   \node[labelpoint, left] at (P68) {$1$};
    
    \draw[-o-={0.5}{>}] (M8) -- node[midway, right] {$H_1$} (MAB8);
    \draw[-o-={0.5}{>}] (M8) -- (MBC8);
    \draw[-o-={0.5}{>}] (M8) -- (MCA8);
    \node[labelpoint, above] at (M8) {$3$};
\end{scope}

\begin{scope}[shift={(2*\dx,-2*\dy)}]
    \trianglecoords{0}{0}{\r}{9}
    \draw (A9) -- (B9) -- (C9) -- cycle;
    \sidepoints{A9}{B9}{C9}{9}
    \drawpoints{9}
    
    \node[labelpoint, below] at (P19) {$1$};  \node[labelpoint, below] at (P29) {$2$};
    \node[labelpoint, right] at (P39) {$1$};  \node[labelpoint, right] at (P49) {$2$};
    \node[labelpoint, left] at (P59) {$1$};  \node[labelpoint, left] at (P69) {$2$};
    
    \draw[-o-={0.5}{>}] (MAB9) -- node[midway, right] {$H_{-1}$}(M9);
    \draw[-o-={0.5}{>}] (MBC9) -- (M9);
    \draw[-o-={0.5}{>}] (MCA9) -- (M9);
    \node[labelpoint, above] at (M9) {$3$};
\end{scope}

\end{tikzpicture}
}

\newcommand{\rhombSLthreeTT}{
\begin{tikzpicture}[scale=1.5]

    \newcommand{\myminus}{{\tiny $\boldsymbol{-}$}}
    \newcommand{\myplus}{{\tiny $\boldsymbol{+}$}}

    \newcommand{\fillVertical}[4]{
        \fill[gray!30] (##1) -- (##2) -- (##3) -- (##4) -- cycle;
        \node[yshift=6pt] at (##1) {\myminus};
        \node[xshift=-5pt] at (##2) {\myplus};
        \node[yshift=-6pt] at (##3) {\myminus};
        \node[xshift=5pt] at (##4) {\myplus};
    }

    \newcommand{\fillLeftLeaning}[4]{
        \fill[gray!30] (##1) -- (##2) -- (##3) -- (##4) -- cycle;
        \node[xshift=6pt, yshift=3.5pt] at (##1) {\myminus};
        \node[xshift=-3.5pt, yshift=3.5pt] at (##2) {\myplus};
        \node[xshift=-6pt, yshift=-3.5pt] at (##3) {\myminus};
        \node[xshift=3.5pt, yshift=-3.5pt] at (##4) {\myplus};
    }

    \newcommand{\fillRightLeaning}[4]{
        \fill[gray!30] (##1) -- (##2) -- (##3) -- (##4) -- cycle;
        \node[xshift=3.5pt, yshift=3.5pt] at (##1) {\myplus};
        \node[xshift=-6pt, yshift=3.5pt] at (##2) {\myminus};
        \node[xshift=-3.5pt, yshift=-3.5pt] at (##3) {\myplus};
        \node[xshift=6pt, yshift=-3.5pt] at (##4) {\myminus};
    }

    \newcommand{\defCoords}{
        \coordinate (N00) at (0, 0);                 
        \coordinate (N10) at (2/3, 0);               
        \coordinate (N20) at (4/3, 0);               
        \coordinate (N30) at (2, 0);                 
        \coordinate (N01) at (1/3, {sqrt(3)/3});     
        \coordinate (N11) at (1, {sqrt(3)/3}); 
        \coordinate (N21) at (5/3, {sqrt(3)/3});     
        \coordinate (N02) at (2/3, {2*sqrt(3)/3});   
        \coordinate (N12) at (4/3, {2*sqrt(3)/3});   
        \coordinate (N03) at (1, {sqrt(3)});         
    }

    \newcommand{\drawOuter}{
        \draw[thick] (N00) -- (N30) -- (N03) -- cycle;
    }
    \newcommand{\drawHorizontals}{
        \draw (N01) -- (N21); 
        \draw (N02) -- (N12);
    }
    \newcommand{\drawSixties}{ 
        \draw (N10) -- (N12); 
        \draw (N20) -- (N21);
    }
    \newcommand{\drawOneTwenties}{ 
        \draw (N10) -- (N01); 
        \draw (N20) -- (N02);
    }

    \newcommand{\drawLabels}{
        \node[below, font=\small] at (N10) {$e_{11}$};
        \node[below, font=\small] at (N20) {$e_{12}$};
        \node[right, font=\small] at (N21) {$e_{21}$};
        \node[right, font=\small] at (N12) {$e_{22}$};
        \node[left, font=\small] at (N02) {$e_{31}$};
        \node[left, font=\small] at (N01) {$e_{32}$};
        \node[above] at (N03) {$0$};
        \node[right] at (N30) {$0$};
        \node[left] at (N00) {$0$};
    }

    \begin{scope}[xshift=-3cm]
        \defCoords
        \fillVertical{N10}{N11}{N02}{N01}
        \fillVertical{N20}{N21}{N12}{N11}
        \fillVertical{N11}{N12}{N03}{N02}
        \drawOuter
        \drawSixties
        \drawOneTwenties
        \drawLabels
        \node[font=\small, below=2pt] at (N11) {$e$};
    \end{scope}

    \begin{scope}[xshift=0cm]
        \defCoords
        \fillLeftLeaning{N00}{N10}{N11}{N01}
        \fillLeftLeaning{N10}{N20}{N21}{N11}
        \fillLeftLeaning{N01}{N11}{N12}{N02}
        
        \drawOuter
        \drawHorizontals
        \drawSixties
        \drawLabels
        \node[font=\small, above right, yshift=-2pt] at (N11) {$e$};
    \end{scope}

    \begin{scope}[xshift=3cm]
        \defCoords
        \fillRightLeaning{N10}{N20}{N11}{N01}
        \fillRightLeaning{N20}{N30}{N21}{N11}
        \fillRightLeaning{N11}{N21}{N12}{N02}
        
        \drawOuter
        \drawHorizontals
        \drawOneTwenties
        \drawLabels

        \node[font=\small, above left, yshift=-2pt] at (N11) {$e$};
    \end{scope}

\end{tikzpicture}
}

\newcommand{\TidyingUps}{\begin{tikzpicture}[
    scale=1.0,
    wall/.style={draw=black, thick}, 
    strand/.style={draw=black, thick},        
    arrow/.style={draw=black, thin, -{To[scale=1.0]}}, 
    biarrow/.style={draw=black, thin, {To[scale=1.0]}-{To[scale=1.0]}},
    sign/.style={font=\bfseries\scriptsize, inner sep=1pt}
]

    \newcommand{\myplus}{\tiny \textcolor{red}{$+$}}
    \newcommand{\myminus}{\tiny \textcolor{blue}{$-$}}

    \begin{scope}[xshift=0cm]
        \fill[gray!15] (0, -0.6) rectangle (0.6, 0.6);
        \fill[gray!15] (2.2, -0.6) rectangle (2.8, 0.6);
        \draw[wall] (0, -0.6) -- (0, 0.6);
        \draw[strand, name path=arc1, -o-={0.65}{>}] (-0.3, -0.4) to[out=20, in=-20, looseness=2.5] (-0.3, 0.4);
        
        \node[sign, anchor=south west] at (0, 0.25) {\myplus};
        \node[sign, anchor=north west] at (0, -0.25) {\myminus};
        
        \draw[arrow] (0.7, 0) -- (1.3, 0);
        
        \draw[wall] (2.2, -0.6) -- (2.2, 0.6);
        \draw[strand, -o-={0.65}{>}] (1.6, -0.4) to[out=40, in=-40, looseness=1.8] (1.6, 0.4);
    \end{scope}

    \begin{scope}[xshift=3.5cm]
        \fill[gray!15] (0, -0.6) rectangle (0.8, 0.6);
        \fill[gray!15] (2.6, -0.6) rectangle (3.2, 0.6);
        \draw[wall] (0, -0.6) -- (0, 0.6);
        \coordinate (Y1) at (0.3, 0);
        \draw[strand, -o-={0.8}{>}] (Y1) -- (-0.3, -0.5);
        \draw[strand, -o-={0.8}{>}] (Y1) -- (-0.3, 0.5);
        \draw[strand, -o-={0.8}{>}] (Y1) -- (0.8, 0);
        
        \node[sign, anchor=south west] at (0, 0.15) {\myplus};
        \node[sign, anchor=north west] at (0, -0.15) {\myplus};
        
        \draw[arrow] (1.1, 0) -- (1.7, 0);
        
        \draw[wall] (2.6, -0.6) -- (2.6, 0.6);
        \coordinate (Y2) at (2.2, 0);
        \draw[strand, -o-={0.8}{>}] (Y2) -- (1.9, -0.5) ;
        \draw[strand, -o-={0.8}{>}] (Y2) -- (1.9, 0.5);
        \draw[strand, -o-={0.8}{>}] (Y2) -- (3.0, 0);
        
        \node[sign, anchor=north east] at (2.6, 0) {\myminus};
    \end{scope}

    \begin{scope}[xshift=7.8cm]
        \fill[gray!15] (0, -0.6) rectangle (0.8, 0.6);
        \fill[gray!15] (3.1, -0.6) rectangle (3.5, 0.6);
        \draw[wall] (0, -0.6) -- (0, 0.6);
        \draw[strand, -o-={0.5}{<}] (-0.3, 0.4) -- (0.8, 0.4);
        \draw[strand, -o-={0.5}{>}] (-0.3, -0.4) -- (0.8, -0.4);
        \draw[strand] (0.5, -0.4) -- (0.5, 0.4);
        
        \node[sign, anchor=south west] at (0, 0.5) {\myplus};
        \node[sign, anchor=north west] at (0, -0.5) {\myminus};
        
        \draw[biarrow] (1.2, 0) -- (2.0, 0);
        
        \draw[wall] (3.1, -0.6) -- (3.1, 0.6);
        \draw[strand, -o-={0.5}{>}] (2.4, 0.4) -- (3.5, 0.4);
        \draw[strand, -o-={0.5}{<}] (2.4, -0.4) -- (3.5, -0.4);
        \draw[strand] (2.7, -0.4) -- (2.7, 0.4);
        
        \node[sign, anchor=south west] at (3.1, 0.4) {\myminus};
        \node[sign, anchor=north west] at (3.1, -0.4) {\myplus};
    \end{scope}

\end{tikzpicture}
}

\newcommand{\ExceptionalRTdecreaseI}{
\begin{tikzpicture}
    \def\common{
    \coordinate (A) at (0, 0);
    \coordinate (B) at (3, 0);
    \coordinate (C) at (1.5, {1.5*sqrt(3)});
    \draw[edge] (A) -- (B) -- (C) -- cycle;

    \coordinate (W) at (1.6, 1.2);
    \def\r{0.2}
    \draw[edge, fill=white] (W) circle (\r);
    \node[font=\tiny] at (W) {$W'$};

    \draw[edge] ($(W)+(120:\r)$) to[out=120, in=-30] ($(C)!0.25!(A)$);
    \draw[edge] ($(W)+(250:\r)$) to[out=250, in=70] ($(A)!0.15!(B)$);
    \draw[edge] ($(W)+(290:\r)$) to[out=290, in=100] ($(A)!0.35!(B)$);
    \draw[edge] ($(W)+(320:\r)$) to[out=320, in=100] ($(A)!0.75!(B)$);
    
    \draw[edge] ($(W)+(15:\r)$) to[out=0, in=180] ($(C)!0.45!(B)$);
    \draw[edge] ($(W)+(-15:\r)$) to[out=0, in=180] ($(C)!0.7!(B)$);
    }
    \begin{scope}[xshift=0cm]
        \common
        \draw[edge, -o-={0.8}{>}] ($(W)+(150:\r)$) to[out=150, in=-30] ($(C)!0.7!(A)$);
        \draw[edge, very thick, -o-={0.9}{>}] ($(C)!0.5!(A)$) -- ($(A)!0.5!(B)$);
    \end{scope}
    \node at (3.75, 1.0) {\Large $\longrightarrow$};
    \begin{scope}[xshift=4.5cm]
        \common
        \draw[edge, -o-={0.9}{>}, looseness=1.8] ($(W)+(150:\r)$) to[out=150, in=90] ($(A)!0.5!(B)$);
        \draw[edge, very thick, -o-={0.8}{>}, looseness=1.8] ($(C)!0.5!(A)$) to[out=-30, in=-30] ($(C)!0.7!(A)$);
    \end{scope}
    
\end{tikzpicture}
}

\newcommand{\ExceptionalRTdecreaseII}{
\begin{tikzpicture}
    \def\common{
    \coordinate (A) at (0, 0);
    \coordinate (B) at (3, 0);
    \coordinate (C) at (1.5, {1.5*sqrt(3)});
    \draw[edge] (A) -- (B) -- (C) -- cycle;

    \coordinate (W) at (1.2, 0.8);
    \def\r{0.2}
    \draw[edge, fill=white] (W) circle (\r);
    \node[font=\tiny] at (W) {$W'$};

    \draw[edge] ($(W)+(140:\r)$) to[out=140, in=-40] ($(C)!0.55!(A)$);
    \draw[edge] ($(W)+(180:\r)$) to[out=180, in=0] ($(C)!0.75!(A)$);
    
    \draw[edge] ($(W)+(250:\r)$) to[out=250, in=70] ($(A)!0.15!(B)$);
    \draw[edge] ($(W)+(290:\r)$) to[out=290, in=100] ($(A)!0.35!(B)$);
    \draw[edge] ($(W)+(320:\r)$) to[out=320, in=100] ($(A)!0.55!(B)$);
    
    \draw[edge] ($(W)+(15:\r)$) to[out=0, in=180] ($(C)!0.45!(B)$);
    \draw[edge] ($(W)+(-15:\r)$) to[out=0, in=180] ($(C)!0.7!(B)$);
    }
    \begin{scope}[xshift=0cm]
        \common
        \draw[edge, -o-={0.6}{>}] ($(W)+(45:\r)$) to[out=45, in=200] ($(C)!0.2!(B)$);
        \draw[edge, very thick, -o-={0.8}{>}] ($(C)!0.25!(A)$) -- ($(A)!0.75!(B)$);
    \end{scope}
    \node at (3.75, 1.0) {\Large $\longrightarrow$};
    \begin{scope}[xshift=4.5cm]
        \common
        \draw[edge, -o-={0.8}{>}, looseness=1.8] ($(W)+(60:\r)$) to[out=60, in=90] ($(A)!0.75!(B)$);
        \draw[edge, very thick, -o-={0.8}{>}] ($(C)!0.25!(A)$) to[out=-60, in=-120] ($(C)!0.25!(B)$);
    \end{scope}
    
\end{tikzpicture}
}

\newcommand{\coloredDSdiagram}{
\begin{tikzpicture}[
    thin, black,
    point/.style={circle, fill=black, inner sep=1.2pt},
    midarrow/.style={postaction={decorate, decoration={markings, mark=at position 0.6 with {\arrow{stealth}}}}},
    labelpoint/.style={font=\scriptsize, text=black}
]

\def\r{2}
\def\dx{2.5}
\def\dy{2.2}

\begin{scope}[shift={(0,0)}]
    \trianglecoords{0}{0}{\r}{1}
    \draw (A1) -- (B1) -- (C1) -- cycle;
    \sidepoints{A1}{B1}{C1}{1}
    \drawpoints{1}
    
    \node[labelpoint, below] at (P11) {$e_{11}$};
    \node[labelpoint, below] at (P21) {$e_{12}$};
    \node[labelpoint, right] at (P31) {$e_{21}$};
    \node[labelpoint, right] at (P41) {$e_{22}$};
    \node[labelpoint, left] at (P51) {$e_{31}$};
    \node[labelpoint, left] at (P61) {$e_{32}$};
    \node[labelpoint, above] at (M1) {$e$};
\end{scope}

\begin{scope}[shift={(\dx,0)}]
    \trianglecoords{0}{0}{\r}{2}
    \draw (A2) -- (B2) -- (C2) -- cycle;
    \sidepoints{A2}{B2}{C2}{2}
    \fill[color=gray!30] (P22) -- (B2) -- (P32) -- (M2) -- cycle;
    \drawpoints{2}
    
    \node[labelpoint, below] at (P12) {$1$}; \node[labelpoint, below] at (P22) {$2$};
    \node[labelpoint, right] at (P32) {$2$}; \node[labelpoint, right] at (P42) {$1$};
    \node[labelpoint, left] at (P52) {$0$};  \node[labelpoint, left] at (P62) {$0$};
    
    \draw[-o-={0.5}{>}, out=90, in=-150] ($(P12)!0.6!(P22)$) to ($(P32)!0.4!(P42)$); 
    \node[labelpoint, below left] at (M2) {$1$};
\end{scope}

\begin{scope}[shift={(2*\dx,0)}]
    \trianglecoords{0}{0}{\r}{3}
    \draw (A3) -- (B3) -- (C3) -- cycle;
    \sidepoints{A3}{B3}{C3}{3}
    \fill[color=gray!30] (P13) -- (P23) -- (M3) -- (P63) -- cycle;
    \fill[color=gray!30] (P33) -- (P43) -- (P53) -- (M3) -- cycle;
    \drawpoints{3}
    
    \node[labelpoint, below] at (P13) {$2$}; \node[labelpoint, below] at (P23) {$1$};
    \node[labelpoint, right] at (P33) {$1$}; \node[labelpoint, right] at (P43) {$2$};
    \node[labelpoint, left] at (P53) {$0$};  \node[labelpoint, left] at (P63) {$0$};
    
    \draw[-o-={0.5}{<}, out=90, in=-150] ($(P13)!0.6!(P23)$) to  ($(P33)!0.4!(P43)$); 
    \node[labelpoint, below left] at (M3) {$2$};
\end{scope}

\begin{scope}[shift={(0,-\dy)}]
    \trianglecoords{0}{0}{\r}{4}
    \draw (A4) -- (B4) -- (C4) -- cycle;
    \sidepoints{A4}{B4}{C4}{4}
    \fill[color=gray!30] (P44) -- (C4) -- (P54) -- (M4) -- cycle;
    \drawpoints{4}
    
    \node[labelpoint, below] at (P14) {$0$}; \node[labelpoint, below] at (P24) {$0$};
    \node[labelpoint, right] at (P34) {$1$}; \node[labelpoint, right] at (P44) {$2$};
    \node[labelpoint, left] at (P54) {$2$};  \node[labelpoint, left] at (P64) {$1$};
    
    \draw[-o-={0.5}{>}, out=-150, in=-30] ($(P34)!0.6!(P44)$) to ($(P54)!0.4!(P64)$);
    \node[labelpoint, right] at (M4) {$1$};
\end{scope}

\begin{scope}[shift={(\dx,-\dy)}]
    \trianglecoords{0}{0}{\r}{5}
    \sidepoints{A5}{B5}{C5}{5}
    \draw (A5) -- (B5) -- (C5) -- cycle;
    \fill[color=gray!30] (P25) -- (P35) -- (P45) -- (M5) -- cycle;
    \fill[color=gray!30] (P55) -- (P65) -- (P15) -- (M5) -- cycle;
    \drawpoints{5}
    
    \node[labelpoint, below] at (P15) {$0$}; \node[labelpoint, below] at (P25) {$0$};
    \node[labelpoint, right] at (P35) {$2$}; \node[labelpoint, right] at (P45) {$1$};
    \node[labelpoint, left] at (P55) {$1$};  \node[labelpoint, left] at (P65) {$2$};
    
    \draw[-o-={0.5}{<}, out=-150, in=-30] ($(P35)!0.6!(P45)$) to ($(P55)!0.4!(P65)$);
    \node[labelpoint, right] at (M5) {$2$};
\end{scope}

\begin{scope}[shift={(2*\dx,-\dy)}]
    \trianglecoords{0}{0}{\r}{6}
    \draw (A6) -- (B6) -- (C6) -- cycle;
    \sidepoints{A6}{B6}{C6}{6}
    \fill[color=gray!30] (P66) -- (A6) -- (P16) -- (M6) -- cycle;
    \drawpoints{6}
    
    \node[labelpoint, below] at (P16) {$2$}; \node[labelpoint, below] at (P26) {$1$};
    \node[labelpoint, right] at (P36) {$0$}; \node[labelpoint, right] at (P46) {$0$};
    \node[labelpoint, left] at (P56) {$1$};  \node[labelpoint, left] at (P66) {$2$};
    
    \draw[-o-={0.5}{>}, out=-30, in=90] ($(P56)!0.6!(P66)$) to ($(P16)!0.4!(P26)$);
    \node[labelpoint, right] at (M6) {$1$};
\end{scope}

\begin{scope}[shift={(0,-2*\dy)}]
    \trianglecoords{0}{0}{\r}{7}
    \draw (A7) -- (B7) -- (C7) -- cycle;
    \sidepoints{A7}{B7}{C7}{7}
    \fill[color=gray!30] (P17) -- (P27) -- (P37) -- (M7) -- cycle;
    \fill[color=gray!30] (P47) -- (P57) -- (P67) -- (M7) -- cycle;
    \drawpoints{7}
    
    \node[labelpoint, below] at (P17) {$1$}; \node[labelpoint, below] at (P27) {$2$};
    \node[labelpoint, right] at (P37) {$0$}; \node[labelpoint, right] at (P47) {$0$};
    \node[labelpoint, left] at (P57) {$2$};  \node[labelpoint, left] at (P67) {$1$};
    
    \draw[-o-={0.55}{>}, out=90, in=-30] ($(P17)!0.4!(P27)$) to ($(P57)!0.6!(P67)$);
    \node[labelpoint, right] at (M7) {$2$};
\end{scope}

\begin{scope}[shift={(\dx,-2*\dy)}]
    \trianglecoords{0}{0}{\r}{8}
    \draw (A8) -- (B8) -- (C8) -- cycle;
    \sidepoints{A8}{B8}{C8}{8}
    \filldraw[draw=black!50, fill=gray!30] (P28) -- (P38) -- (P48) -- (M8) -- cycle;
    \filldraw[draw=black!50, fill=gray!30] (P48) -- (P58) -- (P68) -- (M8) -- cycle;
    \filldraw[draw=black!50, fill=gray!30] (P68) -- (P18) -- (P28) -- (M8) -- cycle;
    \drawpoints{8}
    
    \node[labelpoint, below] at (P18) {$2$}; \node[labelpoint, below] at (P28) {$1$};
    \node[labelpoint, right] at (P38) {$2$}; \node[labelpoint, right] at (P48) {$1$};
    \node[labelpoint, left] at (P58) {$2$};   \node[labelpoint, left] at (P68) {$1$};
    
    \draw[-o-={0.5}{>}] (M8) -- (MAB8);
    \draw[-o-={0.5}{>}] (M8) -- (MBC8);
    \draw[-o-={0.5}{>}] (M8) -- (MCA8);
    \node[labelpoint, above] at (M8) {$3$};
\end{scope}

\begin{scope}[shift={(2*\dx,-2*\dy)}]
    \trianglecoords{0}{0}{\r}{9}
    \draw (A9) -- (B9) -- (C9) -- cycle;
    \sidepoints{A9}{B9}{C9}{9}
    \filldraw[draw=black!50, fill=gray!30] (P19) -- (P29) -- (P39) -- (M9) -- cycle;
    \filldraw[draw=black!50, fill=gray!30] (P39) -- (P49) -- (P59) -- (M9) -- cycle;
    \filldraw[draw=black!50, fill=gray!30] (P59) -- (P69) -- (P19) -- (M9) -- cycle;
    \drawpoints{9}
    
    \node[labelpoint, below] at (P19) {$1$};  \node[labelpoint, below] at (P29) {$2$};
    \node[labelpoint, right] at (P39) {$1$};  \node[labelpoint, right] at (P49) {$2$};
    \node[labelpoint, left] at (P59) {$1$};  \node[labelpoint, left] at (P69) {$2$};
    
    \draw[-o-={0.5}{>}] (MAB9) -- (M9);
    \draw[-o-={0.5}{>}] (MBC9) -- (M9);
    \draw[-o-={0.5}{>}] (MCA9) -- (M9);
    \node[labelpoint, above] at (M9) {$3$};
\end{scope}

\end{tikzpicture}
}

\newcommand{\Triangulation}
{\begin{tikzpicture}[scale=1.2]
    \begin{scope}
    \coordinate (T1) at (0, 1);  
    \coordinate (T2) at ({sqrt(3)}, 0); 
    \coordinate (T3) at (0, -1); 
    \coordinate (T4) at ({-sqrt(3)}, 0);
    \coordinate (M1) at ({sqrt(3)/3}, 0);
    \coordinate (M2) at ({-sqrt(3)/3}, 0);

    \draw ($(T2)!0.83!(T3)$)[out=120, in=0] to ($(T1)!0.83!(T3)$);
    \draw ($(T2)!0.66!(T3)$)[out=120, in=0] to ($(T1)!0.66!(T3)$);
    \draw ($(T3)!0.83!(T1)$)[out=0, in=-120] to ($(T2)!0.83!(T1)$);
    \draw ($(T3)!0.66!(T1)$)[out=0, in=-120] to ($(T2)!0.66!(T1)$);
    \draw ($(T1)!0.83!(T2)$)[out=-120, in=120] to ($(T3)!0.83!(T2)$);
    \draw ($(T1)!0.66!(T2)$)[out=-120, in=120] to ($(T3)!0.66!(T2)$);
    \draw (M1) to ($(T1)!0.5!(T3)$);
    \draw (M1) to ($(T2)!0.5!(T1)$);
    \draw (M1) to ($(T3)!0.5!(T2)$);

    \draw ($(T4)!0.83!(T3)$)[out=60, in=180] to ($(T1)!0.83!(T3)$);
    \draw ($(T4)!0.66!(T3)$)[out=60, in=180] to ($(T1)!0.66!(T3)$);
    \draw ($(T3)!0.83!(T1)$)[out=180, in=-60] to ($(T4)!0.83!(T1)$);
    \draw ($(T3)!0.66!(T1)$)[out=180, in=-60] to ($(T4)!0.66!(T1)$);
    \draw ($(T1)!0.83!(T4)$)[out=-60, in=60] to ($(T3)!0.83!(T4)$);
    \draw ($(T1)!0.66!(T4)$)[out=-60, in=60] to ($(T3)!0.66!(T4)$);
    \draw (M2) to ($(T1)!0.5!(T3)$);
    \draw (M2) to ($(T4)!0.5!(T1)$);
    \draw (M2) to ($(T3)!0.5!(T4)$);

    \node at (2, 0) {$T_1$};
    \node at (-2, 0) {$T_2$};

    \draw[thick] (T1) -- node[above right] {$e_2$} (T2);
    \draw[thick] (T2) -- node[below right] {$e_1$} (T3);
    \draw[thick] (T3) -- node[below right] {$e_3$} (T1); 
    \draw[thick] (T4) -- node[above left] {$e_2$} (T1);
    \draw[thick] (T3) -- node[below left] {$e_1$}(T4); 
    \end{scope}

    \begin{scope}[xshift=5cm]
    \coordinate (T1) at (0, 1);  
    \coordinate (T2) at ({sqrt(3)}, 0); 
    \coordinate (T3) at (0, -1); 
    \coordinate (T4) at ({-sqrt(3)}, 0);
    \coordinate (M1) at ({sqrt(3)/3}, 0);
    \coordinate (M2) at ({-sqrt(3)/3}, 0);

    \draw ($(T2)!0.83!(T3)$)[out=120, in=0] to ($(T1)!0.83!(T3)$);
    \draw ($(T2)!0.66!(T3)$)[out=120, in=0] to ($(T1)!0.66!(T3)$);
    \draw ($(T3)!0.83!(T1)$)[out=0, in=-120] to ($(T2)!0.83!(T1)$);
    \draw ($(T3)!0.66!(T1)$)[out=0, in=-120] to ($(T2)!0.66!(T1)$);
    \draw ($(T1)!0.83!(T2)$)[out=-120, in=120] to ($(T3)!0.83!(T2)$);
    \draw ($(T1)!0.66!(T2)$)[out=-120, in=120] to ($(T3)!0.66!(T2)$);
    \draw (M1) to ($(T1)!0.5!(T3)$);
    \draw (M1) to ($(T2)!0.5!(T1)$);
    \draw (M1) to ($(T3)!0.5!(T2)$);

    \draw ($(T4)!0.83!(T3)$)[out=60, in=180] to ($(T1)!0.83!(T3)$);
    \draw ($(T4)!0.66!(T3)$)[out=60, in=180] to ($(T1)!0.66!(T3)$);
    \draw ($(T3)!0.83!(T1)$)[out=180, in=-60] to ($(T4)!0.83!(T1)$);
    \draw ($(T3)!0.66!(T1)$)[out=180, in=-60] to ($(T4)!0.66!(T1)$);
    \draw ($(T1)!0.83!(T4)$)[out=-60, in=60] to ($(T3)!0.83!(T4)$);
    \draw ($(T1)!0.66!(T4)$)[out=-60, in=60] to ($(T3)!0.66!(T4)$);
    \draw (M2) to ($(T1)!0.5!(T3)$);
    \draw (M2) to ($(T4)!0.5!(T1)$);
    \draw (M2) to ($(T3)!0.5!(T4)$);

    \node at (2, 0) {$T_1$};
    \node at (-2, 0) {$T_2$};

    \draw[thick] (T1) -- node[above right] {$e_2$} (T2);
    \draw[thick] (T2) -- node[below right] {$e_1$} (T3);
    \draw[thick] (T3) -- node[below right] {$e_3$} (T1); 
    \draw[thick] (T4) -- node[above left] {$e_1$} (T1);
    \draw[thick] (T3) -- node[below left] {$e_2$}(T4); 
    \end{scope}

    \end{tikzpicture}}

\theoremstyle{definition}
		\newtheorem{definition}{Definition}[section]
		
		\newtheorem*{question}{Question}
		
\theoremstyle{plain}
		\newtheorem{lemma}[definition]{Lemma}
		
		\newtheorem{theorem}[definition]{Theorem}

		\newtheorem{conjecture}[definition]{Conjecture}
		\newtheorem{proposition}[definition]{Proposition}
		\newtheorem{corollary}[definition]{Corollary}

        \newtheorem{mainthrm}{Theorem}

\theoremstyle{remark}
		\newtheorem*{remark}{Remark}
		\newtheorem{example}[definition]{Example}
        \newtheorem*{caveat}{Caveat}

\renewcommand{\inf}{\textup{inf}}

\newcommand{\A}{{\mathbb{A}}}
\newcommand{\C}{{\mathbb{C}}}

\newcommand{\Q}{{\mathbb{Q}}}
\newcommand{\R}{{\mathbb{R}}}
\newcommand{\Z}{{\mathbb{Z}}}

\newcommand{\mfrak}{{\mathfrak{m}}}
\newcommand{\pfrak}{{\mathfrak{p}}}

\newcommand{\vb}{{\mathbf{v}}}
\newcommand{\xb}{{\mathbf{x}}}

\newcommand{\Xb}{{\mathbf{X}}}

\newcommand{\Acal}{{\mathcal{A}}}
\newcommand{\Ecal}{{\mathcal{E}}}
\newcommand{\Fcal}{{\mathcal{F}}}

\newcommand{\Ocal}{{\mathcal{O}}}
\newcommand{\Pcal}{{\mathcal{P}}}
\newcommand{\Xcal}{{\mathcal{X}}}

\newcommand{\gr}{\textup{gr}}

\newcommand{\id}{{\textup{id}}}

\newcommand{\sing}{\textup{sing}}

\newcommand{\Xrel}{X^\mathrm{rel}(\Sigma, G, \mathcal{C})}

\DeclareMathOperator{\Ad}{Ad}

\DeclareMathOperator{\codim}{codim}

\DeclareMathOperator{\Def}{Def}

\DeclareMathOperator{\Ext}{Ext}

\DeclareMathOperator{\Frac}{Frac}

\DeclareMathOperator{\GL}{GL}

\DeclareMathOperator{\Hom}{Hom}

\DeclareMathOperator{\sheafhom}{\mathcal{H}\kern -.5pt \emph{om}}

\DeclareMathOperator{\img}{im}

\DeclareMathOperator{\Irr}{Irr}

\DeclareMathOperator{\LT}{LT}

\DeclareMathOperator{\Mat}{Mat}

\DeclareMathOperator{\Proj}{Proj}

\DeclareMathOperator{\rk}{rk}
\DeclareMathOperator{\Rep}{Rep}

\DeclareMathOperator{\Sk}{Sk}
\DeclareMathOperator{\SL}{SL}

\DeclareMathOperator{\Spec}{Spec}

\DeclareMathOperator{\Supp}{Supp}

\DeclareMathOperator{\tr}{tr}

\begin{document}
\title{Compactification of $\operatorname{SL}(3, \mathbb{C})$-Character Varieties \\ of Surfaces Via Skein Algebras}
\author{Seong Youn Kim}

\begin{abstract}
    We investigate the filtration structure of the skein algebra induced by the (quantum) trace map. Utilizing the structure, we prove that relative $\operatorname{SL}(3, \mathbb{C})$ character varieties of punctured surfaces admit log Calabi-Yau compactifications and prove the (weak) geometric P=W conjecture. 
\end{abstract}

\maketitle

\tableofcontents

\section{Introduction}

Let $\Sigma = \Sigma_{g,m}$ be a Riemann surface of genus $g$ with $m > 0$ punctures and having negative Euler characteristic. Let $X^\mathrm{rel}(\Sigma, G, \mathcal{C})$ be the coarse moduli of $G$-local systems on $\Sigma$ with prescribed boundary monodromy $\mathcal{C}$ along its punctures. We will mainly be concerned with the cases where $G = \SL(2, \C)$ and $G = \SL(3, \C)$. We will study a natural compactification of $\Xrel$ and investigate its global geometry.

\subsection*{Main Results}

Let $X$ be a normal projective variety and $\Delta$ be an effective $\Q$-divisor of $X$. We shall say the pair $(X, \Delta)$ is \textit{log Calabi-Yau} if $K_X + \Delta \sim_\Q 0$. There is a folklore conjecture saying that $\Xrel$ always admits a log Calabi-Yau compactification, which is proved by Whang \cite{whang2020global} for $G = \SL(2, \C)$. We prove the conjecture for $G = \SL(3, \C)$.

\begin{mainthrm}\label{theoremA}
    Regardless of the choice of $\mathcal{C}$, $\Xrel$ admits a compactification $Z$ such that the pair $(Z, \Delta = Z \setminus \Xrel)$ is log Calabi-Yau. 
\end{mainthrm}

We investigate into the global geometry of $\Xrel$ derived from the compactification in Theorem \ref{theoremA}. Let $X$ be a quasiprojective variety admitting a compactification $X \cup \partial X$ with a reduced boundary divisor $\partial X$. Denote its dual (boundary intersection) complex by $\mathbb{D}(\partial X)$. 

The \textit{P=W conjecture} formulated by de Cataldo--Hausel--Migliorini \cite{de2012topology} states that the \textit{perverse Leray filtration} of the moduli of $G$-Higgs bundle coincides with the \textit{weight filtration} of the moduli of $G$-local systems via the nonabelian Hodge correspondence. The P=W conjecture for $G = \GL(n, \C)$ and compact $\Sigma$ holds \cite{maulik2024p, hausel2022p, maulik2025perverse}. For noncompact $\Sigma$ with generic local monodromies is also covered in \cite{hausel2022p}.

The \textit{geometric P=W conjecture} established by Katzarkov--Noll--Pandit--Simpson \cite{katzarkov2015harmonic} states that the top dimensional filtration is preserved in the P=W conjecture. Alternatively, it asserts that there exists a homotopy commutative diagram induced by the nonabelian correspondence and Hitchin fibration. Geometric P=W conjecture for $G = \GL(2, \C)$ and $\Sigma = \Sigma_{0,m}$ holds \cite{komyo2015compactifications, simpson2016dual, szabo2023hitchin}. For rank $3$ local system the conjecture is proven for $\Sigma = \Sigma_{0,3}$ \cite{eper2024asymptotic}. We skip the precise definition of the geometric P=W conjecture and refer the reader to \cite{szabo2023hitchin, su2026cell} for details. Instead, as a weaker version of the geometric P=W conjecture, we have the following.

\begin{conjecture}[(Weak) geometric P=W conjecture]\label{GeometricPW}
    $X = \Xrel$ admits a compactification $X \cup \partial X$ such that $\mathbb{D}(\partial X)$ is homotopy equivalent to $S^{d-1}$ where $d = \dim_\C X$.
\end{conjecture}

Conjecture \ref{GeometricPW} is verified in various situations independently of the geometric P=W conjecture:
\begin{enumerate}
    \item $G = \SL(n, \C)$ or $G = \GL(n, \C)$ and $\Sigma$ is a torus \cite{mauri2022geometric},
    \item $G = \SL(2, \C)$ and $\Sigma = \Sigma_{g,m}$ where $m > 0$ \cite{farajzadeh2026compactifications},
    \item $G = \GL(n, \C)$ and $\Sigma = \Sigma_{g,m}$ where $m > 0$ and the boundary monodromies are \textit{generic} in some sense \cite{su2026cell}.
\end{enumerate}

We cover Conjecture \ref{GeometricPW} for $G = \SL(3, \C)$ and $\Sigma = \Sigma_{g,m}$ without any assumptions on the boundary monodromies. 

\begin{mainthrm}\label{theoremB}
    $\mathbb{D}(\Delta)$ arising from Theorem \ref{theoremA} is homotopy equivalent to $S^{d-1}$ where $d = \dim_\C \Xrel$.
\end{mainthrm}

Let us briefly sketch the proof of Theorem \ref{theoremA}. To compute a compactification of an affine scheme $\Spec A$, one may consider a filtration $A = \bigcup_{i=1}^\infty F_i$ and its corresponding Rees algebra $A^\Fcal = \bigoplus_{i \geq 0} F_i t^i$. Then there exists an open immersion $\Spec A \hookrightarrow \Proj A^\Fcal$ with dense image and we have $\Delta = \Proj A^\Fcal \setminus \Spec A = \Proj A^\Fcal / (t)$. Following Watanabe \cite{watanabe1981some}, it suffices to show that $A^\Fcal$ is a normal domain and Gorenstein with canonical module $A^\Fcal(-1)$. To construct such a filtration satisfying above conditions, we will make use of the theory of \textit{skein algebras}, which will be introduced shortly. The filtration using skein algebra theory for $\SL_2$ is used by Farajzadeh-Tehrani--Frohman \cite{farajzadeh2026compactifications} to prove Conjecture \ref{GeometricPW} for $G = \SL(2, \C)$ and punctured surface $\Sigma_{g,m}$ with $m > 0$. We introduce the filtration using $\SL_3$-skein algebra and prove that its corresponding Rees algebra is Gorenstein. Combined with a simple dimension estimate we prove that the Rees algebra is a normal Gorenstein domain. In the end, the construction of the skein-theoretic filtration naturally recovers the construction of Whang \cite{whang2020global}.

To prove Theorem \ref{theoremB}, we note that the skein-theoretic filtration is naturally equipped with a rational polyhedral cone called the \textit{nonelliptic $3$-web monoid}, generalizing the monoid of simple closed curves for $\SL_2$ case. The filtration imposes the boundary divisor $\Delta$ to have a structure of \textit{toric face ring} introduced in \cite{ichim2007toric}. Thus one can identify an irreducible component of $\Delta$ with a face of the rational polyhedral cone. Then the dual intersection complex is isomorphic to the nerve complex of the faces, which is more feasible to compute.

\subsection*{Combinatorics on nonelliptic $3$-webs on surfaces}
Noticing on the nonelliptic $3$-web monoid and the Gorenstein property, one can consider the Hilbert series corresponding to the filtration. Let $\Fcal= (F_i)$ be a filtration of $\C[\Xrel]$ in Theorem \ref{theoremA}. Define the Hilbert series as
\[
h_\Fcal(t) = \sum_{i \geq 0} (\dim F_i / F_{i-1}) t^i.
\]

\begin{mainthrm}\label{theoremC}
     For any filtration considered above, $h_\Fcal(t)$ is rational and we have 
     \[
     h_\Fcal(t) = h_\Fcal(t^{-1}).
     \]
\end{mainthrm}

The reciprocity of the generating function of multicurves on a surface has been known long before \cite{le1985functional}, while the generating function of nonelliptic $3$-webs has never been explored. 

To explain a specific example, we will use \textit{ad hoc} definition for nonelliptic $3$-webs. Fix an ideal triangulation $\lambda$ of a surface. A \textit{nonelliptic $3$-web} is an positive integral weighted oriented graph whose restriction to each ideal triangle looks like Figure \ref{localSLthreeTTconf}. Each restriction consists of two right turns, two left turns and a $3$-valent graph. We may define the \textit{complexity function} $c$ of a nonelliptic web $W$ to be
\[
c(W) = i(\lambda, W) - (\text{total weight of right turns})
\]
where $i(\lambda, W)$ counts the total number of intersections with ideal arcs and $W$ considered with their weights. A nonelliptic $3$-web is called \textit{essential} if it does not contain any weighted peripheral multicurves. Then the Hilbert series can be rewritten as
\[
h_\Fcal(t) = \sum_{i \geq 0} c_i t^i
\]
where $c_i$ denotes the number of essential nonelliptic $3$-webs with complexity $i$.

Now let $\Sigma = \Sigma_{0,3}$. In Section \ref{example} we prove
\[
h_\Fcal(t) = \frac{(1+t^6)(1+t^9)}{(1-t^6)(1-t^9)}.
\]
Similarly, if $\Sigma = \Sigma_{1,1}$ then we have
\[
h_\Fcal(t) = \frac{(1-t^{12})^2(1-t^{18})}{(1-t^2)^3(1-t^4)^3(1-t^6)(1-t^9)^2}.
\]

\subsection*{Skein algebras and quantum topology}

The $\SL_n$-skein algebra $\Sk_q(\Sigma, \SL_n)$ was introduced by Sikora \cite{sikora2001sln} and improved as the \textit{quantum skein algebra} by L\^{e}-Sikora \cite{le2024stated}. It quantizes the moduli space of $\SL_n$-local systems, recovering it by specializing $q=1$ up to nilradical.  Throughout the paper, we will focus on the specialized skein algebra which is commutative. Rather than abandoning the quantum framework, we exploit the existing machinery of quantum topology and specialize back to the commutative world. We emphasize that the specialized skein algebra corresponds to the free group character varieties which will be denoted as $X(\Sigma, G)$, whose boundary monodromies are not predetermined.

The $\SL_n$-skein algebra is generated by $n$-webs, each of which is a special $n$-valent oriented graph, with the skein relation. The skein relation is a geometric manifestation of Procesi's fundamental theorem of invariant theory, describing the relation between the trace functions \cite{procesi1976invariant}. A key strength of the skein algebra lies in the geometric realization. For example, we will prove many algebro-geometric properties of $X^\mathrm{rel}(\Sigma, \SL(3, \C), \mathcal{C})$ by observing how $3$-webs interact with each other. We will also compute the Hilbert series of $X^\mathrm{rel}(\Sigma, \SL(3, \C), \mathcal{C})$ by counting the number of $3$-webs to verify the Gorenstein property of the skein algebra.

The \textit{quantum trace map} is an algebra homomorphism of the quantum skein algebra into a quantum torus, which has long been a crucial object related to the skein algebra. As a specialization of this, the \textit{classical} trace map for $\SL_2$ is observed by Fock-Goncharov \cite{fock2006moduli}. Originating from the link invariant, Bonahon and Wong \cite{bonahon2011quantum} constructed the quantum trace map for $\SL_2$-skein algebra of punctured surfaces. Later, the quantum trace map for $\SL_n$-skein algebra of punctured bordered surface was constructed by L\^{e}-Yu \cite{le2023quantum}. To the extent of the current research, the injectivity of the quantum trace map is only known for $\SL_2$ and $\SL_3$ \cite{kim2020rm, le2023quantum}. As a subalgebra of the (quantum) torus algebra, the algebra structure on $\SL_2$ or $\SL_3$-skein algebra can be viewed as a standard algebra structure of the Laurent polynomials. Even though it is hard to keep track of all of the lower terms of these Laurent polynomials, we can capture the monomial of the highest weight. Inspired by the conjecture of Fock-Goncharov \cite{fock2006moduli}, one can construct the \textit{tropical basis} of the skein algebra using the highest weight monomial. Two crucial constructions of the tropical basis were given by Frohman-Sikora, Douglas--Sun \cite{frohman20223, douglas2024tropical}. In conclusion, the skein algebra has a multiplication structure similar to the standard Laurent polynomial algebras but twisted in their lower terms, which will be introduced as \textit{tropical multiplication formula} in this paper.

Since the Laurent polynomial ring is graded, the filtration given by the (quantum) trace map has a nice property called \textit{semidegree} \cite{mondal2014projective}, saying that the filtration degree acts like a polynomial degree. Thus the filtration of $X(\Sigma, \SL(3, \C))$ induces a compactification whose corresponding boundary divisor is an irreducible toric variety as observed in Farajzadeh-Tehrani--Frohman \cite{farajzadeh2026compactifications} for $\SL(2, \C)$. In this nature we shall say the boundary toric variety a \textit{toric degeneration} of $X(\Sigma, \SL(3, \C))$. Manon \cite{manon2018toric} first constructed such compactification of $\SL(2, \C)$-character variety of free groups without directly using skein theories. By considering the quotient filtration on $\Xrel$, the filtration still has a property called \textit{subdegree} \cite{mondal2014projective}.

\subsection*{Some remarks on log Calabi-Yau compactifications}

The crux of this study lies in the global analysis of $\Xrel$, which is highly likely to be singular. The smooth locus $\Xrel_\mathrm{sm} \subseteq \Xrel$ has a natural symplectic structure originated by Goldman \cite{goldman1984symplectic} and generalized by Biswas--Guruprasad \cite{biswas1993principal} or Guruprasad--Huebschmann--Jeffrey--Weinstein \cite{10.1215/S0012-7094-97-08917-1}. The holomorphic volume form on $\mathcal{M}_B^s(\Sigma, G, \mathcal{C})$ trivializes the canonical bundle over the smooth loci. It is still hard to control, however, its behavior on either singular loci or boundary divisor given by the compactification.

The Fock-Goncharov $\mathcal{X}$-space(resp. $\mathcal{A}$-space) corresponding to the pair $(\Sigma, G)$ is the moduli of framed(resp. decorated) $G$-local systems on $\Sigma$. As cluster varieties, both spaces are and have already been shown to be a log Calabi-Yau variety \cite{grossHackingKeel2015birational}. We stress that the character variety or the skein algebra is related to, but different from these cluster varieties. First of all, we have a natural morphism from $\mathcal{X}$(resp. $\mathcal{A}$) to $\mathcal{M}_B(\Sigma, G)$ by forgetting its framing(resp. decoration). In other words every single monodromy need to be conjugated into a Borel(resp. unipotent) subgroup of $G$.  If $G = \SL(n, \C)$ then the morphism from $\mathcal{X}$-space is indeed surjective, but never flat in general. It would be interesting if one can prove Theorem \ref{theoremA} from the property of $\mathcal{X}$-spaces or $\mathcal{A}$-spaces.

Whang \cite{whang2020global} also proved the finite generation of $X^\mathrm{rel}(\Sigma, \SL(2, \C), \mathcal{C})(\Z)$ by the mapping class group action. As the previous work does, the interest in the existence of the log Calabi--Yau compactification originates from the Diophantine analysis of the relative $\SL(3, \C)$-character variety. Write $X^\mathrm{rel}(\Sigma, \SL(3, \C), \mathcal{C})(\Z) = X^\mathrm{rel}(\Z)$. The study of $X^\mathrm{rel}(\Z)$ leaves much room for further exploration, compared to algebraic tori or abelian varieties. However, like the cluster varieties, as a degeneration of toric varieties one can ask whether $X^\mathrm{rel}(\Z)$ is \textit{finitely generated} in some suitable sense or not. Whang \cite{whang2020nonlinear} proved the finite generation of $X^\mathrm{rel}(\Z)$ for $\SL_2$ by the mapping class group action of the surface. In contrast to the case of $\SL_2$, in his 2015 talk Burger conjectured that the mapping class group orbit of the integral points on the $G$-Hitchin(or maximal) component of the surface is infinite if $G$ is of rank greater than $1$. This is partially verified by Audibert \cite{audibert2022zariski, audibert2025maximal} and the author's upcoming work. The author believes the cluster mutation will be the \textit{suitable} symmetry of $X^\mathrm{rel}(\Z)$ for the higher rank case.

\subsection*{Outline of the paper}

In Section \ref{Section2}, we introduce the definition of $\SL_n$-skein algebras and $\SL_n$-character varieties. We emphasize the relation between skein algebras and character varieties for $\SL_2$ and $\SL_3$ from the \textit{integral} viewpoint. We will introduce the classical trace map of the skein algebra, inducing a tropical basis and filtration of the skein algebra. Finally, we will introduce coordinates to facilitate further computations.

In Section \ref{Section3} and \ref{Section4}, we gather some properties of skein algebras needed for further investigation. Since the filtration of the skein algebra is monoid-valued in its nature, we have to determine the direction inducing a nice filtration. To deal with this we introduce the notion of \textit{tidying-ups} and \textit{defects}. To determine the polyhedral cone corresponding to the value monoid, we compute the Hilbert basis of the monoid. Finally we find some features regarding peripheral skeins. This shows some basic algebro-geometric properties as irreducibility or connectedness.

In Section \ref{Compactification}, we prove two main theorems. We will prove the relative character variety and its compactification is a normal domain and Gorenstein of canonical module shifted by $-1$. For the topology of its boundary, we will introduce the \textit{moment polytope} and compute the link of the peripheral faces.

In Section \ref{example}, we compute some examples concerning the relative $\SL(3, \C)$-character of surfaces with low complexity, $\Sigma_{0,3}$ and $\Sigma_{1,1}$. We compute the Hilbert series and check the functional equation of reciprocity. We also compute the moment polytope and the dual complex. Readers who wish to examine specific examples are encouraged to skip ahead to Section \ref{example}.

\subsection*{Acknowledgement}

The author warmly thanks his advisor Junho Peter Whang for suggesting this problem and his enlightening guidance. The author thanks his advisor Gye-Seon Lee for his continuous support and useful conversations. The author thanks Hyun Kyu Kim and Zhihao Wang for their willingness to answer frequent inquiries and listening to the development of this work. The author thanks Pinaki Mondal for his patient explanations. This work was partially supported by Samsung Science and Technology Foundation under Project Number SSTF-BA2201-03.

\section{Skein algebras and character varieties}\label{Section2}

\subsection{$\SL_n$-skein modules and algebras}

We follow the definition of \cite{le2024stated} in order to use a unified notation for $n \in \{2, 3\}$. Every ring $R$ or an algebra $A$ is commutative with a unity. Every isotopy $\varphi_t$ of a manifold $M$ with boundary $\partial M$ is \textit{proper}, $\varphi_t(\partial M) \subseteq \partial M$ for each $t$.

\begin{definition}\label{defweb}
    Let $M$ be a smooth oriented $3$-manifold with boundary, which is possibly empty. An \textbf{$n$-web} on $M$ is a disjoint union of smoothly embedded oriented circles and a finite directed graph embedded in $M$. The embedded graph $\Gamma$ satisfies the following properties:
    \begin{enumerate}
        \item Every vertex of $\Gamma$ is either an $n$-source or an $n$-sink.
        \item Every edge of $\Gamma$ is a smoothly embedded closed interval.
        \item $\Gamma$ is equipped with a continuous transversal framing.
    \end{enumerate}
    An empty web $\emptyset$ is considered to be isotopic (in $M$) only to itself.
\end{definition}

The existence of framing leads one to have a \textit{blackboard framing}. One may identify a web with a graph locally immersed in $\R^2$ with transversal over(under) crossings.

\begin{definition}\label{defskeinmod}
    Let $R$ be a commutative ring and $n \geq 2$ be an integer. The (classical) \textbf{$\SL_n$-skein module of $M$ over $R$}, denoted by $\Sk_R^n(M)$, is a quotient module of the free module generated by isotopy classes of $n$-webs on $M$ by the local relation 
    \begin{align}
        \cross{p}{>}{>} \quad \;
        &= \qquad \qquad \qquad \cross{n}{>}{>} \label{commutativity} \; ,\ \\
        \kink \; \; \;
        &= \qquad \quad \; (-1)^{n-1} \cdot \horizontaledge{>} \; , \ \\
        \circlediag{<} \; \; \; 
        &= \qquad \quad \; (-1)^{n-1} \cdot n \cdot \emptyset \; ,\ \\
        \sourcesinks \,
        &= (-1)^{\binom n2}\cdot \sum_{\sigma\in S_n}
(-1)^{\ell(\sigma)}  \coupon,
    \end{align}
    where the ellipse enclosing $\sigma_+$ is the minimum crossing diagram representing a permutation $\sigma\in S_n$ and $\ell(\sigma)$ is the parity of $\sigma\in S_n$. 
\end{definition}

In particular for $n=2$, one can use (1)-(4) to deduce the traditional skein relation:
    \begin{align}
        \horizontaledge{>} \quad \; \; &= \quad \; \; \horizontaledge{} \quad \; \; = \quad \; \; \horizontaledge{<} \; , \\
        \largecrossAL \quad \; \; &= \quad \verresolAL \quad + \quad \horresolAL \label{xresolution} \;.
    \end{align}
Similarly for $n=3$, one can deduce the following useful skein relations:
    \begin{align}
    \twogon \quad &= \qquad \qquad 2 \cdot \horizontaledge{>} \label{2gonresolution} \; ,\ \\
    \fourgon \quad \; \; &= \quad \verresol \quad + \quad \horresol \label{4gonresolution} \; ,\ \\
    \largecross \quad \; \; &= \quad \horresolii \quad + \quad \hshape \label{hresolution} \; .
    \end{align}
Here in (\ref{xresolution}) and (\ref{hresolution}) the crossing denotes there is either an over or under crossing.

\begin{remark}[Quantum skein module]
    The skein module in this paper is obtained by $q^{1/2n} = 1$ for the quantized skein module in \cite{le2024stated}. 
\end{remark}

\begin{remark}[Functoriality and base change]
    For any orientation preserving embedding $M \hookrightarrow M'$, we have a natural $R$-module morphism
    \[
    \Sk_R^n(M) \rightarrow \Sk_R^n(M').
    \]
    
    There is a natural $R$-module isomorphism of base change,
    \[
    \Sk_\Z^n(M) \otimes_\Z R \simeq \Sk_R^n (M).
    \]
\end{remark}

\begin{definition}
    A \textbf{surface} is an orientable connected $2$-manifold whose fundamental group is finitely generated. Let $\overline{\Sigma}$ be a compact surface with at least one, and finitely many marked points. A marked point is called a \textbf{puncture}, whose collection is denoted by $P$. A \textbf{punctured surface} $\Sigma$ is $\overline{\Sigma} \setminus P$ with $P \neq \emptyset$. We will write $\Sigma_{g,m}$ for a punctured surface of genus $g$ with $m > 0$ punctures.
\end{definition}
 
\begin{definition}\label{defskeinalg}
    Let $\Sigma$ be a surface. The skein module $\Sk_R^n(\Sigma \times (-1,1))$  has a multiplication structure given by stacking. We will call it a \textbf{$\SL_n$-skein algebra of $\Sigma$ over $R$}. We will write $\Sk_R^n(\Sigma)$ for $\Sk_R^n(\Sigma \times (-1,1))$. 
\end{definition}

\begin{remark}
    In this paper we will not consider any boundary arcs. Let $\Sigma$ be a surface with boundaries. Then $\Sigma^\circ \subseteq \Sigma$ can be considered as a punctured surface and we have
    \[
    \Sk_R^n(\Sigma) \simeq \Sk_R^n(\Sigma^\circ).
    \]
    From now on, we may assume every surface $\Sigma$ is without boundary.
\end{remark}

By the skein relation (\ref{commutativity}), $\Sk_R^n(\Sigma)$ is always a commutative algebra. The functoriality and base change also holds in the $R$-algebra level. \ \\

\subsection{Character varieties}

Let $\pi$ be a finitely generated group. The \textbf{$\SL_n$-representation variety} $\Rep(\pi, \SL_n)$ is an affine scheme defined by the functor
\[
R \mapsto \Hom(\pi, \SL_n(R))
\]
where $R$ is any commutative ring.
The \textit{reduced} $R$-algebra corresponding to $\Rep(\pi, \SL_n)$ will be denoted by $R[\Rep(\pi, \SL_n)]$. Since $\pi$ is finitely generated, $\Rep(\pi, \SL_n)$ is a closed subscheme of $(\SL_n)^m$ where $m$ denotes the number of generators. One can write
\[
R[\Rep(\pi, \SL_n)] = \left( R[x_{ij}^k : 1 \leq i, j \leq n, 1 \leq k \leq m] / I \right) / \sqrt{0}
\]
where $\pi$ is generated by $m$ elements and $x_{ij}^k$ denotes a matrix variable. Hence the natural model over $\Z$ is noetherian.

\begin{caveat}
The classical invariant theory concerns the ring of invariants as a subring of a possibly nonreduced ring 
\[
R[x_{ij}^k : 1 \leq i, j \leq n, 1 \leq k \leq m] / I
\]
with the same notation $R[\Rep(\pi, \SL_n)]$. In this nonreduced setting we can say $\Rep(\pi, \SL_n)$ is indeed an affine scheme with a \textbf{natural model over $\Z$}, that is, we have a natural isomorphism
\[
\Z[\Rep(\pi, \SL_n)] \otimes_\Z R \simeq R[\Rep(\pi, \SL_n)].
\]
If $R$ is reduced and $\pi$ is a free group, we have $R[\Rep(\pi, SL_n)]$ is also reduced and there is no ambiguity in the notation. 
\end{caveat}

\begin{definition}
    The \textbf{$\SL_n$-character variety of $\pi$ over $R$} is a GIT quotient scheme over $R$ by the conjugation action of $\SL_n$,
    \[
    \Rep(\pi, \SL_n)_R \sslash \SL_n = \Spec R[\Rep(\pi, \SL_n)]^{\SL_n}
    \]
    which will be denoted by $X(\pi, \SL_n)_R$. We will write $X(\pi, \SL_n) = X(\pi, \SL_n)_\Z$.
\end{definition}

For each $\gamma \in \pi$, let $\tr_\gamma$ be the regular function of $\Rep(\pi, \SL_n)$ defined by $\tr_\gamma(\rho) = \tr \rho(\gamma)$. \ \\

The following proposition is classical. See \cite[Proposition 2.1, Theorem 3.7]{sikora2001sln} and \cite[11.2]{le2024stated}.
\begin{proposition}\label{skeincharvariso}
    Let $\Sigma$ be a surface. Let $k$ be a field of characteristic $0$. Then
    \[
    \Sk_k^n(\Sigma) / \sqrt{0} \simeq k[\Rep(\pi, \SL_n)]^{\SL_n}.
    \]
\end{proposition}

We now describe this isomorphism. Let $L_\gamma$ be the image of the smooth map $i : S^1 \rightarrow \Sigma \times (-1,1)$ representing the conjugacy class of an element $\gamma \in \pi_1 (\Sigma \times (-1,1)) = \pi_1 \Sigma$. The fact that the skein algebra is generated by $L_\gamma$ can also be checked via the skein relations: one can resolve all sinks and sources. By \cite[Theorem 3.7]{sikora2001sln}, the isomorphism sends
\[
L_\gamma \mapsto \tr_\gamma.
\]
\label{fromskeintotrace}

\begin{remark}[Skein algebra for topological space]
    Let $\pi$ be a finitely generated group and $\Xb$ be a topological space with $\pi_1 \Xb = \pi$. Sikora \cite{sikora2001sln} defined $\Sk_R^n(\Xb)$ (in the notation $\A_n(\Xb)$), the \textbf{$\SL_n$-skein algebra of $\Xb$} with
    \[
    \Sk_k^n(\Xb) / \sqrt{0} \simeq k[\Rep(\pi, \SL_n)]^{\SL_n}
    \]
    for any characteristic $0$ field $k$.
    It can be shown that $\Sk_R^n(\Xb)$ depends only on $\pi$. The skein algebra of $\Sigma$ arises as $\Xb = \Sigma$.
\end{remark}

\begin{corollary}
    Let $\Sigma$ be a surface and $k$ be a field of characteristic $0$. Then $\Sk_\Z^n(\Sigma) / \sqrt{0}$ is a natural model over $\Z$ of $X(\pi, \SL_n)_k$.
\end{corollary}

\subsection{Classical trace maps}

\begin{definition}
    A surface $\Sigma$ is called \textbf{hyperbolic} if $\chi(\Sigma) < 0$. Let $\Sigma$ be a punctured hyperbolic surface.
    \begin{enumerate}
        \item An \textbf{ideal arc} is a proper embedding $\ell : [0,1] \rightarrow \overline{\Sigma}$ with $\ell(0), \ell(1) \in P$. 
        \item An \textbf{ideal triangulation} $\lambda$ is a maximal collection of distinct isotopy classes of pairwise disjoint ideal arcs.
        \item An \textbf{ideal triangle} is a connected component of the complementary region of $\lambda$.
        \item $\lambda$ is called \textbf{without self-folded triangles} if for all $p \in P$ there exist at least two ideal arcs having $p$ as an endpoint.
    \end{enumerate}     
\end{definition}

\begin{definition}
    Let $(\Sigma, \lambda)$ be a punctured hyperbolic surface with an ideal triangulation.
    The \textbf{$n$-triangulation} of $(\Sigma, \lambda)$ is a subdivision of each ideal triangle $T$, where $T$ is identified with
        \[
        \{(x, y, z) \in \R_{\geq 0}^3 : x+y+z = n \}.
        \]
    The subdivision restricted to each ideal triangle $T$ is given by the vertex set
        \[
        V_\lambda^n|_T = \{(x, y, z) \in \Z_{\geq 0}^3 : x+y+z = n \} \setminus \{(n, 0, 0), (0, n, 0), (0, 0, n) \} .
        \]
\end{definition}

Note that $|V_\lambda^n|$ only depends on $\chi(\Sigma)$. By Gauss-Bonnet the number of ideal triangles is 
\[
-2 \chi(\Sigma).
\]
The number of vertices lying on an edge is $n-1$ and the number of vertices lying in the interior of an ideal triangle is $\frac{(n-1)(n-2)}{2}$. Thus
\[
|V_\lambda^n| = -2 \chi(\Sigma) \cdot \frac 32 (n-1) -2 \chi(\Sigma) \frac{(n-1)(n-2)}{2} = (1 - n^2) \chi(\Sigma).
\]
Also note that the number equals the dimension of the $\SL_n$-character variety of $\pi = \pi_1 \Sigma = F_{1 - \chi(\Sigma)}$. Here $F_d$ denotes the free group of rank $d$.

\begin{definition}
    Let $R$ be a commutative ring. Let $(\Sigma, \lambda)$ be a punctured hyperbolic surface with an ideal triangulation $\lambda$. The (classical) \textbf{$\SL_n$-Fock-Goncharov torus algebra over $R$} is the algebra of Laurent polynomials 
    \[
    \Xcal_R(\Sigma, \lambda, \SL_n) :=  R[x_v^{\pm 1} : v \in V_\lambda^n] = \bigoplus_{\mathbf{v} \in \Z^N} R \cdot x^\mathbf{v}
    \]
    where $\Z^N = \Z^{V_\lambda^n}$ is the weight lattice.
\end{definition}

\begin{theorem}[{\cite{fock2006moduli,kim2020rm,le2023quantum}}]\label{injtracemap}
    Let $(\Sigma, \lambda)$ be a punctured surface with an ideal triangulation. There exists an algebra homomorphism
    \[
    \tr_\lambda : \Sk^n_R(\Sigma) \rightarrow \Xcal_R(\Sigma, \lambda, \SL_n).
    \]
    Moreover, if $n=2$ or $3$ then $\tr_\lambda$ is injective.
\end{theorem}

\begin{corollary}\label{skeindomain}
    Let $\Sigma$ be a punctured surface and $R$ be a domain. Let $n \in \{2, 3\}$. Then $\Sk_R^n(\Sigma)$ is a domain.
\end{corollary}

The domain property will also be observed in Section \ref{Section4}.

We will later prove that if $R$ is noetherian then $\Sk_R^2(\Sigma)$ and $\Sk_R^3(\Sigma)$ are also noetherian. Hence $\Sk_\Z^2(\Sigma)$ and $\Sk_\Z^3(\Sigma)$ are \textit{noetherian} natural models over $\Z$. \ \\

\subsection{Invariant theory over $\Z$}\label{invariantz}
    
     By De Concini-Procesi \cite[Section 3]{de1982characteristic} or Hashimoto \cite[Section 4]{hashimoto2005another}, for any commutative ring $R$ the base change holds
    \[
    R[\Mat_n^m]^{\SL_n} \simeq R \otimes_\Z \Z[\Mat_n^m]^{\SL_n}.
    \]
    By Seshadri \cite[Theorem 2]{seshadri1977geometric}, for any finitely generated group $\pi$ by $m$ elements, we have  $\Z[\Rep(\pi, \SL_n)]^{\SL_n}$ noetherian. There exists a subring generated by the trace functions
    \[
    \Z[\tr_{\gamma}]_{\gamma \in \pi} \subseteq \Z[\Rep(\pi, \SL_n)]^{\SL_n}.
    \]
    We will write $\Ocal_X = \Z[\Rep(\pi, \SL_n)]^{\SL_n}$ and $\Ocal_X^{\tr} = \Z [\tr_\gamma]_{\gamma \in \pi}$.

\begin{proposition}
        Let $\pi$ be a finitely generated free group and $n = 2$ or $n=3$. Then
        \[
        \Ocal_X = \Ocal_X^{\tr}.
        \]
\end{proposition}

\begin{proof}
    This is based on the content of Quast \cite[Theorem 4.6]{quast2026deformations}. The invariant ring
    \[
    \mathbb Z[(\SL_n)^m]^{\SL_n}
    \]
    is generated by functions of the form
    \[
    \sigma_i\!\left(X^{(j_1)}\cdots X^{(j_s)}\right),
    \quad 1\le i\le n,
    \]
    where $X^j$ denotes the matrix variable. Here $\sigma_i(A)$ denotes the $i$-th coefficient of the characteristic polynomial.

    Let $n=2$ and fix a word $x_{j_1} \cdots x_{j_s} \in F_m$. Then $\sigma_1(X^{(j_1)} \cdots X^{(j_s)}) = \tr_w$. If $n=3$, similarly $\sigma_1(X^{(j_1)} \cdots X^{(j_s)}) = \tr_w$ and $\sigma_2(X^{(j_1)} \cdots X^{(j_s)}) = \tr_{w^{-1}}$.

    Hence
    \[
    \mathbb Z[(\SL_n)^m]^{\SL_n}
    =
    \mathbb Z[\operatorname{tr}_w : w\in F_m].
    \]
    holds for $n=2$ and $n=3$.
\end{proof}

\begin{remark}
    Since the group scheme $\SL_n$ is not linearly reductive over $\Z$, the result does not generalize to arbitrary finitely generated group $\pi$.
\end{remark}

\begin{remark}
    In Section \ref{fromskeintotrace} we have considered the isomorphism
    \[
    L_\gamma \mapsto \tr_\gamma.
    \]
    For any commutative ring $R$, the algebra morphism
    \[
    \Sk_R^n(\Sigma) / \sqrt{0} \rightarrow R[\Rep(\pi, \SL_n)]^{\SL_n}
    \]
    is well-defined. The image of the above morphism is an $R$-subalgebra generated by $\tr_\gamma$.
\end{remark}

\begin{proposition}
    Let $\Sigma$ be a punctured surface and $\pi = \pi_1 \Sigma$. Let $n = 2$ or $n = 3$. There exists an isomorphism of rings
    \[
    \Sk_\Z^n(\Sigma) \simeq \Z[\Rep(\pi, \SL_n)]^{\SL_n}.
    \]
\end{proposition}

\begin{proof}
    Consider the morphism
    \[
    L_\gamma \mapsto \tr_\gamma
    \]
    defined above.
    Since $\Ocal_X = \Ocal_X^{\tr}$, the morphism is surjective.

    On the other hand, $\Sk_R^n(\Sigma)$ is a free $R$-module. An explicit basis will be described in the following subsection. After tensoring with $\Q$, to see the morphism is injective it suffices to check the following diagram commutes:
    \[
    \begin{tikzcd}
    \Sk_\Z^n(\Sigma) \arrow[d, hook] \arrow[r] & {\Z[\Rep(\pi, \SL_n)]^{\SL_n}} \arrow[d] \\
    \Sk_\Q^n(\Sigma) \arrow[r, "\simeq"]       & {\Q[\Rep(\pi, \SL_n)]^{\SL_n}}       
    \end{tikzcd}
    \]
    By De Concini--Procesi one has
    \begin{align*}
        \Q[\Rep(\pi, \SL_n)]^{\SL_n} \simeq \Q[\SL_n^m]^{\SL_n} &\simeq \Q[\Mat_n^m]^{\SL_n} / (\det\nolimits_w - 1) \ \\
        &\simeq \Q \otimes_\Z \Z[\Mat_n^m]^{\SL_n} / (\det\nolimits_w - 1) \simeq \Q \otimes_\Z \Z[\Rep(\pi, \SL_n)]^{\SL_n}
    \end{align*}
    and $\Q \otimes_\Z \Sk_\Z^n(\Sigma) \simeq \Sk_\Q^n(\Sigma)$ is given by the base change.

    The first vertical arrow is injective since $\Sk_\Z^n(\Sigma)$ is torsion-free and hence flat. The second horizontal arrow is given by Proposition \ref{skeincharvariso}.

\end{proof}

\begin{example}
    Let $\pi$ be a finitely generated free group and $n=2$. Then
    \[
    \Sk_\Z^2(\Sigma) \simeq \Z[x_\gamma : \gamma \in \pi] / (x_\id - 2, x_\gamma x_\delta - x_{\gamma\delta} - x_{\gamma\delta^{-1}} )_{\gamma, \delta \in \pi} \simeq \Z[\Rep(\pi, \SL_2)]^{\SL_2}.
    \]
\end{example}

\begin{example}\label{F2SL3}
In the work of Lawton \cite{lawton2014sl3ccharactervarietiesrp2structurestrinion} there exists another natural model over $\Z$ for $X(F_2, \SL_3)$ where $F_2 = \langle \alpha, \beta \rangle$ is a free group of rank $2$. The coordinate ring over $\Z$ is written as
    \[
    A := \Z[x_i]_{i=1}^9 / (f)
    \]
    where $f \in \Z[x_i]_{i=1}^9$ and $x_i$ corresponds to the trace functions of the group elements
    \[
    \{ \alpha, \alpha^{-1}, \beta, \beta^{-1}, \alpha \beta, (\alpha \beta)^{-1}, \alpha \beta^{-1}, \beta \alpha^{-1}, [\alpha, \beta] \}.
    \]
    Under the correspondence $L_\gamma \leftrightarrow \tr_\gamma$ one has $\Sk_\Z^3(\Sigma_{0,3}) \simeq A \simeq \Sk_\Z^3(\Sigma_{1,1})$. 
\end{example}
    
\begin{remark}[$\Z$-points and natural model]
    Let $R$ be a commutative ring and let $V_R$ be a scheme over $R$ with a natural model $V_\Z$ over $\Z$ satisfying
    \[
    (V_\Z)_R \simeq V_R.
    \]
    If $V_\Z = \Spec A$ is an affine scheme, the $\Z$-points are the ring homomorphisms $A \rightarrow \Z$, in which sense $V_\Z$ is also called a \textit{Diophantine system}. There is an induced map on functors of points
    \[
    V_\Z(\Z) \rightarrow V_\Z(R) \simeq V_R(R)
    \]
    by postcomposing with $\Z \rightarrow R$. The bijection is given by the universal property of fiber product. Moreover, if $R$ is of characteristic $0$ then $\Z \rightarrow R$ is injective and
    \[
    V_\Z(\Z) \hookrightarrow V_\Z(R) \simeq V_R(R).
    \]
\end{remark}

\begin{remark}[$\Z$-points of the character variety over $\C$]
    Let $V = X(\pi, \SL_n)_\C$. Then it is known that there is a one-to-one correspondence between
    \begin{enumerate}
        \item $V_\Z(\C) \simeq V(\C)$;
        \item Conjugacy classes of semisimple representations $\rho : \pi \rightarrow \SL_n(\C)$.
    \end{enumerate}

Now let $n \in \{2, 3\}$. It follows from the discussion above that, $V_\Z(\Z)$ is identified with the set of conjugacy classes of semisimple representations $\rho : \pi \rightarrow \SL_n(\C)$ satisfying the following equivalent conditions:
    \begin{enumerate}
        \item For any $\gamma \in \pi$, $\tr(\rho(\gamma)) \in \Z$.
        \item There exists a finite subset $S \subseteq \pi$ (independent of $\rho$) such that for any $\gamma \in S$, $\tr(\rho(\gamma)) \in \Z$.
    \end{enumerate}
Here the finite subset $S$ can be explicitly written by Procesi \cite{procesi1976invariant} for all $n$. The optimal $S$ is given by Goldman \cite{goldman2009trace} for $n=2$ and Lawton \cite{lawton2008minimal} for $n=3$.
\end{remark}

\subsection{Bangle basis}

We now recall a $R$-basis of $\Sk_R^n(\Sigma)$ where $n \in \{2, 3\}$.

\begin{definition}
    Let $\Sigma$ be a surface. A \textbf{multicurve} is a finite disjoint union of unoriented noncontractible simple closed curves on $\Sigma$.
\end{definition}

\begin{theorem}[Sikora--Westbury, {\cite[Section 4]{sikora2007confluence}}]
    Let $\Sigma$ be a surface and $R$ be a commutative ring. Then $\Sk_R^2(\Sigma)$ is a free $R$-module with the \textbf{multicurve basis}, corresponding to the regular map
    \[
    \tr_\Gamma = \prod_{\gamma \subseteq \Gamma} \tr_\gamma
    \]
    where $\Gamma$ is a multicurve and $\gamma$ is a simple closed curve component of $\Gamma$.
\end{theorem}

The set of isotopy classes of multicurves on $\Sigma$ is denoted by $\Lambda(\Sigma, \SL_2)$.

\begin{definition}
    Let $\Sigma$ be a surface. A $3$-web $W$ (projected) on $\Sigma$ is called \textbf{nonelliptic} if it has neither self-intersection nor contractible loops and each complementary region of $W$ is neither a $2$-gon nor a $4$-gon. 
\end{definition}

\begin{definition}
    Let $\Sigma$ be a surface. Two nonelliptic $3$-webs $W_1, W_2$ on $\Sigma$ are called \textbf{isotopic} if they are isotopic in $\Sigma \times (-1,1)$. The set of isotopy classes of nonelliptic webs on $\Sigma$ is denoted by $\Lambda(\Sigma, \SL_3)$.
\end{definition}

\begin{caveat}\label{parallelmove}
    Let $W$ be a nonelliptic $3$-web consisting of two puncture curves with two opposite orientations. Let $W'$ be a nonelliptic $3$-web obtained by inverting all orientation of curves. Then $W, W'$ are not isotopic in $\Sigma$ but isotopic in $\Sigma \times (-1,1)$.
\end{caveat}

\begin{theorem}[Sikora--Westbury, {\cite[Section 4]{sikora2007confluence}}]
    Let $\Sigma$ be a surface and $R$ be a commutative ring. Then $\Sk_R^3(\Sigma)$ is a free $R$-module with the \textbf{nonelliptic web basis}.
\end{theorem}

By the correspondence $\Sk_\Z^3(\Sigma) \simeq \Ocal_X$, there is a regular map in $\Ocal_X$ corresponding to a nonelliptic web $W$, which will be denoted by $\tr_W$. As for the multicurves, if two nonelliptic webs $W, W'$ are disjoint in $\Sigma$, then one has
\[
\tr_W \cdot \tr_{W'} = \tr_{W \cup W'}.
\]

\begin{definition}
    The multicurve basis and the nonelliptic web basis are called the \textbf{bangle basis} of $\Sk_R^n(\Sigma)$ for $n=2$ and $n=3$.
\end{definition}

\begin{example}\label{LawtonSL3trinion}
    Let $\Sigma = \Sigma_{0,3}$ with $\pi = \pi_1 \Sigma = \langle \alpha, \beta \rangle$, illustrated in Figure \ref{pairofpants} (a). Here the conjugacy class of the boundary curves are given as $\alpha, \beta, (\alpha\beta)^{-1}$ in clockwise order.

    A \textbf{theta web} $W_{\theta}$ is a nonelliptic $3$-web consisting of a single sink and source, which is illustrated in Figure \ref{pairofpants} (b). By the skein relation in Definition \ref{defskeinmod} one can deduce
    \[
    \tr_{W_\theta} = \tr_{\alpha \beta^{-1}} - \tr_\alpha \tr_{\beta^{-1}} = \tr_{\beta \alpha \beta} - \tr_\beta \tr_{\alpha \beta} = \tr_{\beta^{-1} \alpha^{-2}} - \tr_{\beta^{-1} \alpha^{-1}} \tr_{\alpha^{-1}}.
    \]
    Here the three different formulas are obtained by the different choices of the ellipse $\sigma_+$ containing the sink and the source. The formula is also observed in \cite[Lemma 4.9]{parker2012traces}.

    A commutator web $W_\mathrm{comm}$ is a nonelliptic web consisting of $3$ sources and $3$ sinks, which is illustrated in Figure \ref{pairofpants} (d). By the skein relation in Definition \ref{defskeinmod} one can deduce
    \[
    \tr_{W_\mathrm{comm}} = \tr_{[\alpha, \beta]} - \tr_\alpha \tr_{\alpha^{-1}} - \tr_{\beta} \tr_{\beta^{-1}} - \tr_{\alpha \beta} \tr_{(\alpha \beta)^{-1}}.
    \]

    \begin{figure}
        \centering
        \includegraphics[width=0.15\linewidth]{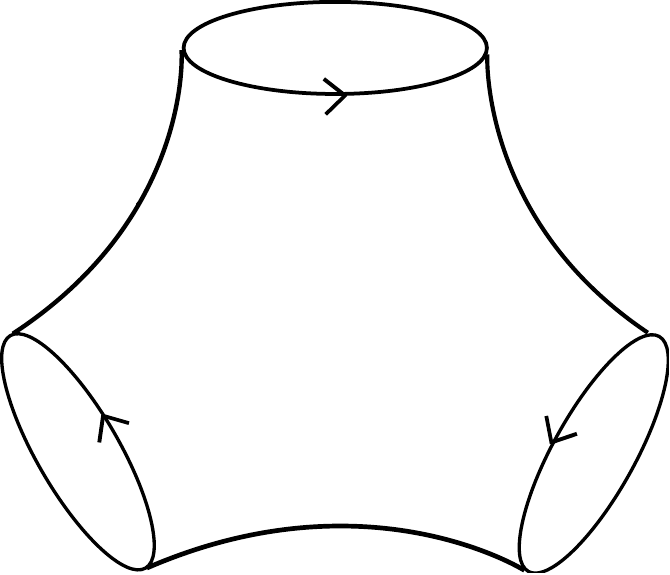}
        \qquad
        \includegraphics[width=0.15\linewidth]{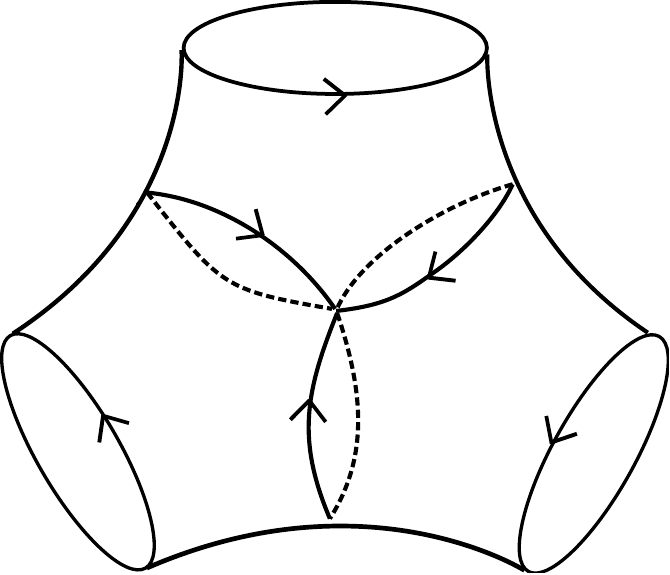}
        \qquad
        \includegraphics[width=0.15\linewidth]{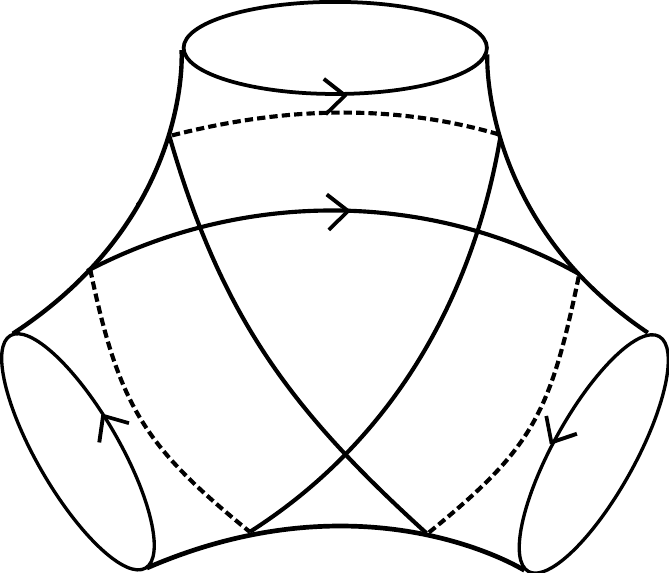}
        \qquad
        \includegraphics[width=0.15\linewidth]{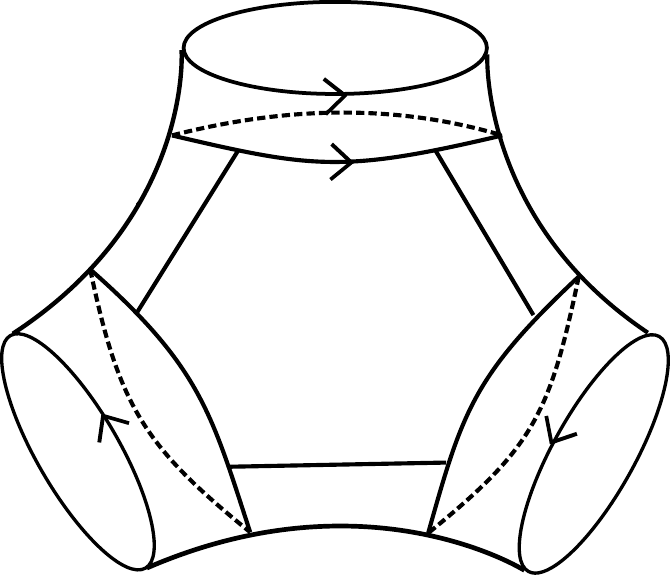}
        \caption{(a) A pair of pants. (b) A theta web. (c) A commutator curve. (d) A commutator web.}
        \label{pairofpants}
    \end{figure}
\end{example}

Comparing the above example with the result of Lawton \cite{lawton2014sl3ccharactervarietiesrp2structurestrinion} in Example \ref{F2SL3}, one has
\[
\Sk_\Z^3(\Sigma) \simeq \Z[\tr_\alpha, \tr_{\alpha^{-1}}, \tr_{\beta}, \tr_{\beta^{-1}}, \tr_{\alpha\beta}, \tr_{(\alpha \beta)^{-1}}, \tr_{W_\theta}, \tr_{W'_\theta}, \tr_{W_\mathrm{comm}}] \simeq \Z[x_i]_{i=1}^9 / (f)
\]
where $W'_\theta$ denotes the nonelliptic $3$-web obtained by inverting all orientations of $W_\theta$.

\subsection{Train tracks, intersection coordinates and Douglas--Sun coordinates}

\begin{definition}
    Let $\Sigma$ be a surface and $\Gamma$ be a multicurve on $\Sigma$. We will also call an isotopy class of a multicurve a (nonnegative) \textbf{global $\SL_2$-train track} on $\Sigma$. 
\end{definition}

The term \textit{train track} is borrowed from \cite{kim2024unicity}. This differs from the conventional \textit{train tracks} on $\Sigma$, for example as in \cite{penner1992combinatorics}.

\begin{definition}
    Let $T$ be an (ideal) triangle, an oriented genus zero surface with a single boundary with $3$ markings. A \textbf{corner arc} is a properly embedded closed interval with its orientation removed, whose endpoints are not  markings. A \textbf{local $\SL_2$-train track} is an isotopy class of a finite collection of disjoint corner arcs.
\end{definition}

We will write a triangle $T = (e_1, e_2, e_3)$ by listing its edges in counterclockwise order with respect to the orientation of $T$.

\begin{definition}
    Let $T = (e_1, e_2, f)$ and $T' = (f, e_3, e_4)$ be two triangles and let $\Gamma, \Gamma'$ be a local $\SL_2$-train track on $T, T'$. $\Gamma$ and $\Gamma'$ are called \textbf{compatible} via gluing $f$ if their number of corner arcs meeting with the edge $f$ are equal.
\end{definition}

    Let $(\Sigma, \lambda)$ be a punctured surface with an ideal triangulation $\lambda$. There is a one-to-one correspondence between
    \begin{gather*}
    \{ \text{Global $\SL_2$-train tracks on $\Sigma$} \} \ \\
    \updownarrow  \ \\
    \{ \text{Compatible collections of local $\SL_2$-train tracks on ideal triangles of $(\Sigma, \lambda)$} \}.
    \end{gather*}
    
    The correspondence also gives an embedding $\Lambda(\Sigma, \SL_2) \hookrightarrow \Z_{\geq 0}^{-3 \chi(\Sigma)}$. Moreover, the target has a natural additive monoid structure and one can easily check $\Lambda(\Sigma, \SL_2)$ forms a submonoid of $\Z_{\geq 0}^{-3 \chi(\Sigma)}$. This gives a natural monoid structure on $\Lambda(\Sigma, \SL_2)$. For any multicurves $\Gamma, \Gamma'$, denote the monoid addition by $\Gamma + \Gamma'$. Note that 
    \[
    \Gamma + \Gamma' \leftrightarrow \Gamma \cup \Gamma ' \leftrightarrow \Gamma \cdot \Gamma'
    \]
    if $\Gamma$ and $\Gamma'$ are disjoint. Here $\Gamma + \Gamma'$ lies in the monoid $\Lambda(\Sigma, \SL_2)$ and $\Gamma \cdot \Gamma'$ is a multiplication in $\Sk_R^2(\Sigma)$.

\begin{figure}
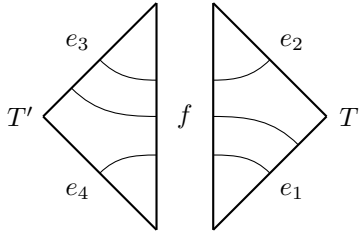

    \centering
    \compSLtwoTT
    \caption{Compatible local $\SL_2$-train tracks.}\label{compSL2TT}
\end{figure}

\begin{definition}
    Let $(\Sigma, \lambda)$ be a punctured surface with an ideal triangulation. Let $\Gamma$ be a multicurve on $\Sigma$. We may assume $\Gamma$ is transverse and has minimal intersections with $\lambda$ by taking an isotopy. Recall $V_\lambda^2$ is the $2$-triangulation of $(\Sigma, \lambda)$. There is a natural correspondence between $V_\lambda^2$ and the ideal arcs of $\lambda$. The \textbf{intersection coordinate} of $\Gamma$ with respect to $\lambda$ is
    \[
    (i(\Gamma, e))_{e \in V_\lambda^2} \in \Z_{\geq 0}^{V_\lambda^2}
    \]
    where $i(-, -)$ denotes the geometric intersection number.
\end{definition}

    The natural monoid morphism
    \[
    \Lambda(\Sigma, \SL_2) \rightarrow \Z_{\geq 0}^{V_\lambda^2}
    \]
    is injective with its image being characterized by the following property called the \textbf{triangle inequality} (or the \textbf{rhombus inequality})
    \begin{quote}\label{IntersectionCoordinate}
        For any ideal triangle $T = (e_1, e_2, e_3)$, one has
        \begin{align*}
            &i(\Gamma, e_2) + i(\Gamma, e_3) - i(\Gamma, e_1) \in 2\Z_{\geq 0}, \ \\
            &i(\Gamma, e_3) + i(\Gamma, e_1) - i(\Gamma, e_2) \in 2\Z_{\geq 0}, \ \\
            &i(\Gamma, e_1) + i(\Gamma, e_2) - i(\Gamma, e_3) \in 2\Z_{\geq 0}.
        \end{align*}
    \end{quote}

This follows by observing, for example, the term
\[
\frac 12 [i(\Gamma, e_1) + i(\Gamma, e_2) - i(\Gamma, e_3)]
\]
equals the number of corner arcs at the opposite angle to $e_3$. In particular, $\Lambda(\Sigma, \SL_2)$ has the structure of a positive rational polyhedral cone. The coordinate depends on the choice of $\lambda$. \ \\

The term \textit{rhombus} is illustrated in Figure \ref{rhombusSL2}. The intersection coordinate is an assignment of nonnegative integers to $V_\lambda^2$. One can also consider it as a function
\[
V_\lambda^2 \cup P \rightarrow \Z_{\geq 0}
\]
by assigning $0$ to each puncture. 

Each ideal triangle $T$ has $3$ rhombi consisting of $2$ subtriangles. Then the triangle inequality can be rewritten as
\begin{align}
\sum \text{(value at the obtuse angle)} - \sum \text{(value at the acute angle)} \in 2\Z_{\geq 0}. \label{rhombineq}
\end{align}

\begin{figure}
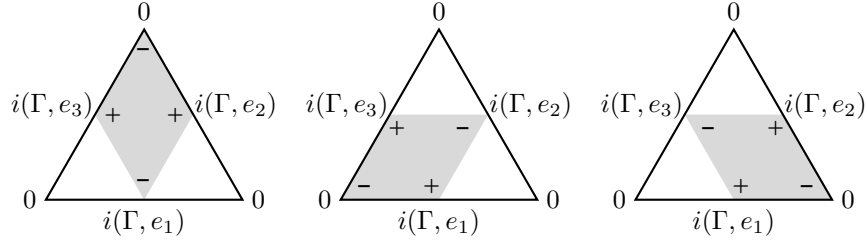

\centering
\rhombSLtwoTT
\caption{Rhombus inequalities for $\SL_2$-train tracks.}\label{rhombusSL2}
\end{figure}

\begin{definition}
    Let $(\Sigma, \lambda)$ be a punctured surface with an ideal triangulation without self-folded triangles. The \textbf{split ideal triangulation}, which we also denote by $\lambda$ by abuse of notation, is obtained by replacing each ideal arc by its tubular neighbourhood. The tubular neighbourhood of each ideal arc is called \textbf{ideal biangle}. Again each connected component of the complementary region of ideal biangles is called \textbf{ideal triangle}.
\end{definition}

We will again write a biangle $B = (e_1, e_1')$ by listing its edges in counterclockwise order with respect to the orientation of $B$.

\begin{definition}
    Let $T = (e_1, e_2, e_3)$ be an (ideal) triangle. Use the cyclic notation for subscripts
    \[
    e_i = e_{i+3}
    \]
    for any integer $i$.
    A \textbf{right(resp. left) turn (corner arc)} is a properly embedded arc $\gamma : [0,1] \hookrightarrow T$ with endpoints $\gamma(0) \in e_i$, $\gamma(1) \in e_{i+1}$(resp. $e_{i-1}$) for $i \in \{1, 2, 3\}$.

    A \textbf{honeycomb of degree $d$} ($d \in \Z_{>0}$) is the following properly embedded oriented graph in Figure \ref{honeycomb}, whose vertices are either $3$-sinks, $3$-sources or one-ended vertices lying in the boundary. We will fix a \textit{positive} orientation of the honeycomb to be outward. A \textbf{honeycomb of degree $-d$} is obtained by changing every orientations of the arc of the honeycomb of degree $d > 0$. An empty graph is a honeycomb of degree $0$.
    \begin{figure}
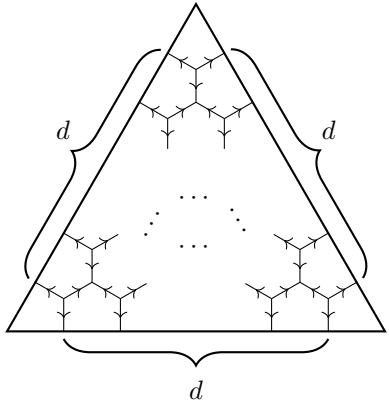

        \centering
        \nhoneycomb
        \caption{Honeycomb of degree $d > 0$.}\label{honeycomb}
    \end{figure}
    
    A \textbf{local $\SL_3$-train track} is an isotopy class (in $T \times (-1,1)$) of a finite collection of disjoint right turns, left turns and a honeycomb of degree $d$ ($d \in \Z$). Every vertex of a local $\SL_3$-train track is \textbf{charged}. The vertex is \textbf{charged by $+$ (resp. $-$)} if it is either a source (resp. a sink) or a one-ended edge whose orientation is toward (resp. away from) the boundary.
\end{definition}

\begin{remark}
 By the Caveat in Section \ref{parallelmove}, one can change the order of right and left turns via an isotopy in $T \times (-1,1)$. Thus a local $\SL_3$-train track only depends on the number of right/left turns at each puncture.

Let $T = (e_1, e_2, e_3)$ be an (ideal) triangle. For a simple description of a local $\SL_3$-train track on $T$, we denote the number of right turns $r_{12}, r_{23}, r_{31}$ and the number of left turns $l_{21}, l_{32}, l_{13}$. As in Figure \ref{localSLthreeTTconf}, consider a weighted oriented graph, which is called the \textbf{configuration} of a local $\SL_3$-train track. The orientation of the Y-shaped locus $H_d$ is determined by the sign of $d$, which is a source if $d>0$ and sink if $d<0$.
\end{remark}

\begin{figure}
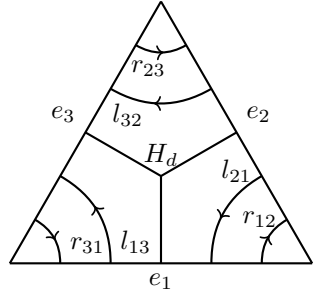

    \centering
    \SLthreeTT
    \caption{The configuration for a local $\SL_3$-train track.}\label{localSLthreeTTconf}
\end{figure}

\begin{definition}
    Let $B = (e_1, e_1')$ be an (ideal) biangle. Choose $d$ distinct vertices on each edge $e_1, e_1'$. Denote
    \[
    V_{e_1} = \{ v_1, \ldots, v_d \}
    \]
    and
    \[
    V_{e_1'} = \{v_1', \ldots, v_d' \}.
    \]
    Suppose that each vertex is charged with either $+$ or $-$ and the \textit{net charge} vanishes i.e. the total number of $+$ and $-$ coincides. Denote
    \[
    V_{e_1} = V_{e_1}^+ \sqcup V_{e_1}^-, \quad V_{e_1'} = V_{e_1'}^+ \sqcup V_{e_1'}^-
    \]
    where $V_{e_1}^+$(resp. $V_{e_1}^-$) is the set of positively (resp. negatively) charged vertices. The definition of $V_{e_1'}^\pm$ is analogous. The minimum crossing diagram connecting
    \[
    V_{e_1}^- \rightarrow V_{e_1'}^+, \quad V_{e_1'}^- \rightarrow V_{e_1}^+
    \]
    with orientation from $-$ to $+$ is called the \textbf{braid diagram} with respect to the charge.

    One may replace every crossing with a single sink-source pair in (\ref{hresolution}). The resulting oriented graph is called the \textbf{ladder web} corresponding to the braid diagram. An example of a pair of ladder web and braid diagram is illustrated in Figure \ref{ladderweb}.
\end{definition}

\begin{figure}
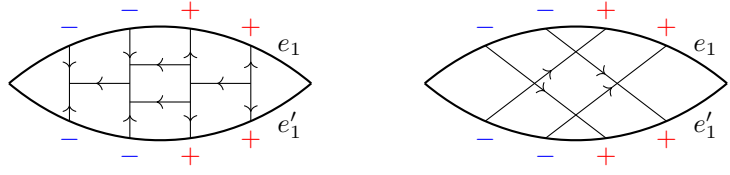

    \centering
    \ladderandbraid
    \caption{A ladder web and its corresponding braid diagram.}\label{ladderweb}
\end{figure}
    
\begin{definition}
    Any isotopy classes (in $\Sigma \times (-1,1)$) of a nonelliptic $3$-web on a surface $\Sigma$ is called a \textbf{global $\SL_3$-train track} on $\Sigma$.
\end{definition}

\begin{definition}
    Let $T = (e_1, e_2, f)$, $T' = (f', e_3, e_4)$ be two ideal triangles and let $W, W'$ be a local $\SL_3$-train track on $T, T'$. $W$ and $W'$ are called \textbf{compatible} by gluing a biangle $B = (f, f')$ if the \textit{net charge} of $B$ vanishes that is, the total number of $+$ and $-$ coincides at $f, f'$.
\end{definition}

\begin{remark}
    The collection of the number of each sign on ideal edges, which are also called the \textbf{signed intersection numbers}, together with the number of right turns is called the \textbf{Frohman-Sikora coordinate}.
\end{remark}

An example of compatible local $\SL_3$-train track is illustrated in Figure \ref{compSL3TT}. Note that the isotopy class of the ladder web in the biangle is not unique. For example, one can interchange a left turn and a right turn at the leftmost triangle in Figure \ref{compSL3TT}, leading another nonisotopic ladder web to connect two local $\SL_3$-train tracks.

\begin{figure}
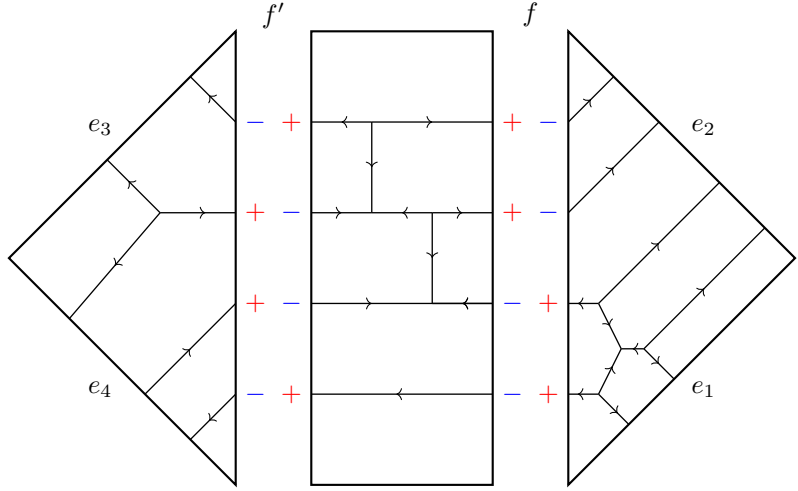

    \centering
    \compSLthreeTT
    \caption{Compatible local $\SL_3$-train tracks.}\label{compSL3TT}
\end{figure}

\begin{theorem}[Frohman-Sikora \cite{frohman20223}, Douglas--Sun \cite{douglas2024tropical}]\label{canonicalposition}
    Let $(\Sigma, \lambda)$ be a punctured surface with a split ideal triangulation $\lambda$ without self-folded triangles. Let $W$ be a nonelliptic $3$-web on $\Sigma$. Then $W$ is isotopic to a nonelliptic $3$-web $W'$ satisfying (which will said to be in the \textbf{canonical position})
    \begin{enumerate}
        \item For any ideal triangle $T \in \lambda$, the isotopy class of $W'|_T$ is a local $\SL_3$-train track.
        \item For any ideal biangle $B \in \lambda$, $W'|_B$ is a ladder web.
    \end{enumerate}
    Moreover, any such $W'$ has a unique local $\SL_3$-train track decomposition as in (1).
\end{theorem}

The detailed process to isotope a nonelliptic web to be in the canonical position will be discussed in Section \ref{tidyingups}. \ \\

Let $(\Sigma, \lambda)$ be a punctured surface with a split ideal triangulation $\lambda$ without self-folded triangles. In conclusion, one again obtains a correspondence
    \begin{gather*}
    \{ \text{Global $\SL_3$-train tracks on $\Sigma$} \} \ \\
    \updownarrow  \ \\
    \{ \text{Compatible collection of local $\SL_3$-train tracks on the ideal triangles on $\Sigma$} \}.
    \end{gather*}

Again the correspondence gives an embedding $\Lambda(\Sigma, \SL_3) \hookrightarrow \Z_{\geq 0}^{-12\chi(\Sigma)} \times \Z^{-2 \chi(\Sigma)}$. Since the compatibility conditions are closed under addition, $\Lambda(\Sigma, \SL_3)$ has a natural monoid structure depending on the choice of $\lambda$.

\begin{definition}[Douglas--Sun \cite{douglas2024tropical}]
    Let $(\Sigma, \lambda)$ be a punctured surface with an ideal triangulation $\lambda$ without self-folded triangles. Let $\lambda'$ be the corresponding split ideal triangulation. Let $V_{\lambda'}^3$ be the $3$-triangulation of ideal triangles in $\lambda'$. Note $|V_{\lambda'}^3| = -14 \chi(\Sigma)$.
    
    Let $W$ be a nonelliptic web on $\Sigma$. We may assume $W$ satisfy the condition in Theorem \ref{canonicalposition} with $\lambda'$ after an isotopy. The \textbf{Douglas--Sun coordinate} of $W$ with respect to $\lambda$ is the image of $W$ under the monoid morphism
    \begin{align}
    \Lambda(\Sigma, \SL_3) \hookrightarrow \Z_{\geq 0}^{-12\chi(\Sigma)} \times \Z^{-2 \chi(\Sigma)} \xrightarrow{\Phi} \Z_{\geq 0}^{V_{\lambda'}^3} \rightarrow \Z_{\geq 0}^{V_\lambda^3}. \label{DScoordmorphism}
    \end{align}
    Let $B = (e_1, e_1')$ be an ideal biangle in $\lambda'$. Then the last arrow in (\ref{DScoordmorphism}) is given by the projection
    \[
    \Z e_1 \oplus \Z e_1' \rightarrow \Z e_1.
    \]
    Note that the compatibility condition says the Douglas--Sun coordinate of $W$ does not depend on the choice of $e_1$ in the ideal biangle $B = (e_1, e_1')$.

    The monoid morphism $\Phi$ is determined by the image of generators of a local $\SL_3$-train track. Use the coordinate for local $\SL_3$-train track
    \[
    (r_{12}, l_{21}, r_{23}, l_{32}, r_{31}, l_{13}, H_d) \in \Z_{\geq 0}^6 \times \Z
    \]
    in Figure \ref{localSLthreeTTconf}. Use the coordinate
    \[
    (e_{11}, e_{12}, e_{21}, e_{22}, e_{31}, e_{32}, e)
    \]
    for the target monoid. Then $\Phi$ is given as
    \begin{align*}
        (1,0,0,0,0,0,0) &\mapsto (1,2,2,1,0,0,1) \ \\
        (0,1,0,0,0,0,0) &\mapsto (2,1,1,2,0,0,2) \ \\
        (0,0,1,0,0,0,0) &\mapsto (0,0,1,2,2,1,1) \ \\
        (0,0,0,1,0,0,0) &\mapsto (0,0,2,1,1,2,2) \ \\
        (0,0,0,0,1,0,0) &\mapsto (2,1,0,0,1,2,1) \ \\
        (0,0,0,0,0,1,0) &\mapsto (1,2,0,0,2,1,2) \ \\
        (0,0,0,0,0,0,1) &\mapsto (2,1,2,1,2,1,3) \ \\
        (0,0,0,0,0,0,-1) &\mapsto (1,2,1,2,1,2,3),
    \end{align*}
    which is illustrated in Figure \ref{DScoordfigure}. We call the coordinates $(e_{11}, e_{12}, e_{21}, e_{22}, e_{31}, e_{32})$ as \textbf{edge coordinates} and $e$ as a \textbf{face coordinate}.
\end{definition}

\begin{figure}
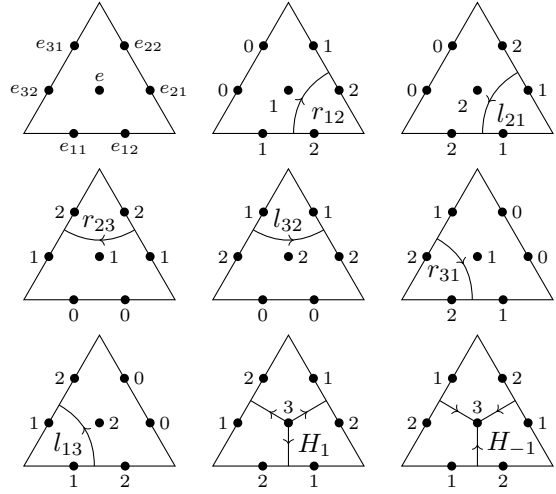

    \centering
    \DSdiagram
    \caption{Douglas--Sun coordinate.}\label{DScoordfigure}
\end{figure}

\begin{theorem}[Douglas--Sun \cite{douglas2024tropical}]\label{DScoordinate}
    The image of the monoid morphism
    \[
    \Lambda(\Sigma, \SL_3) \hookrightarrow \Z_{\geq 0}^{V_\lambda^3}
    \]
    is characterized by the rhombus inequalities 
    \begin{quote}
        For any ideal triangle $T = (e_1, e_2, e_3)$ with corresponding Douglas--Sun coordinate $(e_{11}, e_{12}, e_{21}, e_{22}, e_{31}, e_{32}, e)$, one has
        \begin{align*}
            e_{31} + e_{22} - e &\in 3\Z_{\geq 0}, \ \\
            e_{32} + e - e_{31} - e_{11} &\in 3\Z_{\geq 0}, \ \\
            e + e_{21} - e_{22} - e_{12} &\in 3\Z_{\geq 0}, \ \\
            e_{11} + e_{32} - e &\in 3\Z_{\geq 0}, \ \\
            e_{12} + e - e_{11} - e_{21} &\in 3\Z_{\geq 0}, \ \\
            e + e_{31} - e_{32} - e_{22} &\in 3\Z_{\geq 0}, \ \\
            e_{21} + e_{12} - e &\in 3\Z_{\geq 0}, \ \\
            e_{22} + e - e_{21} - e_{31} &\in 3\Z_{\geq 0}, \ \\
            e + e_{11} - e_{12} - e_{32} &\in 3\Z_{\geq 0}.
        \end{align*}
    \end{quote}
    In particular, $\Lambda(\Sigma, \SL_3)$ has a structure of a positive rational polyhedral cone\footnote{A submonoid $M \leq \Z^n$ is called \textbf{positive} if $x \in M$ and $-x \in M$ implies $x = 0$. It is called a \textbf{rational polyhedral cone} if it is finitely generated.}
\end{theorem}

Again the rhombus inequalities are illustrated in Figure \ref{rhombusSL3} and given by (\ref{rhombineq}).

\begin{remark}
    Let $\Gamma$ be a multicurve on $(\Sigma, \lambda)$. The intersection coordinate can be defined even where $\Gamma$ does not minimally intersect $\lambda$. 

    Similarly, even for
    \begin{enumerate}
        \item a web $W$ not satisfying the conditions in Theorem \ref{canonicalposition}
        \item a web $W$ not nonelliptic
        \item a link element in $\Sk_R^3(\Sigma)$
    \end{enumerate}  
    the Douglas--Sun coordinate can be defined. The explicit construction is given as follows:
    \begin{enumerate}
        \item For every ideal edge $i$ one can count the number of charges. Then
        \begin{align*}
            e_{i1} &= 2 \cdot (\text{\# of positive charges}) + (\text{\# of negative charges}) \\
            e_{i2} &= (\text{\# of positive charges}) + 2 \cdot (\text{\# of negative charges)}.
        \end{align*}
        \item For every ideal triangle $T$ one can count the number of right turns in that triangle. Then
        \[
        e = (\text{\# of total charges}) - (\text{\# of right turns}).
        \]
    \end{enumerate}

    The observation will be used in Section \ref{defects}.
\end{remark}

\begin{figure}
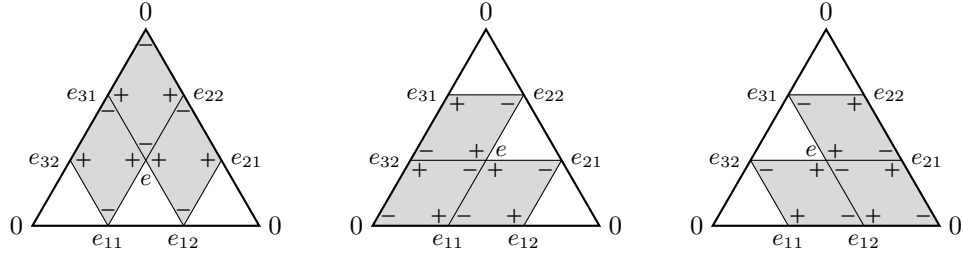

    \centering
    \rhombSLthreeTT
    \caption{Rhombus inequalities for $\SL_3$-train tracks.}\label{rhombusSL3}
\end{figure}

\subsection{Filtrations of skein algebras}

From now on, whenever we write $\Lambda(\Sigma, \SL_n)$ or $\Sk_R^n(\Sigma)$, we assume that $n \in \{2,3\}$ and $\Sigma$ is equipped with an ideal triangulation $\lambda$ without self-folded triangles. We use the notation
\[
\xb \in \Lambda(\Sigma, \SL_n)
\]
and
\[
\tr_{\xb} \in \Sk_R^n(\Sigma)
\]
to clarify the ambient space of a basis element.

In this subsection, let $(\Sigma, \lambda)$ be a punctured surface with an ideal triangulation without self-folded triangles. Every multicurve or nonelliptic web is considered to be transversal to $\lambda$.

Fix a total order on the set $V_\lambda^n$ and write
\[
| \mathbf{v} | := \mathbf{v} \cdot (1, \ldots, 1) \in \Z
\]
for $\mathbf{v} \in \Z^{V_\lambda^n}$. Then both the abelian group $\Z^{V_\lambda^n}$ and the $\SL_n$-Fock-Goncharov torus algebra $\Xcal_R(\Sigma, \lambda, \SL_n)$ have a natural monomial ordering
\[
\mathbf{v} \leq \mathbf{w} \Longleftrightarrow x^{\mathbf{v}} \leq x^{\mathbf{w}} \Longleftrightarrow (| \mathbf{v}|, \mathbf{v}) \leq_{\mathrm{lex}} (|\mathbf{w}|, \mathbf{w}).
\]

$\Lambda(\Sigma, \SL_n)$ are also equipped with an ordering as a submonoid of $\Z^{V_\lambda^n}$.

Since $\Xcal_\C(\Sigma, \lambda, \SL_n)$ is filtered by the monomial ordering and the trace map is injective for $n \in \{2, 3\}$, $\Sk_\C^n(\Sigma)$ is filtered for $n \in \{2, 3\}$. 

\begin{theorem}[Trace map and intersection coordinates \cite{bonahon2011quantum}]
     Let $\Gamma$ be a multicurve on $\Sigma$. Regard $\Gamma \in \Z^{V_\lambda^2}$ via its intersection coordinate. Then the image of $\Gamma$ by the trace map has the leading term $x^\Gamma$.
\end{theorem}

\begin{theorem}[Trace map and Douglas--Sun coordinates \cite{kim2020rm, le2023quantum}]
    Let $W$ be a nonelliptic web on $\Sigma$. Abuse $W \in \Z^{V_\lambda^3}$ by the Douglas--Sun coordinate. Then the image of $\tr_W$ by the trace map has the leading term $x^W$.
\end{theorem}

For $n=2$, write an element of $\Sk_R^2(\Sigma)$ as a finite sum
\[
\sum_\Gamma a_\Gamma \tr_\Gamma
\]
where $a_\Gamma = 0$ for all but finitely many multicurves $\Gamma$. This defines an assignment
\begin{align}
    \sum_\Gamma a_\Gamma \tr_\Gamma \mapsto \max \{ \Gamma : a_\Gamma \neq 0 \} \label{SL2TropAssign}
\end{align}
where the maximum is taken with respect to the total order of $\Lambda(\Sigma, \SL_2)$.

Analogously for $n=3$, write an element of $\Sk_R^3(\Sigma)$ as a finite sum
\[
\sum_W a_W \tr_W
\]
where $a_W = 0$ for all but finitely many nonelliptic webs $W$. One has an assignment
\begin{align}
    \sum_W a_W \tr_W \mapsto \max \{ W : a_W \neq 0 \}. \label{SL3TropAssign}
\end{align}

\begin{theorem}[Tropical multiplication formula for $\SL_2$]\label{TMlawSL2}
   Fix a total order on $V_\lambda^2$. Then for any two multicurves $\Gamma_1, \Gamma_2 \in \Lambda(\Sigma, \SL_2)$ the expansion
    \[
    \tr_{\Gamma_1} \cdot \tr_{\Gamma_2} = \pm \tr_{\Gamma_1 + \Gamma_2} + \sum_{\Gamma < \Gamma_1 + \Gamma_2} r_\Gamma \tr_\Gamma
    \]
    holds in $\Sk_R^2(\Sigma)$.
\end{theorem}

\begin{theorem}[Tropical multiplication formula for $\SL_3$]\label{TMlawSL3}
    Fix a total order on $V_\lambda^3$. Then for any two nonelliptic webs $W_1, W_2 \in \Lambda(\Sigma, \SL_3)$ the expansion
    \[
    \tr_{W_1} \cdot \tr_{W_2} = \tr_{W_1 + W_2} + \sum_{W < W_1 + W_2} r_W \tr_W
    \]
    holds in $\Sk_R^3(\Sigma)$.
\end{theorem}

Note that in Theorem \ref{TMlawSL2} and Theorem \ref{TMlawSL3} one can choose the order on $V_\lambda^2$ or $V_\lambda^3$ arbitrarily. For any $\Gamma \neq \Gamma_1 + \Gamma_2$ and $|\Gamma| = |\Gamma_1 + \Gamma_2|$, one can choose an order on $V_\lambda^2$ such that
\[
\Gamma > \Gamma_1 + \Gamma_2
\]
so $a_\Gamma = 0$. It follows that one can again write
\[
\tr_{\Gamma_1} \cdot \tr_{\Gamma_2} = \pm \tr_{\Gamma_1 + \Gamma_2} + \sum_{|\Gamma| < |\Gamma_1 + \Gamma_2|} r_\Gamma \tr_\Gamma
\]
and similarly for $n=3$
\[
    \tr_{W_1} \cdot \tr_{W_2} = \tr_{W_1 + W_2} + \sum_{|W| < |W_1 + W_2|} r_W \tr_W.
\]

\subsection{Relative skein algebras}\label{relativeskeinalgebras}

\begin{definition}\label{relativeskein}
    Let $\Sigma$ be a punctured surface with punctures indexed by $P = \{1, 2, \ldots, m \}$. 
    \begin{enumerate}
        \item Let $p_i \subseteq \Sk_R^2(\Sigma)$ be a unoriented simple closed curve, homotopic to the $i$th puncture. Choose $\mathbf{k} = (k_i)_{i=1}^m \in R^m$.
        The \textbf{relative $\SL_2$-skein algebra of $\Sigma$ over $R$} is
        \[
        \Sk_R^2(\Sigma, P, \mathbf{k}) := \Sk_R^2(\Sigma) / (p_i - k_i)_{i=1}^m.
        \]
        A multicurve all of whose components are among the $p_i$ is called \textbf{peripheral}.
        \item For each puncture, let $p_i$ (resp. $p'_i$ be the clockwise (resp. counterclockwise) oriented simple closed curve homotopic to that puncture. Choose $\mathbf{k}=(k_i, k_i')_{i=1}^m \in R^{2m}$.
        The \textbf{relative $\SL_3$-skein algebra of $\Sigma$ over $R$} is
        \[
        \Sk_R^3(\Sigma, P, \mathbf{k}) := \Sk_R^3(\Sigma) / (p_i - k_i, p_i' - k_i')_{i=1}^m.
        \]
        A nonelliptic web all of whose components are among the $p_i$ and $p_i'$ is called \textbf{peripheral}.
    \end{enumerate}
\end{definition}

Relative $\SL_2$-skein algebras are introduced under the name \textit{sliced skein algebras} in \cite{frohman2025sliced}. Note that the peripheral elements of $\Lambda(\Sigma, \SL_n)$ form a submonoid.

Let $\Sigma = \Sigma_{g,m}$, $\pi = \pi_1 \Sigma$ and consider the morphism
\begin{equation*}\label{boundarymorphism}
        X(\pi, \SL_n)_R \rightarrow (\SL_n \sslash \SL_n)_R^m \simeq \A_R^{m(n-1)}
\end{equation*}

by recording the coefficients of the characteristic polynomial of each simple puncture curve with the clockwise orientation.

A fiber of the morphism is called the \textbf{relative $\SL_n$-character variety of $(\Sigma, P)$ over $R$}, denoted by $X^{\mathrm{rel}}(\pi, \SL_n, \mathbf{k})_R$. Under the correspondence $\gamma \leftrightarrow \tr_\gamma$, if $n=2$ then the relative $\SL_2$-skein algebra is isomorphic to the fiber at the point $\mathbf{k} = (k_1, \ldots, k_m) \in \A_R^m$ and similarly for $n=3$, $\mathbf{k}=(k_1, \ldots, k_m, k_1', \cdots, k_m') \in \A_R^{2m}$. \ \\

\begin{example}\label{relskeinexample}
    \begin{enumerate}
        \item Let $\Sigma = \Sigma_{1,1}$ and $\mathbf{k} = k \in R$. Let $\pi = \pi_1 \Sigma = \langle \alpha, \beta \rangle$. Then 
        \[
        \Sk_R^2(\Sigma, P, \mathbf{k}) \simeq R[x,y,z] / (x^2 + y^2 + z^2 - xyz - k - 2) 
        \]
        where $(x, y, z) = (\tr_\alpha, \tr_\beta, \tr_{\alpha \beta})$.

        Consider the standard embedding into projective space
        \[
        \Spec \Sk_R^2(\Sigma, P, \mathbf{k}) \hookrightarrow \mathbb{P}^3
        \]
        then the boundary divisor is given by $xyz$ and its dual intersection complex is homeomorphic to $S^1$.
        
        \item Let $\Sigma = \Sigma_{0,3}$ and $\mathbf{k} = (k_1, k_1' , k_2, k_2', k_3, k_3') \in R^6$. Let $\pi = \pi_1 \Sigma = \langle \alpha, \beta \rangle$ and $(X, Y, Z) = (\tr_{\alpha \beta^{-1}}, \tr_{\alpha^{-1} \beta}, \tr_{[\alpha, \beta]})$. Then 
        \[
        \Sk_R^3(\Sigma, P, \mathbf{k}) \simeq R[X,Y,Z] / (X^3 + Y^3 - XYZ + f_{\mathbf{k}}(X, Y, Z))
        \]
        where $f_\mathbf{k}(X, Y, Z) \in R[X,Y,Z]$ is a quadratic polynomial depending on $\mathbf{k}$. One can only ensure that $f_\mathbf{k} \in R[X,Y][Z]$ is monic as a polynomial in $Z$ by Lawton \cite{lawton2014sl3ccharactervarietiesrp2structurestrinion}.

        Note that the coordinate change in Example \ref{LawtonSL3trinion}
        \[
        (\tr_{W_\theta}, \tr_{W_\theta '}, \tr_{W_{\mathrm{comm}}}) = (\tr_{\alpha \beta^{-1}}, \tr_{\alpha^{-1} \beta}, \tr_{[\alpha, \beta]}) - (k_1 k_2', k_1' k_2, k_1 k_1' + k_2 k_2' + k_3 k_3')
        \]
        is a translation. 
        
        Applying the translation
        \[
        (x, y) = (X - k_1 k_2', Y - k_1' k_2)
        \]
        eliminates the cross terms $xz, yz$ to separate variables. Similarly, applying the translation
        \[
        z = Z - k_1 k_1' - k_2 k_2' - k_3 k_3'
        \]
        fixes the coefficient of $xy$ to $-6$.

        One can conclude that
        \[
        \Sk_R^3(\Sigma, P, \mathbf{k}) \simeq R[x, y, z] / (x^3 + y^3 - xyz - 6xy + f^1_{\mathbf{k}}(x) + f^2_{\mathbf{k}}(y) + f^3_\mathbf{k}(z)) =: R[x,y,z]/(f)
        \]
        where $f^1_\mathbf{k}, f^2_\mathbf{k}, f^3_\mathbf{k}$ is a quadratic polynomial in a single variable and $f^3_\mathbf{k}$ is furthermore monic. Let $\overline{z} = xy - z$ be another root of the polynomial of $f$ as a monic quadratic polynomial in $z$.

        Consider a weighted projective embedding
        \[
        \Spec \Sk_R^3(\Sigma, P, \mathbf{k}) \hookrightarrow \mathbb{P}(1,2,2,3,3)
        \]
        where $(1, 2, 2, 3, 3) = (\deg t, \deg x, \deg y, \deg z, \deg \overline{z})$ and $t$ be a homogenizing variable. The choice of the weights will be justified in Section \ref{Compactification} and \ref{example}. By the relation
        \[
        z + \overline{z} = xy, \quad z\overline{z} = x^3 + y^3 - 6xy + f^1_\mathbf{k}(x) + f^2_\mathbf{k}(y),
        \]
        the boundary divisor is given as
        \[
        D = \Proj R[x,y,z,\overline{z}]/(xy, z\overline{z} - x^3 - y^3).
        \]
        If $R$ were a domain, $z\overline{z} - w^3 \in R[z, \overline{z}, w]$ is irreducible and one can conclude that
        \[
        D_1 = \Proj R[x,y,z,\overline{z}]/(x, z\overline{z} - x^3 - y^3), \quad D_2 = \Proj R[x,y,z,\overline{z}]/(y, z\overline{z} - x^3 - y^3)
        \]
        be two irreducible components of $D$. They intersect at two points $[x : y : z : \overline{z} : t] = [0 : 0 : 1 : 0 : 0]$ and $[x : y : z : \overline{z} : t] = [0 : 0 : 0 : 1 : 0]$.
    \end{enumerate}    
\end{example}

\section{Topology of curves and webs}\label{Section3}

\subsection{Tidying-ups}\label{tidyingups}

Inspired by \cite{frohman20223}, we introduce a tidying-up procedure of two tangled multicurves or nonelliptic webs in $\Sk_R^n(\Sigma)$. 

Let $\Gamma_1, \Gamma_2$ be two multicurves on a surface $\Sigma$ with an ideal triangulation $\lambda$. We may assume $\Gamma_1, \Gamma_2$ intersect minimally with $i(\Gamma_1, \Gamma_2) = d$. To compute $\tr_{\Gamma_1} \cdot \tr_{\Gamma_2}$ in $\Sk_R^2(\Sigma)$ up to $\pm 1$, 
\begin{enumerate}
    \item Use the skein relation (\ref{xresolution}) to obtain $2^d$ multicurves.
    \item Resolve any bigon formed by an ideal edge of $\lambda$ and one of the resulting multicurves.
\end{enumerate}

This procedure is called the \textbf{tidying-up} of two tangled multicurves $\Gamma_1, \Gamma_2$ in $\Sk_R^2(\Sigma)$. 

\begin{definition}
    Let $W$ be a nonelliptic web on a surface $\Sigma$ with an ideal triangulation. Let $T, T'$ be two ideal triangles sharing an edge $e$.
    \begin{enumerate}
        \item If $e$ and $W$ form a bigon, one can isotope $W$ to resolve that bigon. The process is called the \textbf{cap move}.
        \item If there exists a sink or a source of $W$ forming a triangle with $e$, one can isotope $W$ to resolve that triangle. The process is called the \textbf{vertex move}.
        \item If there exists a consecutive sink and source forming a $4$-gon with $e$, one can isotope $W$ to push the $4$-gon across $e$. The process is called the \textbf{crossbar move}.
    \end{enumerate}
    Each of these procedures is called a \textbf{move} and illustrated in Figure \ref{TidyingUpOperations}. Every move is \textit{local}, that is to say, only 
    \begin{enumerate}
        \item the edge coordinate associated with $e$
        \item the face coordinate of $T$ and $T'$
    \end{enumerate} near the deformation changes by the move.
\end{definition}

\begin{figure}
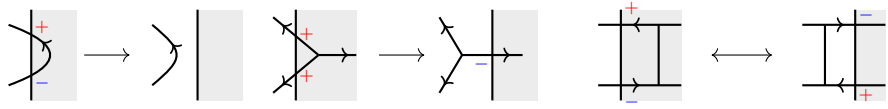

    \centering
    \TidyingUps
    \caption{(a) The cap move. (b) The vertex move. (c) The crossbar move.}
    \label{TidyingUpOperations}
\end{figure}

Let $W_1, W_2$ be nonelliptic webs on a surface $\Sigma$ with an ideal triangulation $\lambda$. We may assume $W_1, W_2$ does not form a bigon and $|W_1 \cap W_2| = d$. To compute $\tr_{W_1} \cdot \tr_{W_2}$ in $\Sk_R^3(\Sigma)$,
\begin{enumerate}
    \item Use the skein relation (\ref{hresolution}) to obtain $2^d$ webs, which may not be nonelliptic. 
    \item Use the skein relations (\ref{2gonresolution}), (\ref{4gonresolution}) to replace each web nonelliptic.
    \item Fix an ideal triangle $T$. Use cap moves, vertex moves, and crossbar moves to make $W|_T$ a local $\SL_3$-train track. 
    \item Consider a split ideal triangulation. Note each move in (3) may be regarded as \textit{dumping} some pieces into an ideal biangle.
    \item Iterate the process for every ideal triangle.
\end{enumerate}

The process is called the \textbf{tidying-up} of two nonelliptic webs $W_1, W_2$ in $\Sk_R^3(\Sigma)$. After completing the tidying-up, one obtains a local $\SL_3$-train track on each ideal triangle and a ladder web on each ideal biangle\cite{frohman20223, douglas2024tropical}.

\begin{remark}
    Let $W_1, W_2$ be as above. While tidying-up $W_1 \cup W_2$, observe the following facts:
    \begin{enumerate}
        \item If the skein relation (\ref{2gonresolution}) occurs in the interior of a triangle $T$, the edge coordinates of $T$ do not change and the face coordinate of $T$ does not increase.
        \item If the bigon in the skein relation (\ref{2gonresolution}) intersects an ideal edge $e$ adjacent to two triangles $T, T'$, the edge coordinate of $e$ decreases by either $(3, 0)$ or $(0, 3)$ and the face coordinate of $T$ and $T'$ decreases.
        \item If the skein relation (\ref{4gonresolution}) occurs in the interior of a triangle $T$, the edge coordinates of $T$ do not change and the face coordinate of $T$ does not increase.
        \item If the $4$-gon in the skein relation (\ref{4gonresolution}) intersects an ideal edge $e$ connecting two triangles $T, T'$, up to symmetry there are two cases: $e$ transverses either opposite edges or consecutive edges of the $4$-gon.
            \begin{enumerate}
                \item If the former occurs, the edge coordinate of $e$ of in one term decreases by $(3, 3)$ and the other term remains unchanged. The face coordinate of $T, T'$ decreases.
                \item If the latter occurs, the edge coordinate in each term decreases by either $(3, 0)$ or $(0, 3)$. The face coordinate of $T, T'$ decreases.
            \end{enumerate}
        \item The cap move on $(e, T)$ decreases the edge coordinate of $e$ by $(3, 3)$. The face coordinate of $T$ decreases by $2$.
        \item The vertex move on $(e, T)$ decreases the edge coordinate of $e$ by either $(3, 0)$ or $(0, 3)$. The face coordinate of $T$ decreases.
        \item The crossbar move on $(e, T)$ does not change the edge coordinate of $e$. The face coordinate of $T$ does not increase.
        \item Process (2) can be obtained by a vertex move followed by Process (1).
        \item Process (4)-(a) can be obtained by a crossbar move followed by Process (3) and a cap move.
        \item Process (4)-(b) can be obtained by a vertex move followed by Process (3).
    \end{enumerate}
\end{remark}

Let $W$ be a $3$-web on $\Sigma$ with an ideal triangulation. Let $T = (e_1, e_2, e_3)$ where the edges are listed in counterclockwise order and $W' = W|_T$. The restriction $W'$ has a charge vector
\[
(e_1^+, e_2^+, e_3^+, e_1^-, e_2^-, e_3^-).
\]
$W'$ is said to be \textbf{balanced} if there exists a local $\SL_3$-train track $W_0$ whose charge vector coincides the one of $W'$.

\begin{remark}
    Let $W$ be a nonelliptic web on $\Sigma$ in the canonical position. Let $T$ be an ideal triangle and $B_1, B_2, B_3$ be ideal biangles adjacent to $T$. Let $T' = T \cup B_1 \cup B_2 \cup B_3$. Then $W' = W|_{T'}$ is balanced. With respect to its coordinate, the balancedness of a charge vector in $\Z_{\geq 0}^6$ is equivalent to the condition
    \[
    e_1^+ - e_1^- = e_2^+ - e_2^- = e_3^+ -e_3^-.
    \]
\end{remark}

\begin{lemma}
    Let $W$ be a nonelliptic web on $\Sigma$ and $T$ be an ideal triangle. If $W' = W|_T$ is not balanced then the face coordinate of $T$ decreases while tidying-up.
\end{lemma}

\begin{proof}
    Since $W$ is nonelliptic, it suffices to consider the moves. Since $W'$ is not balanced, at least a move is needed and the first move cannot be a crossbar move.
\end{proof}

\begin{theorem}
    Let $W_1, W_2$ be two nonelliptic webs on $\Sigma$. After the tidying-up of $W_1$ and $W_2$, the face coordinate of each ideal triangle does not increase.
\end{theorem}

\begin{proof}
We may assume the following:
\begin{enumerate}
    \item $W_1 \cap W_2$ does not meet the ideal triangulation of $\Sigma$.
    \item The restrictions of $W_1$ and $W_2$ to each ideal triangle are balanced.
\end{enumerate}
To satisfy the second assumption, consider a split triangulation with respect to which $W_1, W_2$ are in the canonical position. Choosing one of the parallel ideal edges gives another ideal triangulation. 

Only skein relation (\ref{hresolution}) can decrease the number of right turns. The situation occurs only when a right turn overlaps a non-arc component as in Figures \ref{ExceptionalRTdecreaseI} and \ref{ExceptionalRTdecreaseII}, up to symmetry. 

If the web $W'$ in on the left-hand side of the figure is nonelliptic and balanced, after resolving the first intersection, one obtains a web possibly with self-intersections. The charge vector changes as follows:
\[
(e_1^+, e_2^+, e_3^+, e_1^-, e_2^-, e_3^-) \mapsto (e_1^+ -1, e_2^+ +1, e_3^+ , e_1^-, e_2^-, e_3^-)
\]
up to reordering $e_1, e_2, e_3$. One can resolve each intersection and then resolve the $2$-gons and $4$-gons in the interior of $T$. Since the process does not change the charge vector, every resulting $3$-web restricted to the triangle is \textit{not balanced}. Hence each exceptional resolution that decreases the number of right turn should be followed by a process of decreasing the face coordinate.
\end{proof}

\begin{figure}
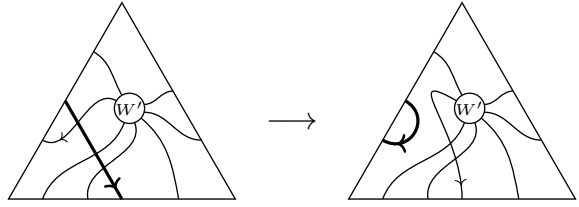

    \centering
    \ExceptionalRTdecreaseI
    \caption{A resolution decreasing the number of right turns: the first case.}
    \label{ExceptionalRTdecreaseI}
\end{figure}

\begin{figure}
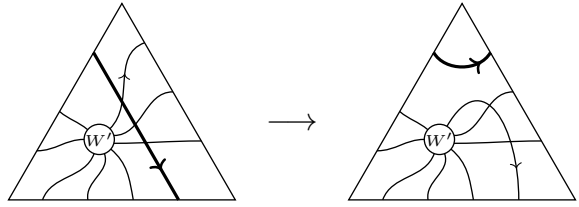

    \centering
    \ExceptionalRTdecreaseII
    \caption{A resolution decreasing the number of right turns: the second case.}
    \label{ExceptionalRTdecreaseII}
\end{figure}

\subsection{Defects}\label{defects}

Consider $\xb, \xb' \in \Lambda(\Sigma, \SL_n)$ and write 
\[
\tr_{\xb_1} \cdot \tr_{\xb_2} = \pm \tr_{\xb_1 + \xb_2} + \sum_{|\xb| < |\xb_1 + \xb_2|} r_\xb \tr_\xb.
\]
Define
\[
\Supp(\xb_1, \xb_2) := \{ \xb \in \Lambda(\Sigma, \SL_n) : r_\xb \neq 0\}.
\]

For a monoid $M$, its \textit{Grothendieck group} or the \textit{groupification} will be denoted by $K_0(M)$. Since $M$ is cancellative, we have an embedding $\Lambda(\Sigma, \SL_n) \hookrightarrow K_0(\Lambda(\Sigma, \SL_n)) \subseteq \Z^N$.

\begin{definition}
    The \textbf{defect} of $\xb, \xb' \in \Lambda(\Sigma, \SL_n)$ is the \textit{subset}
    \[
    \Def(\xb, \xb', \SL_n) := \{ \xb + \xb' - \vb : \vb \in \Supp(\xb, \xb') \} \subseteq K_0(\Lambda(\Sigma, \SL_n))
    \]
    The \textbf{defect} of $\Sk_R^n(\Sigma)$ is the \textit{subset}
    \[
    \Def(\Sigma, \SL_n) := \bigcup_{(\xb, \xb') \in \Lambda(\Sigma, \SL_n)^2} \Def(\xb, \xb', \SL_n) \subseteq K_0(\Lambda(\Sigma, \SL_n)).
    \]
\end{definition}

\begin{proposition}\label{SL2defect}
    Via the intersection coordinate, 
    \[
    \Def(\Sigma, \SL_2) \subseteq (2\Z_{\geq 0})^{-3 \chi(\Sigma)} \subseteq K_0(\Lambda(\Sigma, \SL_2)).
    \]
\end{proposition}

\begin{proof}
    Resolving intersections of two multicurves does not change the intersection coordinate. Tidying-up consists of resolving bigons, which decreases the relevant intersection coordinate by $2$. 
\end{proof}

\begin{proposition}\label{SL3defect}
    Via the Douglas--Sun coordinate,
    \[
    \Def(\Sigma, \SL_3) \subseteq (3\Z_{\geq 0})^{-8 \chi(\Sigma)} \subseteq K_0(\Lambda(\Sigma, \SL_3)).
    \]
\end{proposition}

\begin{proof}
    By Section \ref{tidyingups}, we have the followings:
    \begin{enumerate}
        \item Tidying-up decreases the edge coordinate by a multiple of $3$. \label{tidyingup1}
        \item After tidying-up, no face coordinate increases. \label{tidyingup2}
    \end{enumerate}
    
    As a result, after tidying-up the Douglas--Sun coordinate does not increase. By the tropical multiplication formula, there is a unique term whose edge coordinate remains unchanged during tidying-up. By the rhombus congruence conditions, each face coordinate after tidying-up also decreases by a multiple of $3$, which completes the proof.
\end{proof}

\subsection{Irreducible and primitive elements in $\Lambda(\Sigma, \SL_n)$}

\begin{definition}
    Let $M$ be a positive monoid.
    \begin{enumerate}
        \item $m \in M \setminus \{0\}$ is \textbf{irreducible} if $m = m_1 + m_2$ for some $m_1, m_2 \in M$ then $m_1 = m$ or $m_2 = m$.
        \item $m \in M \setminus \{0\}$ is \textbf{preprimitive} if $m = a m'$ for some $a$ then $m = m'$.
    \end{enumerate}

    $\xb \in \Lambda(\Sigma, \SL_n)$ is \textbf{primitive} if it is preprimitive and connected.
\end{definition}

If $\xb \in \Lambda(\Sigma, \SL_n)$ is irreducible then it is primitive but the converse does not hold.

\begin{lemma}
    Any $p_i \in \Lambda(\Sigma, \SL_2)$(resp. $p_i, p_i' \in \Lambda(\Sigma, \SL_3)$) is irreducible.
\end{lemma}

\begin{proof}
    Consider $p_i \in \Lambda(\Sigma, \SL_2)$ as $p_i \subseteq \Sigma$. By the intersection coordinate, there exists an embedded punctured disk $D_i'$ containing $p_i$. Any essential multicurve in $D_i'$ is copy of $p_i$ so $p_i$ is irreducible. Using the Frohman-Sikora coordinates, the proof for $p_i, p_i'$ with orientation is analogous.
\end{proof}

\begin{lemma}\label{primitiveskein}
    Any simple closed curve in $\Lambda(\Sigma, \SL_n)$ is primitive.
\end{lemma}

\begin{proof}
    Suppose $\xb \in \Lambda(\Sigma, \SL_n)$ is a simple closed curve satisfying $\xb = a\xb'$. The local configuration of $\xb$ on each triangle should be $a$ copies of the local configuration of $\xb'$. This forces $a=1$ and $\xb$ is primitive.
\end{proof}

Denote the set of irreducible (resp. primitive) elements of $\Lambda(\Sigma, \SL_n)$ as $\Lambda^{\mathrm{irr}}(\Sigma, \SL_n)$ (resp. $\Lambda^{\mathrm{prim}}(\Sigma, \SL_n)$). By the peripheral decomposition which will be introduced in Section \ref{PerEssSkeins}, one has
\[
\Lambda^{\mathrm{irr}}(\Sigma, \SL_2) \cap \Lambda^{\mathrm{ess}}(\Sigma, \SL_2) = \Lambda^{\mathrm{irr}}(\Sigma, \SL_2) \setminus \{p_i\}_{i=1}^m
\]
and
\[
\Lambda^{\mathrm{irr}}(\Sigma, \SL_3) \cap \Lambda^{\mathrm{ess}}(\Sigma, \SL_3) = \Lambda^{\mathrm{irr}}(\Sigma, \SL_3) \setminus \{p_i, p_i'\}_{i=1}^m.
\]
The analogous statements hold with $\Lambda^{\mathrm{irr}}$ replaced by $\Lambda^{\mathrm{prim}}$. By Lemma \ref{primitiveskein} and the multicurve basis, one has
\[
\Lambda^{\mathrm{prim}}(\Sigma, \SL_2) = \{ \text{isotopy classes of simple closed curves in $\Sigma$} \} 
\]
and
\[
\Lambda^{\mathrm{prim}}(\Sigma, \SL_2) \cap \Lambda^{\mathrm{ess}}(\Sigma, \SL_2)= \{ \text{isotopy classes of essential simple closed curves in $\Sigma$} \}.
\]

\begin{remark}
    $\Lambda^{\mathrm{irr}}(\Sigma, \SL_n)$ is also called the \textbf{Hilbert basis} of $\Lambda(\Sigma, \SL_n)$. The geometric classification of the Hilbert basis of $\Lambda(\Sigma, \SL_2)$ with respect to an ideal triangulation is described in \cite[Section 2.5]{farajzadeh2026compactifications}.
\end{remark}

\section{Algebraic properties of skein algebras}\label{Section4}

\subsection{Peripheral and essential skeins}\label{PerEssSkeins}

Let $\Sigma$ be a surface. By abuse of notation, let $\Pcal$ denote the subalgebra of $\Sk_R^2(\Sigma)$(resp. $\Sk_R^3(\Sigma)$) generated by peripheral elements $p_i$(resp. $p_i, p_i'$) in Definition \ref{relativeskein}.

Each $p_i, p_i'$ can be isotoped to be disjoint from any simple closed curve or nonelliptic web. Thus
\[
\tr_\gamma \cdot \tr_\Gamma = \tr_{\gamma + \Gamma}
\]
for any $\tr_\gamma \in \Pcal$ and $\Gamma \in \Lambda(\Sigma, \SL_2)$. Similarly
\[
\tr_\gamma \cdot \tr_W = \tr_{\gamma + W}
\]
for any $\tr_\gamma \in R[p_i, p_i']_{i=1}^m$ and $W \in \Lambda(\Sigma, \SL_3)$.

\begin{definition}
    A multicurve or a nonelliptic web is \textbf{essential} if it has no peripheral component. Denote the set of essential multicurves or nonelliptic webs by $\Lambda^{\mathrm{ess}}(\Sigma, \SL_n)$.
\end{definition}

Unlike the peripheral elements, $\Lambda^{\mathrm{ess}}(\Sigma, \SL_n)$ is not a submonoid of $\Lambda(\Sigma, \SL_n)$. Each multicurve $\Gamma \in \Lambda(\Sigma, \SL_2)$ or nonelliptic web $W \in \Lambda(\Sigma, \SL_3)$ is uniquely decomposed as
\[
W = W_1 + W_2
\]
where $W_1$ peripheral and $W_2 \in \Lambda^{\mathrm{ess}}(\Sigma, \SL_n)$. This decomposition is called \textbf{peripheral decomposition}.

Let $\Ecal$ be the free $R$-submodule of $\Sk_R^n(\Sigma)$ spanned by $\Lambda^{\mathrm{ess}}(\Sigma, \SL_n)$.

\begin{proposition}
    $\Pcal$ is a free $R$-algebra. There exists a $\Pcal$-module isomorphism
    \[
    \Pcal \otimes_R \Ecal \simeq \Sk_R^n(\Sigma),
    \]
    given by $a \otimes b \mapsto ab$.
\end{proposition}

\begin{corollary}
    $\Sk_R^n(\Sigma)$ is a free $\Pcal$-module and $\Sk_R^n(\Sigma, P, \mathbf{k})$ is a free $R$-module with basis indexed by $\Lambda^{\mathrm{ess}}(\Sigma, \SL_n)$.
\end{corollary}

The algebra morphism
\[
\Pcal \hookrightarrow \Sk_R^n(\Sigma)
\]
corresponds to the boundary morphism in Section \ref{boundarymorphism}. Moreover, one has
\[
\Sk_R^n(\Sigma, P, \mathbf{k}) \simeq \Pcal_\mathbf{k} \otimes_R \Ecal
\]
where $\Pcal_\mathbf{k}$ is the fiber corresponding to $\mathbf{k}$.

\begin{corollary}\label{FlatBoundary}
    Let $\Sigma = \Sigma_{g,m}$ with $\pi = \pi_1 \Sigma = F_{2g+m-1}$. The boundary morphism
    \[
    \SL_n^{2g+m-1} \sslash \SL_n \rightarrow (\SL_n \sslash \SL_n)^m
    \]
    is flat. In particular, $\dim \Sk_R^n(\Sigma, P, \mathbf{k})$ does not depend on the choice of $\mathbf{k}$.
\end{corollary}

\begin{corollary}\label{CMproperty}
    If $\Sk_R^n(\Sigma)$ is Cohen--Macaulay (resp. Gorenstein) then $\Sk_R^n(\Sigma, P, \mathbf{k})$ is also Cohen--Macaulay(resp. Gorenstein).
\end{corollary}

\begin{proof}
    Both $(p_i - k_i)_{i=1}^m$ and $(p_i - k_i, p_i'- k_i')_{i=1}^m$ are regular sequences in $\Pcal$.
\end{proof}

If $R = \C$ then $\Sk_\C^n(\Sigma)$ is Cohen--Macaulay by Hochster--Roberts and so is $\Sk_\C^n(\Sigma, P, \mathbf{k})$.

\subsection{Peripheral skeins in the monoid $\Lambda(\Sigma, \SL_n)$}

Recall the intersection coordinate
\[
\Lambda(\Sigma, \SL_2) \hookrightarrow \Z_{\geq 0}^{-3\chi(\Sigma)}
\]
and the Douglas--Sun coordinate
\[
\Lambda(\Sigma, \SL_3) \hookrightarrow \Z_{\geq 0}^{-8 \chi(\Sigma)}
\]
whose images are determined by the rhombus inequality. Then
\[
K_0(\Lambda(\Sigma, \SL_2)) \leq \Z^{-3 \chi(\Sigma)}
\]
is a sublattice whose image is determined by the \textbf{rhombus congruence conditions}
\begin{quote}
    For any ideal triangle $T = (e_1, e_2, e_3)$, one has
    \begin{align*}
            &i(\Gamma, e_2) + i(\Gamma, e_3) - i(\Gamma, e_1) \in 2\Z, \ \\
            &i(\Gamma, e_3) + i(\Gamma, e_1) - i(\Gamma, e_2) \in 2\Z, \ \\
            &i(\Gamma, e_1) + i(\Gamma, e_2) - i(\Gamma, e_3) \in 2\Z.
        \end{align*}
\end{quote}

Similarly
\[
K_0(\Lambda(\Sigma, \SL_3)) \leq \Z^{-8 \chi(\Sigma)}
\]
is a sublattice whose image is characterized by the \textbf{rhombus congruence conditions}
\begin{quote}
    For any ideal triangle $T = (e_1, e_2, e_3)$, one has
        \begin{align*}
            e_{31} + e_{22} - e &\in 3\Z, \ \\
            e_{32} + e - e_{31} - e_{11} &\in 3\Z, \ \\
            e + e_{21} - e_{22} - e_{12} &\in 3\Z, \ \\
            e_{11} + e_{32} - e &\in 3\Z, \ \\
            e_{12} + e - e_{11} - e_{21} &\in 3\Z, \ \\
            e + e_{31} - e_{32} - e_{22} &\in 3\Z, \ \\
            e_{21} + e_{12} - e &\in 3\Z, \ \\
            e_{22} + e - e_{21} - e_{31} &\in 3\Z, \ \\
            e + e_{11} - e_{12} - e_{32} &\in 3\Z.
        \end{align*}
    \end{quote}

Let $M$ be a commutative monoid and $N \subseteq K_0(M)$ be a subset. We write $M[N]$ for the submonoid of $K_0(M)$ generated by $M$ and $N$.

By abuse of notation, let $\Pcal \leq \Lambda(\Sigma, \SL_n)$ denote the submonoid of peripheral elements. 

\begin{proposition}\label{Groupification}
    \[
    K_0(\Lambda) := K_0(\Lambda(\Sigma, \SL_n)) = \Lambda(\Sigma, \SL_n)[\Pcal^{-1}]
    \]
\end{proposition}

\begin{proof}
    Let $P \in \Pcal$ be the sum of consisting of all irreducible peripherals. This corresponds to
    \[
    \mathbf{p} = (2, 2, \ldots, 2) \in \Lambda(\Sigma, \SL_2)
    \]
    and
    \[
    \mathbf{p} = (6, \ldots, 6, 9, \ldots, 9) \in \Lambda(\Sigma, \SL_3).
    \]
    Here the value $6$ is assigned for all edge coordinates and the value $9$ is assigned for all face coordinates. Then for any $\xb \in K_0(\Lambda)$, if $d>0$ sufficiently large then all of the coordinate of $\xb + d \mathbf{p}$ are nonnegative and satisfies the rhombus inequalities and congruence conditions.
\end{proof}

\begin{definition}
    Let $M$ be a commutative monoid.
    \begin{enumerate}
        \item Let $N \subseteq M$ be a submonoid. $m \in M$ is \textbf{integral over $N$} if $a \cdot m \in N$ for some $a > 0$. The \textbf{integral closure of $N$} in $M$ is the submonoid of elements integral over $N$.
        \item $M$ is \textbf{normal} (also called either \textbf{saturated} or \textbf{integrally closed}) if $M$ coincides with its integral closure in $K_0(M)$.
    \end{enumerate}
\end{definition}

\begin{proposition}\label{NormalMonoid}
    $\Lambda(\Sigma, \SL_n)$ is normal.
\end{proposition}

\begin{proof}
    Any $\xb \in K_0(\Lambda)$ is uniquely decomposed as
    \[
    \xb = \xb_1 + \xb_2
    \]
    where $\xb_1 \in K_0(\Pcal)$ peripheral and $\xb_2 \in \Lambda^{\mathrm{ess}}(\Sigma, \SL_n)$. Thus for any $a>0$, 
    \[
    a \cdot \xb \in \Lambda(\Sigma, \SL_n) \Longleftrightarrow  \xb_1 \in \Pcal \Longleftrightarrow \xb \in \Lambda(\Sigma, \SL_n).
    \]
\end{proof}

\begin{definition}
    Let $M$ be a commutative monoid. A submonoid $F \leq M$ is called a \textbf{face} if for any $a, b \in M$ with $a + b \in F$ then $a, b \in F$.
\end{definition}

\begin{proposition}
    Let $M$ be a normal monoid and let $F \leq M$ a face. Then $F$ is normal.
\end{proposition}

\begin{proof}
    Let $x = m - n \in K_0(F)$ with $m, n \in F$ and $a \cdot x \in F$ for some $a > 0$. Then $x \in M$ since $M$ is normal. Now $x + n = m \in F$ and $x, n \in M$ so $x \in F$.
\end{proof}

\begin{proposition}\label{ConnectedReduced}
    \begin{enumerate}
        \item If $R$ is reduced, then both $\Sk_R^n(\Sigma)$ and $\Sk_R^n(\Sigma, P, \mathbf{k})$ are reduced.
        \item If $R$ is connected, then both $\Sk_R^n(\Sigma)$ and $\Sk_R^n(\Sigma, P, \mathbf{k})$ are connected.
    \end{enumerate}
     
\end{proposition}

\begin{proof}
    The statement for $\Sk_R^n(\Sigma)$ also follows from the classical trace map in Section \ref{skeindomain}.

    Consider $f \in \Sk_R^n(\Sigma)$ and its bangle basis expansion
    \[
    f = \sum_{\xb \in \Lambda(\Sigma, \SL_n)} r_\xb \cdot \tr_\xb
    \]
    with finite support. Denote the \textbf{leading term} of $f$ with respect to the total order on $\Lambda(\Sigma, \SL_n)$ by $\LT(f):=r_{\xb_f} \cdot \tr_{\xb_f}$.

    By the tropical multiplication formula, 
    \[
    \LT(f^N) = \LT((r_{\xb_f} \cdot \tr_{\xb_f})^N) = r 
    _{\xb_f}^N \cdot (\tr_{N\xb_f}).
    \]
    
    Suppose $R$ is reduced. Then $f \in \sqrt{0}_{\Sk_R^n(\Sigma)}$ implies $\xb_f = \emptyset$, which means $f \in \sqrt{0}_R$ so $f = 0$. Thus $\Sk_R^n(\Sigma)$ is reduced.

    Suppose $R$ is connected. If $f$ were an idempotent then $2 \xb_f = \xb_f$. Since $\Lambda(\Sigma, \SL_n)$ is cancellative $\xb_f = \emptyset$ and $f$ is an idempotent of $R$, which will be trivial.

    For $f \in \Sk_R^n(\Sigma, P, \mathbf{k})$ one can use the \textit{essential} bangle basis expansion. Note that if $\xb$ were essential, then so is $a \xb$ for any $a > 0$, by the uniqueness of the peripheral decomposition.
\end{proof}

\begin{remark} \label{LeadingTerm}
    The leading term argument also shows that $\Sk_R^n(\Sigma)$ is a domain. However, it does not ensure $\Sk_R^n(\Sigma, P, \mathbf{k})$ is a domain.

    Let $\Sigma = \Sigma_{0,4}$ and consider $\Lambda(\Sigma, \SL_2)$. Fix a tetrahedral ideal triangulation $\lambda$ of $\Sigma$. Then there exists $3$ separating essential simple closed curves $\gamma_1, \gamma_2, \gamma_3$ which intersect $\lambda$ minimally. Then
    \[
    \gamma_1 + \gamma_2 + \gamma_3 = (2,2,2,2,2,2) = \mathbf{p} \in \Pcal.
    \]
\end{remark}

\begin{corollary}\label{FGtorussubalg}
    Let $k$ be a field. The monoidal torus algebra
    \[
    k[\Lambda] := k[\Lambda(\Sigma, \SL_n)] = \bigoplus_{\mathbf{v} \in \Lambda(\Sigma, \SL_n)} kx^{\mathbf{v}}
    \]
    with multiplication $x^{\mathbf{v}} x^{\mathbf{w}} = x^{\mathbf{v+w}}$, is a normal $k$-algebra.
\end{corollary}

\begin{proof}
    This is a classical result in \cite[Theorem 6.1.4]{bruns1998cohen}.
\end{proof}

\subsection{Hilbert basis and extremal rays}

Let $\Sigma$ be a surface with an ideal triangulation $\lambda$ without self-folded triangles. 

\begin{proposition}
    The collection $\{\tr_\xb \in \Z[\Rep(\pi, \SL_n)]^{\SL_n} : \xb \in \Lambda^{\mathrm{irr}}(\Sigma, \SL_n)\}$ generates $\Z[\Rep(\pi, \SL_n)]^{\SL_n}$.
\end{proposition}
 
\begin{proof}
   Note $\Lambda^{\mathrm{irr}}(\Sigma, \SL_n)$ generates $\Lambda(\Sigma, \SL_n)$. By the tropical multiplication formula, $\{\tr_\xb \in \Z[\Rep(\pi, \SL_n)]^{\SL_n} : \xb \in \Lambda^{\mathrm{irr}}(\Sigma, \SL_n)\}$ generates the filtered algebra
    \[
    \Z[\Rep(\pi, \SL_n)]^{\SL_n} = \bigcup_{\xb \in \Lambda(\Sigma, \SL_n)} \bigoplus_{\xb' \leq \xb} \Z \tr_{\xb'}. 
    \]
\end{proof}

\begin{definition}
    Let $M \subseteq \Z^n$ be a positive rational polyhedral cone. Write
    \[
    \R_+ M := \left\{ \sum a_i \mathbf{v}_i \in \R^n :  a_i \in \R_{\geq 0}, \mathbf{v}_i \in M  \right \}.
    \]
    $m \in M$ is called \textbf{extremal} if it is irreducible and the ray $\R_+ m$ is a face of the monoid $\R_+ M$.
\end{definition}

If $m \in M$ is extremal then it is irreducible but the converse does not hold. Denote the set of extremal element of $\Lambda(\Sigma, \SL_n)$ as $\Lambda^\mathrm{ext}(\Sigma, \SL_n)$.

\begin{example}\label{HilbertBasis}
    Let $\Sigma = \Sigma_{0,3}$ and let $\lambda$ be the ideal triangulation in Figure \ref{Triangulation}. A direct computation shows $\Lambda^{\mathrm{irr}}(\Sigma, \SL_3) = \Lambda^\mathrm{ext}(\Sigma, \SL_3)$ and this set consists of 
    \begin{enumerate}
        \item $6$ peripheral curves $p_i, p_i'$.
        \item $2$ theta webs $W_\theta, W_\theta'$.
        \item $2$ commutator webs $W_{\mathrm{comm}}, W_{\mathrm{comm}}'$ where $W_{\mathrm{comm}}'$ denotes the nonelliptic $3$-web obtained by inverting all orientation of $W_\mathrm{comm}$.
    \end{enumerate}
    Let $\Sigma = \Sigma_{1,1}$ and $\lambda$ be an ideal triangulation in Figure \ref{Triangulation}. Let $\pi = \pi_1 \Sigma = \langle \alpha, \beta \rangle$ where $\alpha, \beta, \alpha \beta$ are $3$ simple closed curves minimally intersecting $\lambda$. By a direct computation $\Lambda^{\mathrm{irr}}(\Sigma, \SL_3) = \Lambda^\mathrm{ext}(\Sigma, \SL_3)$ and it consists of 
    \begin{enumerate}
        \item $6$ simple closed curves $\alpha, \beta, \alpha\beta$ and their inverses.
        \item $2$ theta webs $W_\theta, W_\theta'$, whose restriction on each triangle are degree $\pm 1$ honeycombs.
        \item $2$ commutator curves $[\alpha, \beta]$ and $[\alpha, \beta]^{-1}$.
    \end{enumerate}
\end{example}

\begin{example}
    Using PyNormaliz, once an ideal triangulation of $\Sigma$ is fixed, one can explicitly compute the Hilbert basis of a rational polyhedral cone. The Python code for the computation of the Hilbert basis of $\Lambda(\Sigma_{0,4}, \SL_3)$ is included in Appendix \ref{pythoncode}.

    The complete list of the Hilbert basis of $\Lambda(\Sigma_{0,4}, \SL_3)$ up to symmetry is given in Table \ref{Sigma04HB}. The result shows 
    \[
    \Lambda^\mathrm{irr}(\Sigma_{0,4}, \SL_3) = \Lambda^\mathrm{ext}(\Sigma_{0,4}, \SL_3)
    \]
    and
    \[|\Lambda^{\mathrm{irr}}(\Sigma_{0,4}, \SL_3)| = 100.\]
\end{example}

\begin{question}
    Does $\Lambda^\mathrm{irr}(\Sigma, \SL_n) = \Lambda^\mathrm{ext}(\Sigma, \SL_n)$ always hold?
\end{question}

\begin{table}
\centering
\begin{tabular}{rrll}
\toprule
\# of vectors & Edge coordinates & Face coordinates \\
\midrule
4 & (0,0,0,0,1,2,0,0,1,2,1,2) & (1,1,1,0) \\
4 & (0,0,0,0,2,1,0,0,2,1,2,1) & (2,2,2,0) \\
6 & (0,0,1,2,1,2,1,2,1,2,0,0) & (2,2,1,1) \\
12 & (0,0,1,2,2,1,1,2,2,1,1,2) & (3,3,2,1) \\
12 & (1,2,1,2,2,1,3,3,1,2,1,2) & (4,3,3,2) \\
6 & (1,2,1,2,3,3,3,3,2,1,2,1) & (2,4,4,2) \\
4 & (1,2,1,2,1,2,3,3,3,3,3,3) & (6,2,2,2) \\
4 & (1,2,3,3,3,3,2,1,2,1,3,3) & (4,3,4,4) \\
12 & (0,0,3,3,3,3,3,3,3,3,3,3) & (3,6,3,3) \\
12 & (1,2,1,2,4,2,3,3,3,3,3,3) & (3,5,5,2) \\
12 & (1,2,3,3,3,3,2,4,5,4,3,3) & (7,3,4,4) \\
12 & (1,2,3,3,3,3,5,4,2,4,3,3) & (7,3,4,4) \\
\bottomrule
\end{tabular}
\vspace{6pt}
\caption{The complete list of the Hilbert basis of $\Lambda(\Sigma_{0,4}, \SL_3)$}\label{Sigma04HB}
\end{table}

\section{Compactification of skein algebras}\label{Compactification}

\subsection{Filtered algebras and degree-like functions}

In this subsection we mostly work over $\C$. Our terminology is consistent with \cite{mondal2014projective}.

\begin{definition}
\begin{enumerate}
    \item Let $R$ be a commutative ring and let $A$ be a commutative $R$-algebra as well as a free $R$-module. Let $\Lambda$ be a totally ordered commutative monoid. $A$ is called \textbf{$\Lambda$-filtered} if there exists a family $\Fcal = (F_\lambda)_{\lambda \in \Lambda}$ of submodules satisfying
    \[
    \cdots \subseteq F_0 \subseteq F_{\lambda_\alpha} \subseteq F_{\lambda_\beta} \subseteq \cdots \subseteq \bigcup_{\lambda \in \Lambda} F_\lambda = A
    \]
    with $1 \in F_0$ and $F_\lambda F_{\lambda'} \subseteq F_{\lambda + \lambda'}$. Here the inclusion $F_\lambda \subseteq F_{\lambda'}$ holds whenever $\lambda < \lambda'$.
    \item The \textbf{Rees algebra} associated to a $\Z$-filtered $\C$-algebra $(A, \Fcal)$ is the $\Z_{\geq 0}$-graded algebra
\[
A^\Fcal := \bigoplus_{i \geq 0} F_i t^i \subseteq A[t].
\]
    \item A $\Z$-filtered $\C$-algebra $(A, \Fcal)$ is \textbf{finitely generated} if the associated Rees algebra $A^\Fcal$ is a finitely generated graded algebra. 
    \item $\Fcal$ is \textbf{projective} if $(A, \Fcal)$ is finitely generated and $F_0 = \C$.
\end{enumerate}
    
\end{definition}
Let $(A, \Fcal)$ be a $\Z$-filtered $\C$-algebra. There are natural morphisms
\[
A^\Fcal \rightarrow A^\Fcal / (t-1) \simeq A
\]
and
\[
A^\Fcal \rightarrow A^\Fcal / (t) =: \gr A^\Fcal.
\]
The second morphism preserves grading whose target is called the \textbf{associated graded algebra} of $\Fcal$. 

There is an open immersion
\[
\Spec A \rightarrow \Proj A^\Fcal
\]
given by
\[
\pfrak \mapsto \bigoplus_{i \geq 0} \pfrak \cap F_i.
\]
The complement is a closed immersion
\[
\Proj A^\Fcal / (t) \rightarrow \Proj A^\Fcal.
\]
Moreover, if $(A, \Fcal)$ is finitely generated, then $\Proj A^\Fcal$ is a subvariety of a (generalized) weighted projective space. If $\Fcal$ is projective then $\Proj A^\Fcal$ is a projective variety.

\begin{definition}
    A \textbf{degree-like function} $\delta : A \rightarrow \Z \cup \{ - \infty \}$ is a function satisfying
    \begin{enumerate}
        \item $\delta(0) = - \infty$ and $\delta(\C^\times) = 0$.
        \item $\delta(f+g) \leq \max \{ \delta(f), \delta(g) \}$ and the strict inequality holds only when $\delta(f) = \delta(g)$.
        \item $\delta(fg) \leq \delta(f) + \delta(g)$.
    \end{enumerate}
    $\delta$ is called a \textbf{semi-degree} if it additionally satisfies
    \begin{enumerate}
        \item[($3^{\prime}$)] $\delta(fg) = \delta(f) + \delta(g)$. 
    \end{enumerate}
    $\delta$ is called a \textbf{subdegree} if there exists semi-degrees $\delta_1, \ldots, \delta_r$ such that $\delta = \max\{ \delta_1, \ldots, \delta_r \}$ pointwise.
\end{definition}

\begin{remark}
    \begin{enumerate}
        \item If $A = \bigcup_{i \in \Z} A_i$ is $\Z$-filtered, then the assignment 
        \[
        a \mapsto \inf \{ i : a \in A_i \}
        \]
        is a degree-like function. Conversely, given a degree-like function $\delta$,
        \[
        \Fcal_\delta = (F_i = \{ a : \delta(a) \leq i \})_{i \in \Z}
        \]
        is a $\Z$-filtration.
        \item Let $A$ be a $\Z$-filtered algebra equipped with a degree-like function $\delta$. Then the Rees algebra $A^{\Fcal_\delta}$ is equipped with a degree-like function $\delta_{A^\Fcal} = \max\{0, \delta \}$ where $0$ denotes the trivial valuation $0(A \setminus \{0\}) = 0$. 
        \item If $\delta$ is a semi-degree then $-\delta$ is a valuation on $A$.
    \end{enumerate}
\end{remark}

\begin{lemma}{\cite[Theorem 4.1]{mondal2014projective}}\label{Subdegree}
    Let $(A, \Fcal)$ be a finitely generated $\Z$-filtered algebra. Then $(t) \leq A^\Fcal$ is radical if and only if the corresponding degree-like function $\delta$ is a subdegree.
\end{lemma}

\begin{proof}
    Suppose that $(t)$ is radical. Write a primary decomposition $(t) = \bigcap_{j=1}^s \pfrak_j$ where each $\pfrak_j$ is a minimal prime ideal. We may assume $\pfrak_j$ is not redundant and minimal. One can write $\pfrak_j$ as a conductor ideal $\{ x \in A^\Fcal : (f_j t^{d_j}) \cdot x \in (t) \}$ for some $f_j \in \Fcal_{d_j} \setminus \Fcal_{d_j-1}$.

    For any $f \in A$, define
    \[
    \delta_j(f) = \lim_{k \rightarrow \infty} \delta((f_j)^k f) - \delta( (f_j)^k ).
    \]
    By \cite[Section 2.2]{mondal2010towards}, $\delta_j$ is a semi-degree satisfying $\delta = \max \{ \delta_j \}_j$.

    Conversely, if $\delta = \max\{\delta_1, \ldots, \delta_r \}$ is a subdegree we have $\delta(f^k) = \max \{ k\delta_1(f), \ldots, k\delta_r(f) \} =  k\delta(f)$ for any $f \in A$ and $k \in \Z_{\geq 0}$. Now $ft^d \notin (t)$ implies $\delta(f) = d$ and then $\delta(f^k) = dk$. This implies $(ft^d)^k \notin (t)$ for any $k \in \Z_{\geq 0}$ so $(t)$ is radical.
\end{proof}

\begin{corollary}
    Let $(A, \Fcal)$ be a finitely generated $\Z$-filtered algebra and $\delta$ be the corresponding degree-like function.
    \begin{enumerate}
        \item If $\gr A^\Fcal$ is a domain then $\delta$ is a semi-degree.
        \item If $\gr A^\Fcal$ is reduced then $\delta$ is a subdegree.
    \end{enumerate}
\end{corollary}

\subsection{Toric degeneration}

Let $A = \Sk_\C^n(\Sigma)$, $\Lambda = \Lambda(\Sigma, \SL_n)$ and $K_0(\Lambda)^* = \Hom_\Z(K_0(\Lambda), \Z)$. Write
\[
\left(K_0(\Lambda)^* \right)_{\geq 0} = \{ \varphi \in K_0(\Lambda)^* : \varphi(\xb) \geq 0 \quad \forall \xb \in \Lambda \}.
\]
and
\[
\left(K_0(\Lambda)^* \right)_+ = \{ \varphi \in K_0(\Lambda)^* : \varphi(\xb) > 0 \quad \forall \xb \in \Lambda \setminus \{ 0 \} \}.
\]
Consider
\[
\vec{1} = (1, 1, \ldots, 1)
\]
then the linear functional $\varphi(\xb) = \vec{1} \cdot \xb$ induces a semi-degree by the tropical multiplication formula. By fixing a standard inner product $(-, -)$, we have an injection
\[
\Z^{|V_\lambda^n|} \hookrightarrow K_0(\Lambda)^*
\]
given by $\mathbf{v} \mapsto (\mathbf{v}, -)$
Moreover, since $\Lambda$ is a positive cone we have
\[
(\Z_{\geq 0})^{|V_\lambda^n|} \hookrightarrow (K_0(\Lambda)^*)_{\geq 0}.
\]

\begin{proposition}\label{SemidegreeVectors}
    If $\mathbf{v} \in (\Z_{\geq 0})^{|V_\lambda^n|}$, then $\varphi_\mathbf{v}(\xb) =  \mathbf{v} \cdot \xb$ gives a semi-degree on $A$. If $\varphi_\mathbf{v} \in (K_0(\Lambda)^*)_+$, it induces a projective filtration on $\Sk_\C^n(\Sigma)$.
\end{proposition}

\begin{proof}
    For any $\xb, \xb' \in \Lambda(\Sigma, \SL_n)$ by the tropical multiplication formula
    \[
    \tr_{\xb} \tr_{\xb'} = \pm \tr_{\xb+\xb'} + \sum_{\vec{0} \neq \mathbf{d} \in \Def(\Sigma, \SL_n)} r_{\mathbf{d}}\tr_{\xb + \xb' - \mathbf{d}}.
    \]
    
    By Propositions \ref{SL2defect} and \ref{SL3defect}, we have
    \[
    \Def(\Sigma, \SL_n) \subseteq (n\Z_{\geq 0})^{|V_\lambda^n|}
    \]
    and note
    \[
    (\Z_{\geq 0})^{|V_\lambda^n|} \subseteq ((n\Z)^*)_{\geq 0} \cap (K_0(\Lambda)^*)_{\geq 0}.
    \]
    Thus $\mathbf{v} \cdot (\xb + \xb' - \mathbf{d}) \leq \mathbf{v} \cdot (\xb + \xb')$ for any $\mathbf{v} \in (\Z_{\geq 0})^{|V_\lambda|}$, $\xb, \xb' \in \Lambda$ and $\mathbf{d} \in \Def(\Sigma, \SL_n)$.

    It remains to check that the filtration $\Fcal$ corresponding to the semidegree is projective. If $\mathbf{v} \in (K_0(\Lambda)^*)_+$ then $\Fcal_0 = \C$. $A^\Fcal$ is generated by the collection
    \[
    \{ \tr_{\xb} t^{\varphi_{\mathbf{v}}(\xb)} : \xb \in \Lambda^\mathrm{irr} (\Sigma, \SL_n) \} \cup \{ t \}
    \]
    which is finite.
\end{proof}

\begin{remark}\label{ProjectiveVector}
    $\mathbf{v} \in (\Z_{> 0})^{|V_\lambda^n|}$ is not a necessary condition for $\mathbf{v} \in (K_0(\Lambda)^*)_+$. 
    Let $\mathbf{v} \in (\Z_{\geq 0})^{|V_\lambda^3|}$ be the vector
    \[
    \mathbf{v} = (v_i) = 
    \begin{cases}
        1 & (v_i \text{ is a face coordinate})\\
        0 & (v_i \text{ is an edge coordinate}).
    \end{cases}
    \]
    By Figure \ref{DScoordfigure}, one can check $\varphi_{\mathbf{v}}(\xb) = 0$ if and only if $\xb = 0$.
\end{remark}

Denote the Rees algebra corresponding to a $\Z$-filtration as $A^{\Fcal_\mathbf{v}}$ where $\mathbf{v} \in (\Z_{\geq 0})^{|V_\lambda|}$.

\begin{definition}\label{TorDegen}
    Let $\mathbf{v} \in (\Z_{\geq 0})^{|V_\lambda^n|} \cap (K_0(\Lambda)^*)_+$, then $(A, \Fcal_\mathbf{v})$ is projective with a semi-degree. The $\Z$-graded algebra $\gr A^{\Fcal_\mathbf{v}}$ is called the \textbf{toric degeneration} of $A = \Sk_R^n(\Sigma)$. If $\mathbf{v} \in (\Z_{>0})^{|V_\lambda|}$, we call that the toric degeneration is \textbf{strict}.
\end{definition}

\begin{remark}
    If the toric degeneration is strict, then the choice of $\mathbf{v}$ will be nearly auxiliary as far as Theorem \ref{theoremA} and Theorem \ref{theoremB} concerns. We will omit the subscript $\mathbf{v}$ and write $\varphi_{\mathbf{v}}(\xb) = |\xb|$. We discuss Theorem \ref{theoremA} and Theorem \ref{theoremB} for a non-strict toric degeneration in Section \ref{NonStrictToricDegeneration}.
\end{remark}

\subsection{Normal Gorenstein compactifications of skein algebras}

From now on, assume $R = \C$, $A = \Sk_\C^n(\Sigma)$. Consider a strict toric degeneration $\gr A^\Fcal$ of the skein algebra $\Sk_\C^n(\Sigma)$. We identify
\[
\gr A^\Fcal \simeq \C[\Lambda], \quad \tr_{\mathbf{v}} \mapsto x^{\mathbf{v}}
\]
in Corollary \ref{FGtorussubalg} as a $\Z$-graded algebra with $\C[\Lambda]_0 = \C$. Note $\C[\Lambda]$ is a graded local ring.

Let $\Lambda^\circ := \Lambda \cap (\R_+ \Lambda)^\circ$ where $\circ$ denotes the interior with respect to the relative Euclidean topology.

\begin{proposition}{\cite[Section 6.3]{bruns1998cohen}}
    Let $k$ be a field and $C$ be a normal submonoid of $\Z^n$. $k[C]$ is Gorenstein if and only if there exists $\mathbf{v} \in C^\circ$ such that $C^\circ = \mathbf{v} + C$, whose canonical module is $\omega_{k[C]} \simeq k[C](-\deg \mathbf{v})$.
\end{proposition}

Let $\mathbf{p} \in \Lambda(\Sigma, \SL_n)$ be the sum of all irreducible peripheral elements in Proposition \ref{Groupification}.

\begin{theorem}
    $\C[\Lambda]$ is Gorenstein with canonical module $\omega \simeq \C[\Lambda](-|\mathbf{p}|)$.
\end{theorem}

\begin{proof}
    The set $\Lambda^\circ$ is determined by the strict part of the rhombus inequality
    \begin{align*}
        &i(\Gamma, e_2) + i(\Gamma, e_3) - i(\Gamma, e_1) \in 2\Z_{> 0}, \ \\
        &i(\Gamma, e_3) + i(\Gamma, e_1) - i(\Gamma, e_2) \in 2\Z_{> 0}, \ \\
        &i(\Gamma, e_1) + i(\Gamma, e_2) - i(\Gamma, e_3) \in 2\Z_{> 0}
    \end{align*}
    or
    \begin{align*}
            e_{31} + e_{22} - e &\in 3\Z_{> 0}, \ \\
            e_{32} + e - e_{31} - e_{11} &\in 3\Z_{> 0}, \ \\
            e + e_{21} - e_{22} - e_{12} &\in 3\Z_{> 0}, \ \\
            e_{11} + e_{32} - e &\in 3\Z_{> 0}, \ \\
            e_{12} + e - e_{11} - e_{21} &\in 3\Z_{> 0}, \ \\
            e + e_{31} - e_{32} - e_{22} &\in 3\Z_{> 0}, \ \\
            e_{21} + e_{12} - e &\in 3\Z_{> 0}, \ \\
            e_{22} + e - e_{21} - e_{31} &\in 3\Z_{> 0}, \ \\
            e + e_{11} - e_{12} - e_{32} &\in 3\Z_{> 0}
        \end{align*}
    from Section \ref{IntersectionCoordinate} and Theorem \ref{DScoordinate}. It is clear that $\mathbf{p} \in \Lambda^\circ$ and $\mathbf{p} + \Lambda \subseteq \Lambda^\circ$.
    
    We will prove the converse inclusion. If $n=2$, the value
    \[
    \frac 12 \left[ i(\Gamma, e_{i_1}) + i(\Gamma, e_{i_2}) - i(\Gamma, e_{i_3}) \right]
    \]
    equals the number of corner arcs at the angle opposite to $e_{i_3}$. Thus for any $\xb \in \Lambda^\circ$ the local $\SL_2$-train track configuration contains at least one corner arc at every angle. Thus $\xb = \mathbf{p} + (\xb - \mathbf{p})$ is a suitable decomposition in $\Lambda$.

    If $n=3$, referring to Figure \ref{DScoordfigure}, check the contribution of each part of a local $\SL_3$-train track to the rhombus illustrated in Figure \ref{PositiveDSdiagram}. Note that one cannot have both of honeycomb webs $H_{\pm 1}$ simultaneously. To satisfy every strict inequality of the rhombus inequality, the local $\SL_3$-train track configuration contains at least one left/right turn at every angle. Thus $\xb = \mathbf{p} + (\xb - \mathbf{p})$ is a suitable decomposition in $\Lambda$.
    
    \begin{figure}
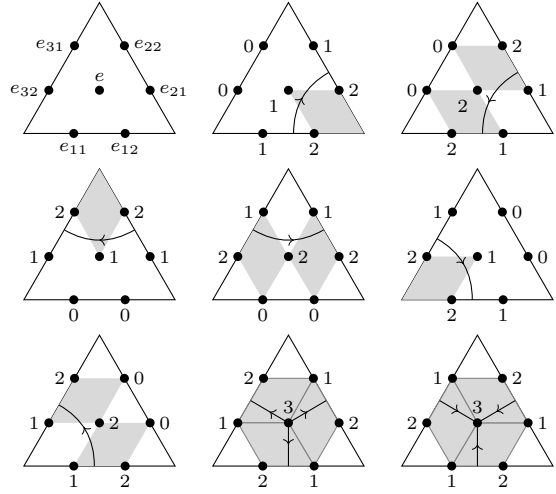

        \centering
        \coloredDSdiagram
        \caption{Positive contributions of local $\SL_3$-train tracks.}
        \label{PositiveDSdiagram}
    \end{figure}
\end{proof}

\begin{corollary}\label{EssentialsFace}
    Let $\Sigma = \Sigma_{g,m}$ and $\Gamma$ (resp. $W$) be an essential element of $\Lambda(\Sigma, \SL_2)$ (resp. $\SL_3$). Then at least $m$(resp. $2m$) of the rhombus inequalities becomes an equality.
\end{corollary}

\begin{proof}
    Let $p_i$ be a peripheral skein in $\Sk_R^2(\Sigma)$. There is an ideal triangle whose local $\SL_3$-train track configuration does not contain a corner arc at that puncture. For each puncture, the corner arcs are disjoint. The proof for $n=3$ is analogous, using Figure \ref{PositiveDSdiagram}.
\end{proof}

\begin{corollary}\label{ReesGorenstein}
    $A^\Fcal$ is Gorenstein with canonical module $\omega \simeq A^\Fcal(-1-|\mathbf{p}|)$.
\end{corollary}

\begin{proof}
    $\Fcal$ is a projective filtration and $t \in A^\Fcal$ is a homogeneous regular element of degree $1$. As $\C[\Lambda] = A^\Fcal / (t)$ is Gorenstein, $A^\Fcal$ is Cohen--Macaulay, graded local and its canonical module exists. By \cite[Corollary 3.6.14]{bruns1998cohen} the result follows.
\end{proof}

\begin{lemma}\label{TorDegenDim}
    $\dim \C[\Lambda] = \dim A$.
\end{lemma}

\begin{proof}
    We have that $\Lambda$ is of full rank, hence
    \[
    \dim \C[\Lambda] = \dim_\R K_0(\Lambda) \otimes \R = |V_\lambda| = \dim A.
    \]
\end{proof}

\begin{theorem}\label{AbsNormalDomain}
    $A$ and $A^\Fcal$ are normal domains.
\end{theorem}

\begin{proof}
    We will make use of \cite[Lemma 2.3]{whang2020global}. By Lemma \ref{TorDegenDim}, we have $\gr A^\Fcal \simeq \C[\Lambda]$ is reduced of dimension $\dim A^\Fcal - 1$. It suffices to check that $A$ is normal. Let $\delta$ be the semi-degree induced by the filtration $\Fcal$, which extends to a semi-degree on $\Frac(A)$ by defining $\delta(f/g) = \delta(f) - \delta(g)$. By abuse of notation, define $\LT(f) = r_{\mathbf{v}}x^{\mathbf{v}} \in \C[\Lambda]$ where $\LT(f) = r_{\mathbf{v}} \tr_{\mathbf{v}} \in A$. Then $\LT$ extends to $\Frac A$ by $\LT(f/g) = \LT(f) / \LT(g) \in \Frac(\C[\Lambda])$.
    
    It suffices to check that $A$ is completely normal. Suppose $x \in K =\Frac(A)$ is almost integral over $A$. Then there exists a nonzero $a \in A$ such that $a x^N \in A$ for any $n \geq 0$. Then unless $x = 0$,
    \[
    \delta(a) + N \cdot \delta(x) \geq 0
    \]
    then $\delta(x) \geq 0$ should hold. 
    
    Let $I$ be the conductor ideal of $x$
    \[
    I = \{ y \in A : y \cdot A[x] \subseteq A \} \leq A
    \]
    then $a \in I$ and $xI \subseteq I$. Consider a graded ideal
    \[
    J = ( \LT(f) : f \in I ) \leq \C[\Lambda].
    \]
    Write $x = f/g$. For any $y \in I$, we have $xy \in A$. Thus
    \[
    \LT(xy) = \LT \left( \frac{fy}{g} \right) = \frac{\LT(f)\LT(y)}{\LT(g)} = \LT(x)\LT(y).
    \]
    Any $j \in J$ can be written as
    \[
    j = \sum_i \LT (y_i)
    \]
    for some $y_i \in I$ and then
    \[
    \LT(x) \cdot j = \sum_i \LT(x y_i) \in J
    \]
    so $\LT(x) \cdot J \subseteq J$. Since $J$ is finitely generated, $\LT(x)$ is integral over $\C[\Lambda]$ by Cayley-Hamilton. Since $\C[\Lambda]$ is normal, we have $\LT(x) \in \C[\Lambda]$.

    Choose $x_0 \in A$ such that $\LT(x) = \LT(x_0)$. Then $\delta(x - x_0) < \delta(x)$ and $x - x_0$ is almost integral over $A$. Iterating the process, for
    \[
    0 \leq \cdots < \delta(x - x_0 - x_1) < \delta(x-x_0) < \delta(x)
    \]
    the process terminates and $x = \sum_i x_i \in A$.
    
\end{proof}

\subsection{Log-CY compactification of relative skein algebras}\label{relcompactification}

\begin{definition}
    Let $A = \Sk_\C^n(\Sigma)$. Using the notation in Definition \ref{relativeskein},
    \begin{enumerate}
        \item If $n=2$, the graded algebra
        \[
        A^\Fcal(P, \mathbf{k}) := A^\Fcal / ((p_i - k_i) t^{|p_i|})_{i=1}^m
        \]
        is called the \textbf{relative Rees algebra} of $A^\Fcal$ with respect to $(P, \mathbf{k})$.
        \item If $n=3$, the graded algebra
        \[
        A^\Fcal(P, \mathbf{k}) := A^\Fcal / ((p_i - k_i) t^{|p_i|}, (p'_i - k'_i) t^{|p'_i|})_{i=1}^m
        \]
        is called the \textbf{relative Rees algebra} of $A^\Fcal$ with respect to $(P, \mathbf{k})$.
    \end{enumerate}
\end{definition}

Let $A^\mathrm{rel} = \Sk_\C^n(\Sigma, P, \mathbf{k})$. The quotient $A \rightarrow A^\mathrm{rel}$ induces the quotient filtration $\Fcal'$ on $A^\mathrm{rel}$. The relative Rees algebra is isomorphic to the Rees algebra associated with $\Fcal'$. The associated graded algebra $\gr A^\Fcal (P, \mathbf{k}) \simeq \gr (A^\mathrm{rel})^{\Fcal'}$ is isomorphic to the quotient
\[
\C[\Lambda] / (x^{p_i})_{i=1}^m
\]
if $n=2$ and
\[
\C[\Lambda] / (x^{p_i}, x^{p'_i})_{i=1}^m
\]
if $n=3$.

\begin{theorem}\label{RelGorenstein}
    $A^\Fcal(P, \mathbf{k})$ is Gorenstein with canonical module $\omega \simeq A^\Fcal(P, \mathbf{k})(-1)$. In particular, Theorem \ref{theoremC} holds.
\end{theorem}

\begin{proof}
    Note $\sum_{i=1}^m |p_i| = |\mathbf{p}|$ if $n=2$ and $\sum_{i=1}^m (|p_i| + |p'_i|) = |\mathbf{p}|$ if $n=3$. We prove the case for $n=2$ and the case for $n=3$ is analogous. Again by \cite[Corollary 3.6.14]{bruns1998cohen}, it suffices to check the sequence
    \[
    ((p_i - k_i) t^{|p_i|})_{i=1}^m 
    \]
    is regular in $A^\Fcal$. Again it suffices to check
    \[
    (t, (p_1 - k_1)t^{|p_1|}, (p_2 - k_2)t^{|p_2|}, \ldots, (p_m - k_m)t^{|p_m|})
    \]
    is regular. This is equivalent to showing that the sequence
    \[
    (p_1, p_2, \ldots, p_m)
    \]
    is regular in $\gr A^{\Fcal} \simeq \C[\Lambda]$. By the peripheral decomposition $\C[\Lambda]$ is a free $\C[\Pcal]$-module generated by $\Lambda^\mathrm{ess}(\Sigma, \SL_2)$. $(p_1, p_2, \ldots, p_m)$ is a free variable generating a free commutative algebra $\C[\Pcal]$ so regular. The reciprocity of the generating function is followed by Stanley \cite{stanley1996combinatorics}. Observe $\dim A^\mathrm{rel}$ is even.
\end{proof}

\begin{corollary}\label{RelDim}
    $\dim \gr A^\Fcal(P, \mathbf{k}) = \dim A^\mathrm{rel}$.
\end{corollary}

\begin{proposition}\label{quotientfiltrationsubdegree}
    Let $\Fcal$ be a filtration on $A$ inducing a semi-degree. Then the quotient filtration $\Fcal'$ on $A^\mathrm{rel}$ is induced by a subdegree and projective.
\end{proposition}

\begin{proof}
    Since $\Fcal$ is finitely generated, so is $\Fcal'$. Let $\delta$ be the degree-like function corresponding to $\Fcal'$. By Lemma \ref{Subdegree}, it suffices to check $\delta(f^k) = k\delta(f)$ for any $f \in A^\mathrm{rel}$ and $k>0$. We may assume $f = \tr_\xb$ where $\xb \in \Lambda^\mathrm{ess}(\Sigma, \SL_n)$ and the result follows from the proof of Proposition \ref{ConnectedReduced}.
\end{proof}

Let $I \leq \C[\Lambda]$ be a monomial ideal. Define 

    \[
    (\log I)^c = \{ \mathbf{v} \in \Lambda(\Sigma, \SL_n) : x^{\mathbf{v}} \notin I \}.
    \]

If $I = \pfrak$ is a monomial prime ideal, $(\log \pfrak)^c$ is a submonoid and face of $\Lambda(\Sigma, \SL_n)$. Conversely, if $F \subseteq \Lambda$ is a face then the monomial ideal $(x^\mathbf{v} : \mathbf{v} \notin F )$ is prime. There exists an order-reversing correspondence
    \[
    I \subseteq J \Longleftrightarrow (\log J)^c \subseteq (\log I)^c \Longleftrightarrow I \cap (\log J)^c = \emptyset. \label{PrimFaceCorr}
    \] 
\ \\

Consider $\gr A^\Fcal(P, \mathbf{k}) \simeq \C[\Lambda] / (x^{p_i})_{i=1}^m$ or $\C[\Lambda] / (x^{p_i}, x^{p_i'})_{i=1}^m$  which is a noetherian graded algebra. Write $I = (x^{p_i})_{i=1}^m$ or $I = (x^{p_i}, x^{p'_i})_{i=1}^m$.\label{idealIdef}

By Proposition \ref{quotientfiltrationsubdegree}, $I$ is a radical ideal. Consider a primary decomposition 

\[
I = \bigcap_{j=1}^s \pfrak_j
\]

by minimal prime ideals containing $I$. We may assume $\pfrak_j$s are not redundant and minimal. Then $(\log \pfrak_j)^c$ is a maximal face avoiding the ideal $I$. By the isomorphism $\C[\Lambda] / \pfrak_j \simeq \C[(\log \pfrak_j)^c]$,  $\C[\Lambda]/{\pfrak_j}$ is a normal domain.

\begin{remark}
    The associated graded algebra $\gr A^\Fcal(P, \mathbf{k})$ is not a domain, as one sees in Remark in Section \ref{LeadingTerm}. However one has an embedding
    \[
    \C[\Lambda] / I \hookrightarrow \prod_{j=1}^s \C[\Lambda] / \pfrak_j
    \]
    into a normal ring. $\C[\Lambda] / I$ has the structure of \textit{toric face ring}, \cite{ichim2007toric}.
\end{remark}

\begin{theorem}\label{RelNormalDomain}
    $A^\mathrm{rel}$ and $A^\Fcal(P, \mathbf{k})$ are normal domains.
\end{theorem}

\begin{proof}
     Observe the following:
    \begin{enumerate}
        \item $\gr A^\Fcal(P, \mathbf{k}) \hookrightarrow \prod_{j=1}^s \C[\Lambda] / \pfrak_j$ so $\gr A^\Fcal(P, \mathbf{k})$ is reduced.
        \item By Corollary \ref{RelDim}, $\gr A^\Fcal(P, \mathbf{k})$ has dimension $\dim A^\Fcal(P, \mathbf{k}) - 1$.
        \item By Proposition \ref{ConnectedReduced}, $A^\mathrm{rel}$ is connected and so is $A^\Fcal(P, \mathbf{k})$.
    \end{enumerate} 
    
    Again by \cite[Lemma 2.3]{whang2020global} it suffices to check that $A^\mathrm{rel}$ is normal. By Corollary \ref{CMproperty}, $A^\mathrm{rel}$ is Cohen--Macaulay so it suffices to check that $A^\mathrm{rel}$ is regular in codimension $1$. We use the notations and the results introduced in Appendix \ref{appendixa}. Following Whang \cite{whang2020global}, we will estimate dimension of the singular locus of $A^\mathrm{rel}$ denoted by $\sing (A^\mathrm{rel})$. It suffices to show where $n=3$, we can write
    \[
    \sing (A^\mathrm{rel}) = \sing (A^\mathrm{rel})_\mathrm{irr} \sqcup \sing (A^\mathrm{rel})_\mathrm{red}
    \]
    where the semisimplification of $[\rho] \in \sing(A^\mathrm{rel})_\mathrm{irr}$ is irreducible and $[\rho] \in \sing(A^\mathrm{rel})_\mathrm{red}$ is reducible.
    
    Reducible representations in $X(\pi, \SL_3)_\C$ arises from the image of $\GL(2, \C)$-representations given by the morphism
    \[
    \Rep(\pi, \GL_2)_\C \rightarrow \Rep(\pi, \SL_3)_\C
    \]
    given by $\rho \oplus \det \rho^{-1}$ which is compatible with the GIT quotient. By Lemma \ref{finiteboundarychoice}, it suffices to compute the dimension of the $\GL_2$-character variety $X^\mathrm{rel}(\pi, \GL_2, \mathbf{k})_\C$ to bound $\dim \sing(A^\mathrm{rel})_\mathrm{red}$.
    By Lemma \ref{GL2reldimension}, we have
    \[
    \dim \sing(A^\mathrm{rel})_\mathrm{red} \leq 8g + 2m - 6
    \]
    and
    \[
    \codim \sing(A^\mathrm{rel})_\mathrm{red} \geq (16g + 6m - 16) - (8g + 2m - 6) = -2 - 4\chi(\Sigma).
    \]
    Thus for any hyperbolic punctured surface, $\chi(\Sigma) \leq -1$ implies $\codim \sing(A^\mathrm{rel})_\mathrm{red} \geq 2$.

     By Lemma \ref{boundarydifferential}, if $\rho \in \Rep(\pi, \SL_3)_\C$ is irreducible then the morphism $(\SL(3, \C))^{2g} \times \prod_{i=1}^m \mathcal{C}_i \rightarrow \SL(3, \C)$ is smooth. If all of the boundary monodromies are regular, the differential $\SL(3, \C) \rightarrow \C^2 \simeq \SL(3, \C) \sslash \SL(3, \C)$ is also surjective and $[\rho] \in X(\pi, \SL_3, \mathbf{k})$ is a smooth point. By Lemma \ref{regularcodim2} we have 
     \[
     \dim \sing(A^\mathrm{rel})_\mathrm{irr} \leq 8(2g - 1) + \sum_{i=1}^m \dim \mathcal{C}_i - 8 \leq 16g + 6m - 18
     \]
     and 
     \[
     \codim \sing(A^\mathrm{rel})_\mathrm{irr} \geq (16g + 6m - 16) - (16g + 6m - 18) = 2.
     \] 
     Hence $A^\mathrm{rel}$ is normal regardless of the choice of $\mathbf{k}$.

\end{proof}

\begin{definition}[Watanabe]
    Let $A$ be a normal $\Z_{\geq 0}$-graded ring with $A_0 = \C$ and $\mfrak = A_+$. The \textbf{$a$-invariant} of $A$ is
    \[
    a(A) := \max \{ n : (H^{\dim A}_\mfrak(A))_n \neq 0 \} = - \min \{ n : (\omega_A)_n \neq 0 \}
    \]
    where $H_\mfrak^\bullet(-) \simeq \varinjlim \Ext^\bullet(A / \mfrak^i, -)$ is a local cohomology and $\omega_A$ denotes the canonical module of $A$.
\end{definition}

\begin{definition}
    A projective variety $X$ is \textbf{log Calabi--Yau(log-CY)} if there exists a reduced divisor $D$ such that $K_X + D \sim_\Q 0$.
\end{definition}

\begin{remark}
    \begin{enumerate}
        \item Sometimes we need $D$ to be $\Q$-Cartier. For $X = \Proj A^\Fcal(P, \mathbf{k}), D = \Proj \gr A^\Fcal(P, \mathbf{k})$,
    $D$ is reduced so $D \sim \mathrm{div}(t)$ and $t \in H^0(X, \Ocal_X(1))$ where $X \hookrightarrow \mathbb{P}$ for some weighted projective space $\mathbb{P}$. Since $\Ocal_\mathbb{P}(n)$ is invertible for some sufficiently large $n$ and so is $\Ocal_X(n)$, $D$ is $\Q$-Cartier. 
        \item Sometimes we need $D$ to have simple normal crossings. This is not always true as examples in Section \ref{example} shows. Smoothness of each irreducible boundary divisor is also not concerned in this paper.
    \end{enumerate}
\end{remark}

Whang \cite{whang2020global} used the following theorem to prove the log-CY compactification of relative $\SL(2, \C)$-character variety exists. 

\begin{theorem}[Watanabe, \cite{watanabe1981some}]\label{Watanabe}
    Let $R$ be a normal positively graded Gorenstein domain with $R_0 = \C$ and $t \in R_1$ be a homogeneous nonzerodivisor. If $a(R) = -1$ and $R / (t)$ is reduced then $(Z = \Proj R,  \Delta = \Proj R/(t) )$ is a log-CY pair.
\end{theorem}

\begin{proof}{Proof of Theorem \ref{theoremA}}
    By Theorem \ref{RelGorenstein} and Theorem \ref{RelNormalDomain}, $A = A^\Fcal(P, \mathbf{k})$ satisfies the hypothesis of Theorem \ref{Watanabe}.
\end{proof}

\subsection{Weak geometric P=W conjecture}

Let $\mathbf{v} \in (K_0(\Lambda)^*)_+$. Then
\[
\mathbf{v} \cdot (\R_+ \Lambda \setminus \{0 \}) = \R_+
\]
and the radial projection
\[
\pi : \Lambda \setminus \{ 0 \} \rightarrow H_\mathbf{v} = \left\{ \xb \in \R^{|V_\lambda^n|} : \mathbf{v} \cdot \xb = 1 \right\}
\]
is well-defined and extends to $\R_+ \Lambda \setminus \{0 \}$. The point of $\pi(\Lambda \setminus \{ 0 \})$ corresponds to preprimitive vectors in $\Lambda$.
There exists a polyhedral cell complex structure given by faces of the cone $\R_+ \Lambda$. The $0$-skeleton of the complex corresponds to extremal vectors in $\Lambda$.

\begin{definition}
    The polyhedral complex $\mathbb{P} \Lambda := \pi(\R_+\Lambda \setminus \{0 \})$ is called the \textbf{moment complex} of $\Lambda$.
\end{definition}

\begin{remark}
    \begin{enumerate}
        \item As the definition says the polyhedral structure does not depend on the choice of $\mathbf{v}$.
        \item If $F \subseteq \R_+\Lambda$ were a face, $\pi(F) \subseteq \mathbb{P} \Lambda$ is a face of the polyhedral complex of dimension
        \[
        \dim \pi(F) = \dim F - 1.
        \]
        \item Write $\Lambda^\mathrm{ext}(\Sigma, \SL_3) = \{ \mathbf{v}_1, \ldots, \mathbf{v}_l \}$.  Then any $\xb \in \mathbb{P}\Lambda$ can be written as
        \[
        \xb = \sum_{i=1}^l a_i\mathbf{v}_i = \sum_{i=1}^l a_i(\mathbf{v} \cdot \mathbf{v}_i) \frac{\mathbf{v}_i}{(\mathbf{v} \cdot \mathbf{v}_i)}
        \]
        where $a_i \geq 0$ and $\sum_{i=1}^l a_i (\mathbf{v} \cdot \mathbf{v}_i) = 1$.
        Thus $\mathbb{P}\Lambda$ is the convex hull of its $0$-skeleton. 
    \end{enumerate}
\end{remark}

Recall that we have an order-reversing correspondence
\[
\left \{ \text{monomial prime ideal of $\C[\Lambda]$} \right \} \longleftrightarrow \left \{ \text{face of $\Lambda$}\right \}.
\]

Consider $\C[\Lambda] / I$ in Section \ref{idealIdef}. The primary decomposition $I = \bigcap_{j=1}^s \pfrak_j$ by minimal prime ideal gives a collection of maximal faces of $\mathbb{P} \Lambda$ avoiding the polyhedral complex generated by peripheral elements. Then we have a correspondence
\[
\left \{ \text{irreducible component of $\partial A^\mathrm{rel}$} \right \} \leftrightarrow \left \{ \text{minimal prime ideal of $\C[\Lambda]$} \right \} \leftrightarrow \left \{ \text{maximal face of $\mathbb{P}\Lambda$ avoiding peripherals}\right \}.
\]
Note every minimal prime ideal of $\C[\Lambda]$ over a monomial ideal is monomial. If $I, J$ are monomial ideals then $I + J$ is also monomial ideal. Intersections of irreducible components $\partial A^\mathrm{rel}$ corresponds to a face of $\mathbb{P} \Lambda$ avoiding peripherals.

\begin{definition}
    Let $\{D_j\}_{j=1}^s$ be irreducible components of a strict toric degeneration $\Proj \gr A^\Fcal(P, \mathbf{k})$. The simplicial complex $\mathbb{D}(\partial A^\mathrm{rel})$ defined by $\{ j_1, \ldots, j_k \} \in \mathbb{D}(\partial A^\mathrm{rel})$
    if $\bigcap_{i=1}^k D_{j_i} \neq \emptyset$, is called the \textbf{dual complex} of $A^\mathrm{rel}$.
\end{definition}

\begin{remark}
    The correspondence between faces and irreducible boundary divisors shows that two irreducible boundary divisors can only meet at an irreducible subvariety. In this sense the dual complex is compatible with the traditional definition. This can be compared with example \ref{relskeinexample}, which is obtained by a non-strict toric degeneration.
\end{remark}

In summary, we have an isomorphism of simplicial complexes
\[
\mathbb{D}(\partial A^\mathrm{rel}) \simeq \left \{ \text{Nerve of the maximal faces of $\mathbb{P} \Lambda$ avoiding peripherals} \right \}.
\]
By the nerve theorem, $\mathbb{D}(\partial A^\mathrm{rel})$ is homotopy equivalent to the union of maximal faces avoiding peripherals. \ \\

Recall that any $\xb \in K_0(\Lambda)$ admits a peripheral decomposition $\xb = \xb_1 + \xb_2$ where $\xb_1 \in K_0(\Pcal)$ peripheral and $\xb_2$ essential. 

\begin{lemma}
    The peripheral decomposition extends to $\R_+ \Lambda$ piecewise linearly.
\end{lemma}

\begin{proof}
    For $\Lambda(\Sigma, \SL_2)$, for any corner arc $\gamma$ at a puncture $p \in P$ define the \textbf{local peripheral vector} $\mathbf{v}_{p, \gamma}\in \R^{|V_\lambda^2|}$ by
\[
\mathbf{v}_{p, \gamma}|_T= (e_1, e_2, e_3) = \left( \frac12, \frac12, -\frac12 \right) \text{ if $e_3$ is the opposite side of $\gamma$} \subseteq T
\]
and all other components are set to $0$. Denote the collection of corner arcs at $p$ by $\Psi_p$. Define $m_p : \R_+ \Lambda \rightarrow \R_{\geq 0}$ by
\[
m_p(\xb) = \min_{\gamma \in \Psi_p} \{ \mathbf{v}_{p, \gamma} \cdot \xb \}. 
\]
Set $C^\mathrm{ess} = \{ \xb \in \R_+ \Lambda : m_p(\xb) = 0 \text{ for all } p \in P \}$.
Then we have the peripheral decomposition $\R_+ \Lambda \rightarrow (\R_{\geq 0})^m \times C^\mathrm{ess}$ defined by
\[
\mathbf{x} \mapsto \left( \sum_{p \in P}  m_p(\xb) p, \xb - \sum_{p \in P}  m_p(\xb) p \right)
\]
which is a piecewise linear extension of the peripheral decomposition in $\Lambda$.

For $\Lambda(\Sigma, \SL_3)$, let $p, p' \in \Pcal$ be a clockwise puncture or a counterclockwise puncture. For any right turn $\gamma$ contributing to $p$ or a left turn $\delta$ contributing to $p'$, define the local peripheral vector in $\Z^{|V_\lambda^3|}$ as
\begin{align*}
    \mathbf{v}_{p, \gamma}|_T = (e_{11}, e_{12}, e_{21}, e_{22}, e_{31}, e_{32}, e) = \left( 0,  \frac13, \frac13, 0, 0, 0, -\frac13 \right) \text{ if $\gamma = r_{12} \subseteq T$ in Figure } \ref{DScoordfigure}, \\
    \mathbf{v}_{p', \delta, 1}|_T = (e_{11}, e_{12}, e_{21}, e_{22}, e_{31}, e_{32}, e) = \left(  \frac13, -\frac13, 0, 0, 0, -\frac 13,  \frac13 \right) \text{ if $\delta = l_{21} \subseteq T$ in Figure } \ref{DScoordfigure}, \\
    \mathbf{v}_{p', \delta, 2}|_T = (e_{11}, e_{12}, e_{21}, e_{22}, e_{31}, e_{32}, e) = \left(  0, 0, -\frac 13, \frac13, -\frac13, 0,  \frac13 \right) \text{ if $\delta = l_{21} \subseteq T$ in Figure } \ref{DScoordfigure},
\end{align*}
and the other components are defined to be $0$. Denote the collection of the right turns at $p$ by $\Psi_{p}$ and analogously define $\Psi_{p'}$. Define $m_p : \R_+ \Lambda \rightarrow \R_{\geq 0}$ by
\[
m_p(\xb) = \min_{\gamma \in \Psi_p} \{ \mathbf{v}_{p, \gamma} \cdot \xb \}
\]
and analogously define $m_{p'}$ by
\[
m_{p'}(\xb) = \min_{\delta \in \Psi_{p'}} \left\{ \mathbf{v}_{p', \delta, 1} \cdot \xb, \mathbf{v}_{p', \delta, 2} \cdot \xb \right\}.
\]
Write $C^\mathrm{ess} = \{ \xb \in \R_+ \Lambda : m_{p}(\xb) = m_{p'}(\xb) = 0 \text{ for all } p, p' \in \Pcal \}$. Then we have the peripheral decomposition $\R_+\Lambda \rightarrow (\R_{\geq 0})^{2m} \times C^\mathrm{ess}$ defined by
\[
\xb \mapsto \left( \sum_{p, p'} \left[m_p(\xb)p + m_{p'}(\xb)p'\right], \xb - \sum_{p, p'} \left[m_p(\xb)p + m_{p'}(\xb)p'\right] \right).
\]

\end{proof}

 After some polyhedral subdivison, the PL map gives an PL isomorphism
\[
\R_+ \Lambda \simeq \R_{\geq 0}^d \times C^\mathrm{ess}.
\]
 Here $d=m$ or $d=2m$. After the radial projection to the hyperplane $H_\mathbf{v}$, again there exists an isomorphism
\begin{equation}\label{PLdecomposition}
    \mathbb{P} \Lambda \simeq \mathbb{P} (\R_{\geq 0}^d) * \mathbb{P} C^\mathrm{ess}
\end{equation}
as a cell complex. Here $*$ denotes the topological join and $\mathbb{P}(\R_{\geq 0}^d)$ is a $(d-1)$-simplex.

\begin{proof}[Proof of Theorem \ref{theoremB}]
    By \ref{PLdecomposition}, the union of maximal faces avoiding the cell generated by peripheral elements is $\mathbb{P} C^\mathrm{ess}$. By the join decomposition above, $\mathbb{P} C^\mathrm{ess}$ is homeomorphic to the link of the simplex $\mathbb{P}(\R_{\geq 0}^d)$. For $\mathbf{p} \in \Lambda^\circ$, the simplex $\mathbb{P}(\R_{\geq 0}^d)$ is an interior face of a PL ball. Its link is a sphere of dimension $(|V_\lambda^n|-1)-(d-1)-1 = \dim A^\mathrm{rel} - 1$.
\end{proof}

\subsection{Non-strict toric degenerations}
\label{NonStrictToricDegeneration}

Let $A^\mathrm{rel}
=
\Sk_\C^n(\Sigma,P,\mathbf{k}), \Lambda=\Lambda(\Sigma,\SL_n)$
and
\[
\mathbf{v}
\in
(\Z_{\geq0})^{|V_\lambda^n|}
\cap
(K_0(\Lambda)^*)_+
\]
and let $\Fcal_{\mathbf{v}}'$ be the quotient filtration on
$A^\mathrm{rel}$. By Proposition \ref{SemidegreeVectors}, the corresponding relative Rees algebra
\[
A^{\Fcal_\vb}(P, \mathbf{k})
\simeq (A^\mathrm{rel})^{\Fcal_{\mathbf{v}}'}
\]
is projective with a subdegree.

The purpose of this subsection is twofold. First, we show that
$A^{\Fcal_\vb}(P, \mathbf{k})$ is still a normal Gorenstein domain with
canonical module shifted by $-1$. Thus the corresponding
compactification remains log Calabi--Yau. Second, we explain the
additional issue that arises in the dual complex:
irreducible boundary components, as well as their intersections, may
merge under a non-strict degeneration.

\subsubsection{Log-CY compactification}

From now on, write $\mathcal{R}_\vb := (A^\mathrm{rel})^{\Fcal'_\vb}$ and $|\xb|_\mathbf{u} := \mathbf{u} \cdot \xb$. Choose
$\mathbf{u} \in (\Z_{>0})^{|V_\lambda^n|}$. 

Define an increasing secondary filtration $\mathcal{G} = (G_q)$ on
$\mathcal R_{\mathbf{v}}$ by
\[
G_q\mathcal R_{\mathbf{v}}
:=
\C\left\langle
\tr_{\xb}t^\ell:
\begin{array}{l}
\xb\in\Lambda^\mathrm{ess}(\Sigma,\SL_n),\\
|\xb|_\mathbf{v} \leq \ell,\\
|\xb|_\mathbf{u} \leq q
\end{array}
\right\rangle.
\]
The filtration $\mathcal G$ preserves the original grading given by
the exponent of $t$. Define the \textbf{secondary Rees algebra}
\[
\mathcal{R}_\vb^\mathcal{G} := \bigoplus_{q \geq 0} G_q \mathcal{R}_\vb s^q.
\]

Define the affine monoid
\[
\widetilde{\Lambda}_{\mathbf{v}}
:=
\left\{
(\xb,\ell)\in
\Lambda\times\Z_{\geq0}:
|\xb|_\vb\leq\ell
\right\}.
\]
and its corresponding monoid algebra by
\[
\deg X^{(\xb,\ell)}:=\ell.
\]

For a peripheral element $p$, put
\[
\widetilde p :=\left( p,|p|_\vb \right) \in \widetilde{\Lambda}_{\mathbf{v}}.
\]
Define the peripheral monomial ideal
\[
\widetilde I_{\mathbf{v}} :=
\begin{cases}
\left(X^{\widetilde p_i}
\right)_{i=1}^m,
&
n=2,\\[2mm]
\left(
X^{\widetilde p_i},
X^{\widetilde p_i'}
\right)_{i=1}^m,
&
n=3.
\end{cases}
\]

As before in Section \ref{relcompactification}, we have an isomorphism
\[
\gr(\mathcal R_{\mathbf{v}}^{\mathcal G})
    \simeq
    \C[\widetilde{\Lambda}_{\mathbf{v}}]
    /
    \widetilde I_{\mathbf{v}}.
\]

Let
\[
\widetilde{\Lambda}_{\mathbf{v}}^\circ
:=
\widetilde{\Lambda}_{\mathbf{v}}
\cap
\left(
\R_+\widetilde{\Lambda}_{\mathbf{v}}
\right)^\circ.
\]

\begin{lemma}\label{NonStrictReesMonoidNormal}
    $\widetilde{\Lambda}_{\mathbf{v}}$ is a normal monoid. Moreover,
    \[
    \widetilde{\Lambda}_{\mathbf{v}}^\circ
    =
    \widetilde{\mathbf p}
    +
    \widetilde{\Lambda}_{\mathbf{v}},
    \]
    where $\widetilde{\mathbf p}
    := \left(\mathbf p, \deg_{\mathbf{v}}(\mathbf p)+1 \right)$.
\end{lemma}

\begin{proof}
    The real cone generated by
    $\widetilde{\Lambda}_{\mathbf{v}}$ is
    \[
    \R_+\widetilde{\Lambda}_{\mathbf{v}}
    =
    \left\{
    (\xb,r):
    \xb\in\R_+\Lambda,
    \quad
    r\geq|\xb|_\vb
    \right\}.
    \]
    whose groupification is
    \[
    K_0(\Lambda)\oplus\Z.
    \]

    Suppose $a(\xb,r) \in \widetilde{\Lambda}_{\mathbf{v}}$
    for some $a>0$ and
    $(\xb,r)\in K_0(\Lambda)\oplus\Z$. Then $a\xb\in\Lambda$ and $\xb\in\Lambda$ since $\Lambda$ is normal. Moreover, $ar \geq \deg_{\mathbf{v}}(a\xb) = a|\xb|_\vb$
    implies $r\geq|\xb|_\vb$.
    Hence $(\xb,r)\in
    \widetilde{\Lambda}_{\mathbf{v}}$
    and $\widetilde{\Lambda}_\vb$ is normal.

    Write $\xb=\mathbf p+\mathbf{y}$
    for a unique $\mathbf{y}\in\Lambda$. Since $r$ is integral,
    \[
    r \geq |\mathbf p|_\vb + |\mathbf{y}|_\vb + 1.
    \]
    Therefore
    \[
    (\xb,r) - \widetilde{\mathbf p} = \left( \mathbf{y}, r-\deg_{\mathbf{v}}(\mathbf p)-1 \right)
    \in \widetilde{\Lambda}_{\mathbf{v}}.
    \]
    The converse inclusion is immediate.
\end{proof}

\begin{proposition}\label{NonStrictToricReesGorenstein}
    $\C[\widetilde{\Lambda}_{\mathbf{v}}] / \widetilde I_{\mathbf{v}}$
    is Gorenstein with canonical module is
    \[
    \omega_{
        \C[\widetilde{\Lambda}_{\mathbf{v}}]
        /
        \widetilde I_{\mathbf{v}}
    }
    \simeq
    \left(
    \C[\widetilde{\Lambda}_{\mathbf{v}}]
    /
    \widetilde I_{\mathbf{v}}
    \right)(-1).
    \]
\end{proposition}

\begin{proof}
    By Lemma \ref{NonStrictReesMonoidNormal} and $\C[\widetilde{\Lambda}_{\mathbf{v}}]$ is Gorenstein and
    \[
    \omega_{\C[\widetilde{\Lambda}_{\mathbf{v}}]}
    \simeq
    \C[\widetilde{\Lambda}_{\mathbf{v}}]
    \left(
    -|\mathbf p|_\vb-1
    \right).
    \]

    By the peripheral decomposition, every $(\xb,\ell) \in \widetilde{\Lambda}_{\mathbf{v}}$
    can be written uniquely as
    \[
    (\xb,\ell)
    = \sum_i a_i
    \left( p_i,\deg_{\mathbf{v}}(p_i) \right) +
    \left( \mathbf{y}, \ell-\sum_i a_i\deg_{\mathbf{v}}(p_i) \right)
    \]
    in the $\SL_2$ case, where $\mathbf{y}$ is essential. In the $\SL_3$
    case one also includes the peripheral elements $p_i'$.

    Consequently, the sequence
    \[
    \left(
    X^{\widetilde p_i}
    \right)_{i=1}^m
    \]
    in the $\SL_2$ case, and
    \[
    \left(
    X^{\widetilde p_i},
    X^{\widetilde p_i'}
    \right)_{i=1}^m
    \]
    in the $\SL_3$ case, is regular.

    Therefore the canonical-module formula for a quotient by a
    homogeneous regular sequence gives
    \begin{align*}
    \omega_{
        \C[\widetilde{\Lambda}_{\mathbf{v}}]
        /
        \widetilde I_{\mathbf{v}}
    }
    &\simeq
    \left(
    \C[\widetilde{\Lambda}_{\mathbf{v}}]
    /
    \widetilde I_{\mathbf{v}}
    \right)
    \left(
    -|\mathbf{p}|_\vb-1
    +
    |\mathbf{p}|_\vb
    \right).
    \end{align*}
\end{proof}

\begin{theorem}\label{NonStrictReesGorenstein}
     Theorem \ref{theoremA} holds for every projective filtration
    $\Fcal_{\mathbf{v}}$ arising from Definition \ref{TorDegen}.
\end{theorem}

\begin{proof}
    By the proof of Corollary \ref{ReesGorenstein} and Proposition \ref{NonStrictToricReesGorenstein}, $\mathcal{R}_\vb$ is also Gorenstein with its $a$-invariant $-1$. By Theorem \ref{RelNormalDomain}, $A^\mathrm{rel}$ is a normal
    domain. Since $\mathcal R_{\mathbf{v}}
    \subseteq
    A^\mathrm{rel}[t]$
    the ring $\mathcal R_{\mathbf{v}}$ is a domain. Since the quotient filtration on $A^\mathrm{rel}$ is induced by a
    subdegree, $\mathcal R_{\mathbf{v}}/(t)$ is reduced. Moreover, $t$ is a nonzerodivisor and
    then $\dim
    \mathcal R_{\mathbf{v}}/(t)
    =
    \dim\mathcal R_{\mathbf{v}}-1$. By
    \cite[Lemma 2.3]{whang2020global}, 
    $\mathcal R_{\mathbf{v}}$ is a normal domain.
    The result now follows from Theorem
    \ref{Watanabe} and Proposition \ref{NonStrictToricReesGorenstein}.
\end{proof}

\subsubsection{The dual complex}

Although $\mathcal{R}_{\mathbf{v}}/(t)$ is reduced, it need not be a toric face
ring. Example \ref{relskeinexample} shows that an intersection of irreducible boundary components may have more than one irreducible component. To remedy this, we replace the simplicial nerve used in the strict toric degeneration as follows.

\begin{definition}\label{NonStrictIncidenceComplex}
    Let
    \[
    D=\bigcup_{i\in I}D_i
    \]
    be a reduced divisor of a noetherian algebraic variety, where the
    $D_i$ are its irreducible components.

    For every nonempty subset $J\subseteq I$, put
    \[
    D_J
    :=
    \bigcap_{j\in J}D_j,
    \]
    where the intersection is scheme-theoretic. Define the
    \textbf{boundary incidence poset} $\mathcal S(D)$ by
    \[
    \mathcal S(D) :=
    \left\{ (J,Z): \emptyset\neq J\subseteq I, \quad 
    Z\in(\Irr(D_J)_\mathrm{red}) \right\}
    \]
    where $\Irr(-)$ denotes the set of irreducible components.
    We define a partial order on $\mathcal S(D)$ by
    \[
    (J,Z)\preceq(J',Z')
    \Longleftrightarrow
    J\subseteq J'
    \text{ and }
    Z'\subseteq Z.
    \]

    The \textbf{dual complex} of $D$, denoted by $\mathbb D(D)$
    is the order complex of $\mathcal S(D)$, in other words a chain
    \[
    (J_0,Z_0)
    \prec
    (J_1,Z_1)
    \prec
    \cdots
    \prec
    (J_q,Z_q)
    \]
    forms a $q$-simplex of $\mathbb{D}(D)$ in the boundary incidence poset.
\end{definition}

\begin{remark}\label{GeneralizedDualComplexRemark}
    If the toric degeneration were strict, the definition recovers the barycentric subdivision of the original definition of $\mathbb{D}(\partial A^\mathrm{rel})$. Again neither the smoothness nor SNC condition of divisors is considered. 
\end{remark}

The secondary filtration also applies to the boundary algebra.

\begin{proposition}\label{SecondaryBoundaryDegeneration}
    The filtration $\mathcal G$ induces a filtration on
    $\mathcal{R}_{\vb} / (t)$ satisfying
    \[
    \gr (\mathcal{R}_{\vb} / (t))^\mathcal{G}
    \simeq
    \begin{cases}
    \C[\Lambda]/(x^{p_i})_{i=1}^m,
    &
    n=2,\\[2mm]
    \C[\Lambda]/
    (x^{p_i},x^{p_i'})_{i=1}^m,
    &
    n=3.
    \end{cases}
    \]
\end{proposition}

\begin{proof}
    The element $t$ corresponds to
    the monomial $X^{(0,1)} \in \C[\widetilde{\Lambda}_\vb]$.
    Thus the correspondence 
    \[
    X^{(\xb,|\xb|_\vb)} \leftrightarrow x^\xb
    \]
    gives an isomorphism
    \[
    \C[\widetilde{\Lambda}_\vb]/(X^{(0,1)}) \simeq \C[\Lambda]
    \]
    The peripheral ideal
    becomes the usual peripheral monomial ideal.
\end{proof}

\begin{remark}
    The proposition shows that even though $\mathcal{R}_\vb(t)$ is not a toric face ring, it admits a (flat) degeneration to a toric face ring.
\end{remark}

\begin{remark}
    The secondary degeneration from $\mathcal{R}_\vb / (t)$ to
    $\gr (\mathcal{R}_\vb)^\mathcal{G}$ may split one irreducible component into several boundary components. The example below suggests that the $\Delta$-complex is obtained by contracting a subcomplex of the dual complex given by a strict toric degeneration.
\end{remark}

\begin{question}
    Does Theorem \ref{theoremB} hold for a non-strict toric degeneration?
\end{question}

In the skein-theoretical viewpoint, one needs some further investigation of defects.

\begin{example}
\label{NonStrictPairOfPants}
    In Example
    \ref{relskeinexample} we considered the boundary divisor
    \[ 
    D = \Proj \C[x,y,z,\overline z] /
    \left( xy,\, z\overline{z}-x^3-y^3 \right).
    \]
    with two irreducible components $D_1, D_2$. The dual complex is homeomorphic to $S^1$.

    The boundary given by the secondary Rees algebra is
    \[
    D' = \Proj \C[x,y,z,\overline z] / (xy,z\overline{z}).
    \]
    with four minimal primes
    \[
    (x,z), \quad (x,\overline z), \quad (y,z), \quad (y,\overline z).
    \]
    The dual complex is a $8$-cycle generated by the flags
    \[
    (x, (x,z)), (z, (x,z)), (x, (x,\overline z)), (\overline z, (x, \overline z)), (y, (y,z)), (z, (y,z)), (y, (y, \overline z)), (\overline z, (y, \overline z)),
    \]
    and hence is
    also homeomorphic to $S^1$.

    Note that the $8$-cycle is a barycentric subdivision of the $4$-cycle obtained by subdividing each edge corresponding to 
    $\Irr(D_1 \cap D_2)$. This $4$-cycle can be viewed as blowing up the two zero-dimensional boundary
    strata of $D$.
\end{example}

\subsection{Various filtrations of $\SL_2$-skein algebras}

Let $\Sigma$ be either a closed surface or a punctured surface and let $\Gamma \subseteq \Sigma$ be an essential multicurve. This defines a filtration of $\Sk_R^2(\Sigma)$

\[
\Fcal_{\Gamma, n} = R \langle \tr_{\Gamma'} : i(\Gamma, \Gamma') \leq n\rangle.
\]

\begin{theorem}{\cite[Theorem 12]{przytycki2019skein}}
    If $R$ is a domain then $\gr A^{\Fcal_\Gamma}$ is a domain.
\end{theorem}

It follows that the corresponding degree-like function $\delta_\Gamma$ is a semi-degree.

\begin{remark}
     One can check directly 
     \[
     i(\Gamma, -) \in ((2\Z)^*)_{\geq 0} \cap (K_0(\Lambda)^*)_{\geq 0}
     \]
     by additivity, integrality and nonnegativity. However, $i(\Gamma, -)$ is not strictly positive since $i(\Gamma, \Gamma) = 0$ and does not give a projective filtration.
\end{remark}

Alternatively, let $\Sigma$ be a punctured surface. Consider a collection of disjoint simple ideal arcs $\Acal$ on $\Sigma$, which will be called a \textit{multiarc}. This defines a filtration of $\Sk_R^2(\Sigma)$
\[
\Fcal_{\Acal, n} = R \langle \tr_{\Gamma'} : i(\Acal, \Gamma') \leq n \rangle.
\]

Any multiarc $\Acal$ can be viewed as a weighted collection of disjoint nonparallel ideal arcs, which extends to an ideal triangulation $\lambda$. Then one can realize $\Acal \in (\Z_{\geq 0})^{|V_\lambda^2|}$. By Proposition \ref{SemidegreeVectors}, the corresponding degree-like function $\delta_{\Acal}$ is a semi-degree.

A multiarc is called \textbf{filling} if each component of $\Sigma \setminus \Acal$ is an ideal polygon. Then
\[
i(\Acal, \Gamma) > 0
\]
for any $\Gamma$ and then $\Acal$ induces a projective filtration. \\

Finally, we briefly review Whang's compactification \cite{whang2020global}. Fix a presentation $\pi = \pi_1 \Sigma_{g,m} = F_{2g+m-1} = \langle S \mid R \rangle$ where $S = \{ \alpha_i : i \in I \}$ for some index set $I$.

\begin{definition}
    \begin{enumerate}
        \item The assignment $\ell_S(\cdot, \cdot) : \pi \times \pi \rightarrow \Z_\geq 0$ defined by
        \[
        \ell(g, h) = \min \{ r : gh^{-1} = \alpha_1^\pm \cdots \alpha_r^\pm \}
        \]
        is called the \textbf{word metric} on $\pi$ and $|g|_S = \ell_S(g, e)$ is called the \textbf{word length} on $\pi$.
        \item The assignment $\| \cdot \| : \pi \rightarrow \Z_{\geq 0}$
        \[
        \| g \|_S := \min \{ \ell_S (h, e) : h \text{ is conjugate to } g \}
        \]
        is called the \textbf{word norm} on $\pi$.
        \item The filtration of $\Sk_\C^2(\Sigma)$ 
    \[
    \Fcal_{w,n} = \C \langle \tr_w : \| w \|_S \leq n \rangle
    \]
    is called the \textbf{word length filtration}, denoted by $\Fcal_w$.
    \end{enumerate}
\end{definition}

Here $\tr_w = \tr_\gamma$ where $w$ is realized as an oriented curve with some basepoint in $\Sigma$. Here $\gamma$ is considered to be a homotopy class of an unoriented curve. Since
\[
\| w \|_S = \| w^{-1}\|_S, \quad \| w\|_S = \| gwg^{-1}\|_S
\]
the set $\{ n : \tr_\gamma \in \Fcal_n \}$ does not depend on the choice of the word $w$ realizing $\gamma$ and the word length filtration is well-defined. Also note that the intersection number $i(-, -)$ is defined over basepoint-free homotopy classes of curves or arcs, the notation $i(w, -)$ is well-defined for any $w \in \pi$.

A generating set $S$ of $\pi = \pi_1 \Sigma_{g,m}$ is called \textbf{simple} if its elements can be realized as embedded curves only intersecting at the basepoint. $S$ is called \textbf{minimal} if it is simple and $\Sigma_{g,m}$ deformation retracts to the tubular neighbourhood of the curves, in particular $|S| = 2g + m - 1$.

\begin{lemma}
    For any minimal generating set $S$, there exists a filling multiarc $\Acal$ such that $\| \cdot \|_S = i(\Acal, \cdot)$.
\end{lemma}

\begin{proof}
    We will apply the construction by Erlandsson \cite{erlandsson2019remark}, there exists a collection of curves and arcs $\Acal$ which can be constructed explicitly. Let $\Gamma$ be a ribbon graph consisting of embedded curves intersecting only at the basepoint. Each complementary region of $\Gamma$ is either a once-punctured polygon or a polygon. Since $S$ is minimal, the tubular neighbourhood $N(\Gamma)$ has Euler characteristic $\chi(N(\Gamma)) = \chi(\Sigma)$. Thus every complementary region has Euler characteristic $0$ and hence is a once-punctured polygon. Thus $\Acal$ has no curve components and the arc components are disjoint.
\end{proof}

The word length filtration given is called \textbf{minimal} if the filtration is given by the word norm $\| \cdot \|_S$ for some minimal generating set $S$. One can conclude that any minimal word length filtration $\Fcal_w$ is a special case of the projective filtration $(A, \Fcal_{\mathbf{v}})$ giving a possibly non-strict toric degeneration. Whang \cite{whang2020global} used the standard presentation with $|S| = 2g + m - 1$ which is minimal. It would be interesting to prove Theorem \ref{theoremA} and Theorem \ref{theoremB} by constructing a filtration of the skein algebra with respect to any collection of curves and arcs.

\begin{corollary}
    Theorem \ref{theoremA} holds for any minimal word length filtration.
\end{corollary}

\section{Examples}\label{example}

In this section, we compute the Hilbert series and the dual complex of the skein algebras of surfaces of low complexity.

\subsection{The Hilbert series}

\begin{proposition} \label{HoneycombObs}
    Let $(\Sigma, \lambda)$ be a punctured surface with an ideal triangulation. Let $W$ be a nonelliptic $3$-web in $\Sigma$ and $H(T) \in \Z$ denote the honeycomb degree of $W$ at an ideal triangle $T \in \lambda$. Then 
    \[
    \sum_{T \in \lambda} H(T)= 0.
    \]
\end{proposition}

\begin{proof}
    Net charge of a nonelliptic web $W$ is zero. Local contribution of a left turn or a right turn is zero.
\end{proof}

Let
\[
    \mathbf{v} = (v_i) = 
    \begin{cases}
        1 & (v_i \text{ is a face coordinate})\\
        0 & (v_i \text{ is an edge coordinate}),
    \end{cases}
\]
inducing a toric degeneration of $A = \Sk_\C^3(\Sigma)$. The \textbf{complexity} of a nonelliptic $3$-web $W = \xb \in \Lambda(\Sigma, \SL_3)$ is defined by $|W| = \xb \cdot \mathbf{v}$. As a local $\SL_3$-train track configuration, the contributions of a right turn, left turn and honeycomb of degree $\pm 1$ is $1, 2$ and $3$. 

Consider the Hilbert series of the filtrations $(A, \Fcal)$ and $(A^\mathrm{rel}, \Fcal')$
\begin{align*}
    h_\Sigma(t) &= \sum_{i \geq 0} c_i t^i\\
    h_\Sigma^\mathrm{rel}(t) &= \sum_{i \geq 0} c^\mathrm{ess}_i t^i
\end{align*}
where $c_i$ (resp. $c^\mathrm{ess}_i$) denotes the number of nonelliptic $3$-webs (resp. essential nonelliptic $3$-webs) of complexity $i$. For later use, we also the following modified definition for a surface with boundary with an ideal triangulation without self-folded triangles. Let $W$ be a collection of local $\SL_3$-train track compatible along each glued edge. The complexity of $W$ is defined as the sum of the local complexity. Then 
\[
h_\Sigma(t) = \sum_{i \geq 0} c_i t^i
\]
where $c_i$ denotes the number of compatible local $\SL_3$-train tracks of complexity $i$.

\subsubsection{Triangle}

We use the Frohman--Sikora coordinate. Fix an ideal triangle $T = (e_1, e_2, e_3)$. Let $a_1, b_1, c_1$ denote the total positive charges and $a_2, b_2, c_2$ denote the total negative charges at $e_1, e_2, e_3$. Let $-r$ denote the number of right turns in $T$. Then $(a_1, a_2, b_1, b_2, c_1, c_2, r) \in \Z^7$ and the image of the coordinate map is generated by
\begin{align*}
    \mathbf{v}_1=(0,1,1,0,0,0,-1) \\
    \mathbf{v}_2=(0,0,0,1,1,0,-1) \\
    \mathbf{v}_3=(1,0,0,0,0,1,-1) \\
    \mathbf{v}_4=(0,1,0,0,1,0,0) \\
    \mathbf{v}_5=(1,0,0,1,0,0,0) \\
    \mathbf{v}_6=(0,0,1,0,0,1,0) \\
    \mathbf{v}_7=(1,0,1,0,1,0,0) \\
    \mathbf{v}_8=(0,1,0,1,0,1,0).
\end{align*}

The local complexity is given by the functional $\mathbf{v} \mapsto \mathbf{v} \cdot \vec{1}$ where $\vec{1} = (1, 1, 1, 1, 1, 1, 1)$. By Douglas--Sun \cite{douglas2024tropical}, the image is a one-relator monoid with the relation
\[
\mathbf{v}_4 + \mathbf{v}_5 + \mathbf{v}_6 = \mathbf{v}_7 + \mathbf{v}_8.
\]
In other words, any vector $\mathbf{v} \in \Z^7$ in the image is written uniquely as either $\sum_{i=1}^6 n_i \mathbf{v}_i$, $\sum_{i=1}^6 n_i \mathbf{v}_i + n_7 \mathbf{v}_7$ or $\sum_{i=1}^6 n_i \mathbf{v}_i + n_8 \mathbf{v}_8$ where $n_1, \ldots, n_6 \in \Z_{\geq 0}$ and $n_7, n_8 \in \Z_{> 0}$.

\begin{proposition}
    Let $T$ be an ideal triangle. Then
    \[
    h_T(t) = \frac{1+t^3}{(1-t)^3(1-t^2)^3(1-t^3)}.
    \]
\end{proposition}

\begin{proof}
    \begin{align*}
        h_T(t) =& \sum_{n_1, \ldots, n_6 \geq 0} t^{n_1 + n_2 + n_3 + 2(n_4 + n_5 + n_6)} \\ &+ \sum_{\substack{n_1, \ldots, n_6 \geq 0 \\ 
        n_7 > 0}} t^{n_1 + n_2 + n_3 + 2(n_4 + n_5 + n_6) + 3n_7} \\ &+ \sum_{\substack{n_1, \ldots, n_6 \geq 0 \\ 
        n_8 > 0}} t^{n_1 + n_2 + n_3 + 2(n_4 + n_5 + n_6) + 3n_8} \\
        &= \frac{1}{(1-t)^3(1-t^2)^3} + \frac{2t^3}{(1-t)^3(1-t^2)^3(1-t^3)} = \frac{1+t^3}{(1-t)^3(1-t^2)^3(1-t^3)}.
    \end{align*}
\end{proof}

The Hilbert series satisfies a functional equation
\[
h_T(t^{-1}) = -t^9 h_T(t).
\]

\subsubsection{Surfaces with Euler characteristic $-1$}

Let $\Sigma$ be either $\Sigma_{0,3}$ or $\Sigma_{1,1}$. Fix an ideal triangulation without self-folded triangles of $\Sigma$ as in Figure \ref{Triangulation}. Fix an ideal triangle $T_1$ and label the edges $e_1, e_2, e_3$ as before. 

\begin{figure}
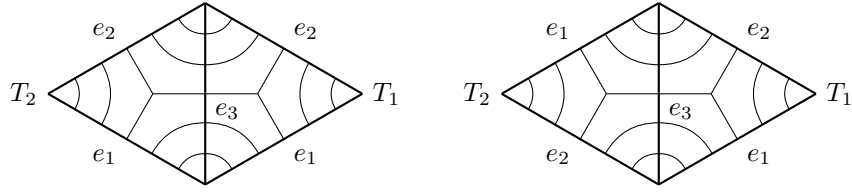

    \centering
    \Triangulation
    \caption{Ideal triangulations of $\Sigma_{0,3}$ and $\Sigma_{1,1}$.}
    \label{Triangulation}
\end{figure}

The honeycomb degree $d_1$ of $T_1$ is given as
\[
d_1 = \frac 13 (a_1 + b_1 + c_1 - a_2 - b_2 - c_2)
\]
and we have $d_1 + d_2 = 0$ by Proposition \ref{HoneycombObs}. 

Let $\mathbf{v} \in \Z^7$ be the local $\SL_3$-train track coordinate for $T_1$. We will specify a $\SL_3$-train track $\mathbf{v'}$ which is compatible with $\mathbf{v}$. We have 
\[
\mathbf{v} = \sum_{i=1}^6 n_i\mathbf{v}_i + n_7\mathbf{v}_7 \Longrightarrow \mathbf{v'}  = \sum_{i=1}^6 n'_i\mathbf{v}_i + n_7\mathbf{v}_8
\]
and
\[
\mathbf{v} = \sum_{i=1}^6 n_i\mathbf{v}_i + n_8\mathbf{v}_8 \Longrightarrow \mathbf{v'}  = \sum_{i=1}^6 n'_i\mathbf{v}_i + n_8\mathbf{v}_7
\]
where $n_7=0$ or $n_8=0$. Then the first six coordinates of $\mathbf{v}'$ for $T_2$ are given by
\[
\mathbf{v}' = \mathbf{v} + n(\mathbf{v}_1+\mathbf{v}_2+\mathbf{v}_3-\mathbf{v}_4-\mathbf{v}_5-\mathbf{v}_6).
\]

Here the integer $n$ runs over $-\min \{n_1, n_2, n_3 \} \leq n \leq \min\{n_4, n_5, n_6\}$.

\begin{proposition}
    \[
    h_{\Sigma_{0,3}}(t) = h_{\Sigma_{1,1}}(t) = \frac{(1-t^{12})(1-t^{18})}{(1-t^2)^3(1-t^4)^3(1-t^6)^2(1-t^9)^2}
    \]
    and
    \begin{align*}
        h_{\Sigma_{0,3}}^\mathrm{rel}(t) &= \frac{(1+t^6)(1+t^9)}{(1-t^6)(1-t^9)} \\
        h_{\Sigma_{1,1}}^\mathrm{rel}(t) &= 
        \frac{(1-t^{12})^2(1-t^{18})}{(1-t^2)^3(1-t^4)^3(1-t^6)(1-t^9)^2}.
    \end{align*}
\end{proposition}

\begin{proof}
     The above observation shows that $h_{\Sigma_{0,3}}(t) = h_{\Sigma_{1,1}}(t)$ is given by
    \[
    \sum_{n_7, n_8} t^{6(n_7 + n_8)}\sum_{n_1, \ldots, n_6 \geq 0} t^{2(n_1 + n_2 + n_3) + 4(n_4 + n_5 + n_6)} \cdot \sum_{n = -\min\{n_1, n_2, n_3 \}}^{\min\{n_4, n_5, n_6\}} t^{-3n}
    \]
    where $n_7, n_8$ satisfy the conditions $n_7 \neq 0 \Rightarrow n_8=0$ and $n_8 \neq 0 \Rightarrow n_7 =0$. 

    The generating function of the honeycomb part is
    \[
    \sum_{n_7, n_8} t^{6(n_7 + n_8)} = 1 + \frac{2t^6}{1-t^6} = \frac{1+t^6}{1-t^6}.
    \]
    
    Observe
    \begin{align*}
    &\sum_{n_1, \ldots, n_6 \geq 0} t^{2(n_1 + n_2 + n_3) + 4(n_4 + n_5 + n_6)}\sum_{n = -\min\{n_1, n_2, n_3 \}}^{\min\{n_4, n_5, n_6\}} t^{-3n} \\
    = & \frac{1}{1-t^{-3}}\sum_{n_1, \ldots, n_6 \geq 0} (t^{2(n_1 + n_2 + n_3) + 4(n_4 + n_5 + n_6) + 3 \min\{n_1, n_2, n_3\}} - t^{2(n_1 + n_2 + n_3) + 4(n_4 + n_5 + n_6) - 3(\min\{n_4, n_5, n_6\} + 1)}) \ \\
    = & \frac{1}{1-t^{-3}} \left( \frac{1 - t^6}{(1-t^2)^3(1-t^4)^3} \sum_{k = \min\{n_1, n_2, n_3 \} \geq 0} t^{9k} -  \frac{t^{-3}(1-t^{12})}{(1-t^2)^3(1-t^4)^3} \sum_{k = \min\{n_4, n_5, n_6 \} \geq 0} t^{9k}\right) \ \\
    = & \frac{1}{1-t^{-3}} \left( \frac{1 - t^6}{(1-t^2)^3(1-t^4)^3(1-t^9)} - \frac{t^{-3}(1-t^{12})}{(1-t^2)^3(1-t^4)^3(1-t^9)}\right) = \frac{1+t^9}{(1-t^2)^3(1-t^4)^3(1-t^9)}.
    \end{align*}
    In conclusion, we have
    \[
    h_{\Sigma_{0,3}}(t) = h_{\Sigma_{1,1}}(t) = \frac{(1+t^6)(1+t^9)}{(1-t^2)^3(1-t^4)^3(1-t^6)(1-t^9)}.
    \]

    To compute the Hilbert series for relative skein algebra, it suffices to compute the generating function for peripheral skeins. If $\Sigma = \Sigma_{0,3}$, we have
    \[
    \sum_{n_1, \ldots, n_6 \geq 0} t^{2(n_1 + n_2 + n_3) + 4(n_4 + n_5 + n_6)} = \frac{1}{(1-t^2)^3(1-t^4)^3}
    \]
    and if $\Sigma = \Sigma_{1,1}$, we have
    \[
    \sum_{n_1, n_2 \geq 0} t^{6n_1 + 12n_2} = \frac{1}{(1-t^6)(1-t^{12})}
    \]
    and the result follows.
\end{proof}

Write $h_{F_2}(t) = h_{\Sigma_{0,3}}(t) = h_{\Sigma_{1,1}}(t)$. One has a functional equation
\[
h_{F_2}(t^{-1}) = t^{18} h_{F_2}(t)
\]
and
\begin{align*}
    h^\mathrm{rel}_{\Sigma_{0,3}}(t^{-1}) = h^\mathrm{rel}_{\Sigma_{0,3}}(t) \\
    h^\mathrm{rel}_{\Sigma_{1,1}}(t^{-1}) = h^\mathrm{rel}_{\Sigma_{1,1}}(t).
\end{align*}

\subsection{The dual complex}

\subsubsection{Three-punctured sphere} 
Let $\Sigma = \Sigma_{0,3}$. In Example \ref{HilbertBasis} we obtained the complete list of the extremal vectors in $\Lambda(\Sigma, \SL_3)$. Explicitly, we have
\[
\Lambda^\mathrm{ext}(\Sigma, \SL_3) \cap \Lambda^\mathrm{ess}(\Sigma, \SL_3) = \{ W_\theta, W'_\theta, W_\mathrm{comm}, W'_\mathrm{comm} \}.
\]

We use the Douglas--Sun coordinate $(e_{11}, e_{12}, e_{21}, e_{22}, e_{31}, e_{32}, e, e') \in \Z_{\geq 0}^8$ where the $e_{ij}$ are edge coordinates and $e, e'$ are face coordinates. Then
\begin{align*}
    \mathbf{v}_1 &= (1, 2, 2, 1, 0, 0, 1, 1)\\
    \mathbf{v}_2 &= (0, 0, 1, 2, 2, 1, 1, 1)\\
    \mathbf{v}_3 &= (2, 1, 0, 0, 1, 2, 1, 1)\\
    \mathbf{v}_4 &= (2, 1, 1, 2, 0, 0, 2, 2)\\
    \mathbf{v}_5 &= (0, 0, 2, 1, 1, 2, 2, 2)\\
    \mathbf{v}_6 &= (1, 2, 0, 0, 2, 1, 2, 2)\\
    \mathbf{v}_7 &= (2, 1, 2, 1, 2, 1, 3, 3) \\
    \mathbf{v}_8 &= (1, 2, 1, 2, 1, 2, 3, 3) \\
    \mathbf{v}_9 &= (3, 3, 3, 3, 3, 3, 3, 6) \\
    \mathbf{v}_{10} &= (3, 3, 3, 3, 3, 3, 6, 3).
\end{align*}

Here $\mathbf{v}_1, \ldots, \mathbf{v}_6$ correspond to peripheral skeins and $\mathbf{v}_7, \mathbf{v}_8, \mathbf{v}_9, \mathbf{v}_{10}$ stand for $W_\theta, W'_\theta, W_\mathrm{comm}, W'_\mathrm{comm}$.
One checks that
\begin{align*}
    \mathbf{v}_7 + \mathbf{v}_8 &= \mathbf{v}_4 + \mathbf{v}_5 + \mathbf{v}_6 \\
    \mathbf{v}_9 + \mathbf{v}_{10} &= \sum_{i=1}^6 \mathbf{v}_i = \mathbf{p}
\end{align*}
so the pairs do not span a face in $\R_+ \Lambda$. Thus the possible maximal faces are following pairs: 
\[
\{ W_\theta, W_\mathrm{comm} \}, \{ W'_\theta, W_\mathrm{comm} \}, \{ W_\theta, W'_\mathrm{comm} \}, \{ W'_\theta, W'_\mathrm{comm} \}
\]
Indeed, one can find the supporting hyperplane of the cone, normal to
\begin{align*}
    (&0, 1, 0, 1, 0, 1, 1, -2) \\
    (&0, 1, 0, 1, 0, 1, -2, 1) \\
    (&1, 0, 1, 0, 1, 0, 1, -2) \\
    (&1, 0, 1, 0, 1, 0, -2, 1),
\end{align*}
corresponding to the pair, respectively. It follows that 
\[
\mathbb{D}(\partial A^\mathrm{rel}) \simeq S^1.
\]

\subsubsection{Once-punctured torus}\label{OncePunctureTorusLSL3}

Let $\Sigma = \Sigma_{1,1}$. We use the same notation $\mathbf{v}_1, \ldots \mathbf{v}_{10}$, where $\mathbf{v}_9, \mathbf{v}_{10}$ are peripheral. By the relation
\[
    \mathbf{v}_9 + \mathbf{v}_{10} = \sum_{i=1}^6 \mathbf{v}_i = \mathbf{p},
\]
the face containing $\mathbf{v}_1, \ldots, \mathbf{v}_6$ should also contain a peripheral. By the relation
\[
    \mathbf{v}_7 + \mathbf{v}_8 = \mathbf{v}_4 + \mathbf{v}_5 + \mathbf{v}_6,
\]
any face containing both $\mathbf{v}_7$ and $\mathbf{v}_8$ should also contain $\mathbf{v}_4, \mathbf{v}_5$ and $\mathbf{v}_6$.

The upshot is that there exists $9$ faces of dimension $5$
\begin{align*}
    \{ \mathbf{v}_1, \mathbf{v}_2, \mathbf{v}_3,& \mathbf{v}_4, \mathbf{v}_5, \mathbf{v}_7 \} \\
    \{ \mathbf{v}_1, \mathbf{v}_2, \mathbf{v}_3,& \mathbf{v}_4, \mathbf{v}_5, \mathbf{v}_8 \} \\
    \{ \mathbf{v}_1, \mathbf{v}_2, \mathbf{v}_3,& \mathbf{v}_4, \mathbf{v}_6, \mathbf{v}_7 \} \\
    \{ \mathbf{v}_1, \mathbf{v}_2, \mathbf{v}_3,& \mathbf{v}_4, \mathbf{v}_6, \mathbf{v}_8 \} \\
    \{ \mathbf{v}_1, \mathbf{v}_2, \mathbf{v}_3,& \mathbf{v}_5, \mathbf{v}_6, \mathbf{v}_7 \} \\
    \{ \mathbf{v}_1, \mathbf{v}_2, \mathbf{v}_3,& \mathbf{v}_5, \mathbf{v}_6, \mathbf{v}_8 \} \\
    \{ \mathbf{v}_1, \mathbf{v}_2, \mathbf{v}_4, &\mathbf{v}_5, \mathbf{v}_6, \mathbf{v}_7, \mathbf{v}_8 \} \\
    \{ \mathbf{v}_1, \mathbf{v}_3, \mathbf{v}_4, &\mathbf{v}_5, \mathbf{v}_6, \mathbf{v}_7, \mathbf{v}_8 \} \\
    \{ \mathbf{v}_2, \mathbf{v}_3, \mathbf{v}_4, &\mathbf{v}_5, \mathbf{v}_6, \mathbf{v}_7, \mathbf{v}_8 \}
\end{align*}
which will be maximal. Indeed, they are realized by the normal vectors
\begin{align*}
    (-2, 2, -2, 0, 0, 0, 1, 1&) \\
    (0, 0, 0, -2, 2, -2, 1, 1&) \\
    (-2, 0, 0, 0, -2, 2, 1, 1&) \\
    (0, -2, 2, -2, 0, 0, 1, 1&) \\
    (0, 0, -2, 2, -2, 0, 1, 1&) \\
    (2, -2, 0, 0, 0, -2, 1, 1&) \\
    (2, 0, 0, 0, 0, 2, -1, -1&) \\
    (0, 0, 0, 2, 2, 0, -1, -1&) \\
    (0, 2, 2, 0, 0, 0, -1, -1&).
\end{align*}

The homology of the nerve complex $\mathcal{N}$ can be computed via Python code, included in Appendix \ref{pythoncode}. The $f$-vector of the complex is given as
\[
(f_0, \ldots, f_7) = (9, 36, 84, 126, 126, 78, 24, 3)
\]
and the torsion part of the homology given by the Smith normal form is trivial. Thus the homology is given as
\[
H_i(\mathcal{N}, \Z) = 
\begin{cases}
    \Z & (i = 0, 5) \\
    0 & (i \neq 0, 5).
\end{cases}
\]

Since $f_i = \binom{9}{i+1}$ for $0 \leq i \leq 4$, the $4$-skeleton of $\mathcal{N}$ is full. Thus $\mathcal{N}$ is $3$-connected, in particular simply connected and then is homotopy equivalent to the sphere $S^5$. Since there is a $7$-simplex, $8$ boundary divisors meet simultaneously and are not SNC.

\begin{appendix}

\section{Singularity analysis of the relative character varieties}\label{appendixa}

\subsection{Variation analysis}\label{appendix1}

Let $\Sigma = \Sigma_{g,m}$ and $\pi = \pi_1 \Sigma$ be given with the standard representation
\[
\pi = \langle \alpha_1, \ldots, \alpha_{2g+m} \mid [\alpha_1, \alpha_2] \cdots [\alpha_{2g-1}, \alpha_{2g}] \alpha_{2g+1} \cdots \alpha_{2g+m} \rangle.
\] Consider the morphism
\[
\Rep(\pi, \SL_n)_\C \simeq (\SL_n)_\C^{2g + m - 1} \rightarrow X(\pi, \SL_n)_\C \rightarrow (\SL_n \sslash \SL_n)_\C^m
\]
where the second arrow takes coefficients of the characteristic polynomial of $\rho(\alpha_{2g+i})$ for $1 \leq i \leq m$. We used the morphism for defining the relative $\SL_n$-character variety in Section \ref{relativeskeinalgebras}. We may define the \textbf{relative representation variety} $\Rep(\pi, \SL_n, \mathbf{k})$ (over $\C$) be the fiber product commuting the diagram
\[
    \begin{tikzcd}
    \Rep(\pi, \SL_n, \mathbf{k})_\C \arrow[d] \arrow[r] & \Rep(\pi, \SL_n)_\C \arrow[d] \\
    X(\pi, \SL_n, \mathbf{k})_\C \arrow[r]       & X^\mathrm{rel}(\pi, \SL_n)_\C.       
    \end{tikzcd}
\]
A matrix $A \in \SL(n, \C)$ is called \textbf{regular} if $\dim_\C C_{\SL(n, \C)}(A) = n-1$ where $C_G(g)$ denotes the centralizer of $g \in G$. Observe the fiber of the morphism
\[
\SL(n, \C) \rightarrow \SL(n, \C) \sslash \SL(n, \C) = \C^{n-1}
\]
consists of finitely many conjugacy classes, say $\mathcal{C}$s.

\begin{lemma}\label{regularcodim2}
    If $A \in \mathcal{C}$ is regular then $\dim_\C \mathcal{C} = n^2 - n$. Let $n=2$ or $n=3$. If $A \in \mathcal{C}$ is not regular then $\dim_\C \mathcal{C} \leq n^2 - n - 2$.
\end{lemma}

\begin{proof}
    If $A \in \mathcal{C}$ is regular, then $\dim_\C \mathcal{C} = (n^2 - 1) - (n - 1) = n^2 - n$. Now suppose $n=2$. By considering Jordan form of $A$, if $A$ is not regular then $A = \pm I$ and $\dim_\C \mathcal{C} = 0 \leq 4 - 2 - 2$. Similarly, suppose $n=3$ then if $A$ is not regular then by the Jordan form of $A$ one can directly check $\dim C_{\SL(3, \C)}(A) \geq 4$. Hence $\dim_\C \mathcal{C} \leq 4 = 9 - 3 - 2$.
\end{proof}

Let $G = \SL(n, \C)$ and $\mathfrak{g} = \mathfrak{sl}(n, \C)$. Denote the stabilizer of the adjoint action of $A \in G$ on $\mathfrak{g}$ by $\mathfrak{g}^A$. We may use the existence of the nondegenerate $\Ad$-invariant bilinear form $\langle -, - \rangle$ on $\mathfrak{g}$, for example the trace form $\langle X, Y \rangle = \tr(XY)$.

\begin{lemma}\label{conjdifferential}
    Let $A \in \mathcal{C}$. Then $(T_A \mathcal{C})^\perp = \mathfrak{g}^A$.
\end{lemma}

\begin{proof}
    Following Goldman \cite{goldman1984symplectic}, the derivative of the tangent curve $\exp(tX)A\exp(-tX)$ is given by $X - \Ad_A X$. Then $Y \in (T_A \mathcal{C})^\perp$ if and only if
    \[
    \langle X, Y \rangle = \langle \Ad_A X, Y \rangle = \langle X, \Ad_A^{-1}Y \rangle
    \]
    for all $X \in \mathfrak{g}$. This is equivalent to $Y \in \mathfrak{g}^{A}$.
\end{proof}

\begin{lemma}\label{commdifferential}
    The commutator map $[-,-] : G \times G \rightarrow G$ is singular at $(A, B)$ if and only if $\mathfrak{g}^A \cap \mathfrak{g}^B \neq 0$.
\end{lemma}

\begin{proof}
    We may identify $\mathfrak{g}$ with $T_{(A, B)}G$. For $X, Y \in \mathfrak{g}$, we have
    \[
    d[-,-]_{(A, B)}(X,Y) = X + \Ad_A Y - \Ad_{ABA^{-1}} X - \Ad_{[A, B]} Y.
    \]
    If $Z \in (\img d[-,-])^\perp$, we have
    \[
    \langle X, Z \rangle = \langle \Ad_{ABA^{-1}}X, Z \rangle = \langle X, \Ad_{ABA^{-1}}^{-1}Z \rangle
    \]
    for any $X \in \mathfrak{g}$ then $Z \in \mathfrak{g}^{ABA^{-1}}$. Similarly, 
    \[
    \langle \Ad_A Y, Z \rangle = \langle \Ad_{[A, B]} Y, Z \rangle
    \]
    for any $Y \in \mathfrak{g}$ implies $\Ad_{A}^{-1}Z \in \mathfrak{g}^{BAB^{-1}}$. Thus we have $\Ad_{A}^{-1} Z \in \mathfrak{g}^B$ and then $\Ad_{A}^{-1} Z \in \mathfrak{g}^B \cap \mathfrak{g}^A$.
    
    Conversely, if $0 \neq Z \in \mathfrak{g}^A \cap \mathfrak{g}^B$ then $\langle \img d[-,-], Z \rangle = 0$.
\end{proof}

\begin{lemma}\label{boundarydifferential}
    Consider the morphism $\partial_\mathcal{C} : G^{2g} \times \prod_{i=1}^m \mathcal{C}_i \rightarrow G$ defined by
    \[
    \rho = (A_1, B_1, \ldots, A_g, B_g, C_1, \ldots, C_m) \mapsto \prod_{i=1}^g [A_i, B_i] \prod_{j=1}^m C_j.
    \]
    Then $d \partial_{\mathcal{C}, \rho}$ is surjective if and only if
    \[
    \mathfrak{g}^\pi := \bigcap_{i=1}^g (\mathfrak{g}^{A_i} \cap \mathfrak{g}^{B_i}) \cap \bigcap_{j=1}^m \mathfrak{g}^{C_j} = 0.
    \]
\end{lemma}

\begin{proof}
    Following Whang \cite{whang2020global}, we write $\alpha_i = [A_i, B_i]$ for $1 \leq i \leq g$ and $\alpha_i = C_{i-g}$ for $g+1 \leq i \leq g+m$. We also write $W_k = \prod_{i=1}^k \alpha_i$ for $1 \leq k \leq g + m$. Then we have $W_0 = W_{g + m} = 1$ and 
    \[
    (d \partial_\mathcal{C})_\rho = \sum_{k=1}^{g+m} \Ad_{W_{k-1}} \circ d\alpha_k.
    \]
    Let $Z \in (\img (d \partial_\mathcal{C})_\rho)^\perp$. As in the proof of Lemma \ref{commdifferential}, by taking differential at the coordinate corresponding to $(A_i, B_i)$ we inductively have $Z \in \mathfrak{g}^{A_i} \cap \mathfrak{g}^{B_i}$. Similarly as in the proof of Lemma \ref{conjdifferential}, by taking differential at the coordinate corresponding to $C_j$ we have $Z \in \mathfrak{g}^{C_j}$.
\end{proof}

\subsection{Reducible loci}

Consider the morphism
\[
\Rep(\pi, \GL_2)_\C \rightarrow \Rep(\pi, \SL_3)_\C
\]
given by $\rho \mapsto \rho \oplus \det \rho^{-1}$, which descends to $X(\pi, \GL_2)_\C \rightarrow X(\pi, \SL_3)_\C$ naturally. Now consider the fiber of the boundary morphism in \ref{appendix1}, we have the relative representation and character variety $\Rep(\pi, \GL_2, \mathbf{k})_\C$. Here $\mathbf{k} \in (\GL_2 \sslash \GL_2)_\C^m \simeq (\A_\C^1 \times (\mathbb{G}_m)_\C)^m$ given by $(g_i)_{i=1}^m \mapsto (\tr g_i, \det g_i)_{i=1}^m$. Since $\prod_{i=1}^m \det g_i = 1$, the boundary datum lies in $(\SL_2 \sslash \SL_2)^m_\C \times (\mathbb{G}_m)_\C^{m-1}$. By the morphism $c : (\GL_2 \sslash \GL_2)_\C^m \simeq (\A_\C^1 \times \mathbb{G}_m)_\C^m \rightarrow \A_\C^{2m} \simeq (\SL_3 \sslash \SL_3)_\C^m$ defined by
\[
(\tr g_i, \det g_i)_{i=1}^m \mapsto (\tr g_i + (\det g_i)^{-1}, (\det g_i)^{-1} \tr g_i + \det g_i)
\]
we have a morphism $\Rep(\pi, \GL_2, \mathbf{k})_\C \rightarrow \Rep(\pi, \SL_3, c(\mathbf{k}))_\C$.

\begin{lemma}\label{finiteboundarychoice}
    The morphism $c$ is finite.
\end{lemma}

\begin{proof}
    We may assume $m=1$. Let $x = \tr g_1$, $y = \det g_1$, $u = x + y^{-1}$ and $v = y^{-1}x + y$. Then
    \[
    y^3 - vy^2 + uy - 1 = 0
    \]
    so $y$ is integral over $\C[u, v]$ and so is $y^{-1}$. Then $x$ is also integral over $\C[u, v]$. Thus $\C[u, v] \rightarrow \C[x, y^{\pm 1}]$ is finite.
\end{proof}

\begin{lemma}\label{GL2reldimension}
    $\dim X(\pi, \GL_2, \mathbf{k})_\C = 4 (\rk (\pi)-1) - 2(m-1) = 8g+2m-6$ regardless of the choice of $\mathbf{k}$.
\end{lemma}

\begin{proof}
    By Whang, we have $\dim X(\pi, \SL_2, \mathbf{k})_\C = 6g + 2m - 6$ regardless of the choice of $\mathbf{k}$. We have an isogeny
    \[
    \SL(2, \C) \times (\mathbb{G}_m)_\C \rightarrow \GL(2, \C)
    \]
    whose kernel is $\mu_2 = \pm(I, 1)$. Since $\mu_2 \leq Z(\SL(2, \C) \times (\mathbb{G}_m)_\C)$, its action on $X(\pi, \SL_2)_\C \times X(\pi, \mathbb{G}_m)_\C$ given by pointwise multiplication is free since the action on $X(\pi, \mathbb{G}_m)_\C$ is a translation. Thus $X(\pi, \SL_2)_\C \times X(\pi, \mathbb{G}_m)_\C$ is a $X(\pi, \Z/2\Z) = \Hom(\pi, \Z / 2\Z) = (\Z / 2\Z)^{2g + m-1}$-torsor over $X(\pi, \GL_2)_\C$. To see the flatness of the boundary morphism, observe the morphism
    \[
    X(\pi, \mathbb{G}_m)_\C = \Hom(\pi, \mathbb{G}_m)_\C \rightarrow (\mathbb{G}_m)_\C^{m-1}
    \]
    is given by the projection $(A_1, B_1, \ldots, A_g, B_g, C_1, \ldots, C_m) \mapsto (C_1, \ldots, C_{m-1})$ is flat whose fiber is always isomorphic to $(\mathbb{G})_m^{2g}$. Thus we have
    \[
    \dim X(\pi, \GL_2, \mathbf{k})_\C = \dim X(\pi, \SL_2 \times \mathbb{G}_m, \mathbf{k})_\C = 6g + 2m - 6 + 2g = 8g + 2m - 6.
    \]
\end{proof}

\section{Python code}\label{pythoncode}

\subsection{Computation of the Hilbert basis}
\begin{lstlisting}
from PyNormaliz import Cone
import itertools, csv, json

# Coordinate order:
# x_ij = edge coordinate on edge {i,j} near puncture i
# c_k = center coordinate of the face opposite puncture k
VAR_ORDER = []
for i in range(4):
    for j in range(i + 1, 4):
        VAR_ORDER += [f"x{i}{j}", f"x{j}{i}"]
VAR_ORDER += [f"c{i}" for i in range(4)]
VI = {v: i for i, v in enumerate(VAR_ORDER)}

# Tetrahedral faces of a 4-punctured sphere
FACES = {
    0: (1, 2, 3),
    1: (0, 3, 2),
    2: (0, 1, 3),
    3: (0, 2, 1),
}


def local_forms(face, cidx):
    """Return the 9 numerators 3*r_ij for one oriented face."""
    t, l, r = face
    a11 = VI[f"x{r}{l}"]
    a12 = VI[f"x{l}{r}"]
    a21 = VI[f"x{l}{t}"]
    a22 = VI[f"x{t}{l}"]
    a31 = VI[f"x{t}{r}"]
    a32 = VI[f"x{r}{t}"]
    a = VI[f"c{cidx}"]

    exprs = [
        [(a22, 1), (a31, 1), (a, -1)],
        [(a, 1), (a32, 1), (a11, -1), (a31, -1)],
        [(a21, 1), (a, 1), (a12, -1), (a22, -1)],
        [(a32, 1), (a11, 1), (a, -1)],
        [(a, 1), (a12, 1), (a21, -1), (a11, -1)],
        [(a31, 1), (a, 1), (a22, -1), (a32, -1)],
        [(a12, 1), (a21, 1), (a, -1)],
        [(a, 1), (a22, 1), (a31, -1), (a21, -1)],
        [(a11, 1), (a, 1), (a32, -1), (a12, -1)],
    ]

    rows = []
    for expr in exprs:
        row = [0] * len(VAR_ORDER)
        for k, v in expr:
            row[k] += v
        rows.append(row)
    return rows


def build_cone():
    inequalities = []
    congruences = []
    for cidx, face in FACES.items():
        forms = local_forms(face, cidx)
        inequalities.extend(forms)       # 3*r_ij >= 0
        congruences.extend([row + [3] for row in forms])  # 3*r_ij == 0 mod 3
    return Cone(inequalities=inequalities, congruences=congruences)


def a4_permutations():
    out = []
    for p in itertools.permutations(range(4)):
        inv = sum(1 for i in range(4) for j in range(i + 1, 4) if p[i] > p[j])
        if inv % 2 == 0:
            out.append(p)
    return out


def a4_action_maps():
    var_tuples = []
    for i in range(4):
        for j in range(i + 1, 4):
            var_tuples += [(i, j), (j, i)]
    var_tuples += [('c', 0), ('c', 1), ('c', 2), ('c', 3)]
    index = {v: i for i, v in enumerate(var_tuples)}

    maps = []
    for p in a4_permutations():
        m = []
        for v in var_tuples:
            if v[0] == 'c':
                w = ('c', p[v[1]])
            else:
                w = (p[v[0]], p[v[1]])
            m.append(index[w])
        maps.append(m)
    return maps


def act(vec, mapping):
    out = [0] * len(vec)
    for old_i, new_i in enumerate(mapping):
        out[new_i] = vec[old_i]
    return tuple(out)


def main():
    cone = build_cone()
    hb = sorted((tuple(v) for v in cone.HilbertBasis()), key=lambda v: (sum(v), v))
    hb_set = set(hb)

    maps = a4_action_maps()
    unseen = set(hb)
    orbit_rows = []
    vec_to_orbit = {}
    orbit_id = 0
    while unseen:
        orbit_id += 1
        rep = min(unseen, key=lambda v: (sum(v), v))
        orb = set(act(rep, m) for m in maps) & hb_set
        for v in orb:
            vec_to_orbit[v] = orbit_id
        orbit_rows.append((orbit_id, len(orb), sum(rep), rep))
        unseen -= orb
    orbit_rows.sort(key=lambda t: (t[2], t[3]))
    old_to_new = {old: i + 1 for i, (old, *_rest) in enumerate(orbit_rows)}
    orbit_rows = [(old_to_new[old], size, deg, rep) for old, size, deg, rep in orbit_rows]
    vec_to_orbit = {v: old_to_new[oid] for v, oid in vec_to_orbit.items()}

    print("Variable order:", VAR_ORDER)
    print("Hilbert basis size:", len(hb))
    print("Extreme ray count:", len(cone.ExtremeRays()))
    print("A4-orbits:")
    for row in orbit_rows:
        print(row)

    with open("tetrahedron_KTGS_Hilbert_basis.csv", "w", newline="") as f:
        writer = csv.writer(f)
        writer.writerow(["id", "orbit", "sum"] + VAR_ORDER)
        for i, v in enumerate(hb, 1):
            writer.writerow([i, vec_to_orbit[v], sum(v), *v])

    with open("tetrahedron_KTGS_Hilbert_basis.json", "w") as f:
        json.dump(
            {
                "variable_order": VAR_ORDER,
                "faces": FACES,
                "basis": [
                    {"id": i, "orbit": vec_to_orbit[v], "sum": sum(v), "vector": list(v)}
                    for i, v in enumerate(hb, 1)
                ],
                "orbits": [
                    {"orbit": oid, "size": size, "sum": deg, "representative": list(rep)}
                    for oid, size, deg, rep in orbit_rows
                ],
            },
            f,
            indent=2,
        )


if __name__ == "__main__":
    main()
\end{lstlisting}

\subsection{Computation of the nerve complex}
\begin{lstlisting}
    from itertools import combinations
import sympy as sp
from sympy.matrices.normalforms import smith_normal_form

M = [
    {1,2,3,4,5,7},
    {1,2,3,4,5,8},
    {1,2,3,4,6,7},
    {1,2,3,4,6,8},
    {1,2,3,5,6,7},
    {1,2,3,5,6,8},
    {1,2,4,5,6,7,8},
    {1,3,4,5,6,7,8},
    {2,3,4,5,6,7,8},
]

simplices = {}

for dim in range(8):
    simplices[dim] = []
    for I in combinations(range(9), dim + 1):
        inter = set.intersection(*(M[i] for i in I))
        if inter:
            simplices[dim].append(I)

f = [len(simplices[d]) for d in range(8)]
print("f-vector:", f)

smith_data = {}
ranks = {}

for k in range(1, 8):
    rows = simplices[k - 1]
    cols = simplices[k]
    row_index = {s: i for i, s in enumerate(rows)}

    B = sp.zeros(len(rows), len(cols))

    for j, sigma in enumerate(cols):
        for t in range(len(sigma)):
            face = sigma[:t] + sigma[t+1:]
            i = row_index[face]
            B[i, j] = (-1) ** t

    S = smith_normal_form(B, domain=sp.ZZ)

    diag = []
    for i in range(min(S.shape)):
        if S[i, i] != 0:
            diag.append(abs(int(S[i, i])))

    smith_data[k] = diag
    ranks[k] = len(diag)

print("ranks from Smith data:", ranks)

betti = []
for k in range(8):
    rk_k = ranks.get(k, 0)
    rk_next = ranks.get(k + 1, 0)
    betti.append(f[k] - rk_k - rk_next)

print("Betti numbers:", betti)

for k in range(1, 8):
    print(f"k={k}, Smith invariants:", sorted(set(smith_data[k])))
\end{lstlisting}
\end{appendix}

\bibliographystyle{plain}
\bibliography{references}

\end{document}